\documentclass[
paper=letter,%
numbers=noendperiod,%
captions=nooneline,%
DIV=10%
]{scrartcl}

\usepackage[T1]{fontenc}%
\usepackage{lmodern}%
\usepackage[american]{babel}%
\usepackage{microtype}%

\usepackage[
hyperref,%
table%
]{xcolor}%

\usepackage{scrlayer-scrpage}%

\usepackage{calc}%
\usepackage{rotating}%
\usepackage{amsmath}%
\usepackage{mathtools}%
\usepackage{amssymb}%
\usepackage{amsthm}%
\usepackage{thmtools}%
\usepackage{etoolbox}%
\usepackage{xparse}%
\usepackage{bm}%
\usepackage{bbm}%
\usepackage{enumitem}%
\usepackage[calc]{datetime2}%
\usepackage{graphicx}%
\usepackage{grffile}%
\usepackage{tikz}%
\usepackage{wrapfig}%
\usepackage{tabularx}%
\usepackage{siunitx}%
\usepackage{booktabs}%
\usepackage{multirow}%
\usepackage{vruler}%
\usepackage{fancyvrb}%
\usepackage{textcomp}%
\usepackage{listings}%
\usepackage{csquotes}%
\usepackage[
style=authoryear,%
dashed=false,%
hyperref=true,%
useprefix=true,%
maxnames=2,%
uniquelist=minyear,%
maxbibnames=6,%
uniquename=false,%
seconds=true,%
date=iso,%
urldate=iso%
]{biblatex}%
\usepackage[
hypertexnames=false,%
setpagesize=false,%
pdfborder={0 0 0},%
pdfstartview=Fit,%
bookmarksopen=true,%
bookmarksnumbered=true%
]{hyperref}%
\usepackage[hyphens]{xurl}
\hypersetup{breaklinks=true}
\setkomafont{pageheadfoot}{\normalfont\normalcolor\sffamily}%
\setkomafont{pagenumber}{\normalfont\normalcolor\sffamily}%
\automark{section}%
\setkomafont{captionlabel}{\normalfont\normalcolor\sffamily\bfseries}%

\newcommand*{\mysquare}{\rule[0.18em]{0.36em}{0.36em}}
\newcommand*{\mytriangle}{\raisebox{0.12em}{\resizebox{0.48em}{0.48em}{$\blacktriangleright$}}}
\newcommand*{\mybar}{\rule[0.32em]{0.62em}{0.08em}}
\newcommand*{\mydot}{\raisebox{0.14em}{\resizebox{0.44em}{!}{$\bullet$}}}
\setlist{%
  align=left,%
  labelindent=0mm, %
  leftmargin=!,%
  itemindent=0mm, %
  listparindent=\parindent,%
  parsep=0mm,%
  topsep=1mm,%
  itemsep=1mm%
}
\setlist[itemize,1]{label={\mysquare}, labelwidth=\widthof{\mysquare\ }}%
\setlist[itemize,2]{label={\mytriangle}, labelwidth=\widthof{\mytriangle\ }}%
\setlist[itemize,3]{label={\mybar}, labelwidth=\widthof{\mybar\ }}%
\setlist[itemize,4]{label={\mydot}, labelwidth=\widthof{\mydot\ }}%
\setlist[enumerate,1]{label=\arabic*), labelwidth=\widthof{9)}}%
\setlist[enumerate,2]{label=\arabic{enumi}.\arabic*), labelwidth=\widthof{9.9)}}%
\setlist[enumerate,3]{label=\arabic{enumi}.\arabic{enumii}.\arabic*), labelwidth=\widthof{9.9.9)}}%
\setlist[enumerate,4]{label=\arabic{enumi}.\arabic{enumii}.\arabic{enumiii}.\arabic*), labelwidth=\widthof{9.9.9.9)}}%

\allowdisplaybreaks%
\newcommand*{\abstractnoindent}{}%
\let\abstractnoindent\abstract
\renewcommand*{\abstract}{\let\quotation\quote\let\endquotation\endquote
  \abstractnoindent}
\deffootnote[1em]{1em}{1em}{\textsuperscript{\thefootnotemark}}%
\renewcommand*{\big}[1]{{\vcenter{\hbox{\scalebox{1.30}{\ensuremath#1}}}}}%

\definecolor{blue}{RGB}{58, 95, 205}%
\definecolor{red}{RGB}{205, 41, 144}%
\definecolor{orange}{RGB}{238, 118, 0}%
\definecolor{chocolate}{RGB}{205, 102, 29}%

\lstdefinestyle{input}{
  backgroundcolor=\color{black!12},%
  commentstyle=\itshape\color{black!50},%
  keywordstyle=\bfseries\color{black},%
  stringstyle=\color{black}%
}

\lstdefinestyle{output}{
  backgroundcolor=\color{black!6}%
}

\lstdefinestyle{codestyle}{
  language={},%
  keywords={},%
  otherkeywords={}%
}

\lstnewenvironment{codeinput}[1][]{%
  \lstset{style=input, style=codestyle}
  #1%
}{\vspace{-0.25\baselineskip}}%

\lstnewenvironment{codeoutput}[1][]{%
  \lstset{style=output, style=codestyle}
  #1%
}{\vspace{-0.25\baselineskip}}%

\expandafter\let\csname Sinput\endcsname\relax
\expandafter\let\csname endSinput\endcsname\relax
\expandafter\let\csname Soutput\endcsname\relax
\expandafter\let\csname endSoutput\endcsname\relax

\lstdefinestyle{Rstyle}{
  language=R,%
  keywords={},%
  otherkeywords={}%
}

\lstnewenvironment{Sinput}[1][]{%
  \lstset{style=input, style=Rstyle}
  #1%
}{\vspace{-0.25\baselineskip}}%

\lstnewenvironment{Soutput}[1][]{%
  \lstset{style=output, style=Rstyle}
  #1%
}{\vspace{-0.25\baselineskip}}%

\lstdefinestyle{Cstyle}{
  language=C,%
  keywords={},%
  otherkeywords={}%
}

\lstnewenvironment{Cinput}[1][]{%
  \lstset{style=input, style=Cstyle}
  #1%
}{\vspace{-0.25\baselineskip}}%

\lstnewenvironment{Coutput}[1][]{%
  \lstset{style=output, style=Cstyle}
  #1%
}{\vspace{-0.25\baselineskip}}%

\lstdefinestyle{Bashstyle}{
  language=bash,%
  keywords={},%
  otherkeywords={}%
}

\lstnewenvironment{Bashinput}[1][]{%
  \lstset{style=input, style=Bashstyle}
  #1%
}{\vspace{-0.25\baselineskip}}%

\lstnewenvironment{Bashoutput}[1][]{%
  \lstset{style=output, style=Bashstyle}
  #1%
}{\vspace{-0.25\baselineskip}}%

\lstdefinestyle{LaTeXstyle}{
  language=[LaTeX]TeX,%
  texcs={},%
  keywords={},%
  otherkeywords={}%
}

\lstnewenvironment{LaTeXinput}[1][]{%
  \lstset{style=input, style=LaTeXstyle}
  #1%
}{\vspace{-0.25\baselineskip}}%

\lstnewenvironment{LaTeXoutput}[1][]{%
  \lstset{style=output, style=LaTeXstyle}
  #1%
}{\vspace{-0.25\baselineskip}}%

\DeclareNameAlias{sortname}{family-given}%
\DefineBibliographyExtras{american}{\DeclareQuotePunctuation{}}%
\renewbibmacro*{volume+number+eid}{%
  \setunit*{\addcomma\space}%
  \printfield{volume}%
  \printfield{number}}
\DeclareFieldFormat*{number}{(#1)}
\DeclareFieldFormat*{title}{#1}%
\DeclareFieldFormat{doi}{%
  \ifhyperref
    {\href{http://dx.doi.org/#1}{\nolinkurl{doi:#1}}}%
    {\nolinkurl{doi:#1}}}%
\renewbibmacro*{in:}{}%
\DeclareFieldFormat{isbn}{ISBN #1}%
\DeclareFieldFormat{pages}{#1}%
\DeclareFieldFormat{url}{\url{#1}}%
\DeclareFieldFormat{urldate}{\mkbibparens{#1}}%
\renewcommand*{\cite}[2][]{\textcite[#1]{#2}}%

\newif\ifstarttheorem
\declaretheoremstyle[%
  spaceabove=0.5em,
  spacebelow=0.5em,
  headfont=\sffamily\bfseries\global\starttheoremtrue,
  notefont=\sffamily\bfseries,
  notebraces={(}{)},
  headpunct={},
  bodyfont=\normalfont,
  postheadspace=\newline%
]{myMainStyle}
\declaretheorem[style=myMainStyle, numberwithin=section]{definition}%
\declaretheorem[style=myMainStyle, sibling=definition]{proposition}
\declaretheorem[style=myMainStyle, sibling=definition]{lemma}
\declaretheorem[style=myMainStyle, sibling=definition]{theorem}
\declaretheorem[style=myMainStyle, sibling=definition]{corollary}
\declaretheorem[style=myMainStyle, sibling=definition]{remark}
\declaretheorem[style=myMainStyle, sibling=definition]{example}

\makeatletter
\preto\itemize{%
  \if@inlabel
    \ifstarttheorem
      \mbox{}\par\nobreak\vskip\glueexpr-\parskip-\baselineskip+0.25em\relax\hrule\@height\z@
    \fi%
  \fi%
  \global\starttheoremfalse%
 \def\tempa{proof}%
 \ifx\tempa\mycurrenvir
    \ifstarttheorem
      \mbox{}\par\nobreak\vskip\glueexpr-\parskip-\baselineskip+0.25em\relax\hrule\@height\z@
    \fi%
 \fi%
 \global\starttheoremfalse%
}
\preto\enditemize{\global\starttheoremfalse}
\makeatother

\makeatletter
\preto\enumerate{%
  \if@inlabel
    \ifstarttheorem
      \mbox{}\par\nobreak\vskip\glueexpr-\parskip-\baselineskip+0.25em\relax\hrule\@height\z@
    \fi%
  \fi%
  \global\starttheoremfalse%
 \def\tempa{proof}%
 \ifx\tempa\mycurrenvir
    \ifstarttheorem
      \mbox{}\par\nobreak\vskip\glueexpr-\parskip-\baselineskip+0.25em\relax\hrule\@height\z@
    \fi%
 \fi%
 \global\starttheoremfalse%
}
\preto\endenumerate{\global\starttheoremfalse}
\makeatother

\makeatletter
\NewDocumentCommand{\tmb}{O{0.1mm} O{0.1mm} O{0.88} m m m}{%
  \mathrel{%
    \vbox{\offinterlineskip\m@th
      \ialign{%
        \hfil##\hfil\cr
        $\scriptscriptstyle\text{\scalebox{#3}{#4}}\mathstrut$\cr%
        \noalign{\vspace{#1}}%
        \vtop{%
          \ialign{%
            \hfil##\hfil\cr
            $#5$\cr\noalign{\vspace{#2}}%
            $\scriptscriptstyle\text{\scalebox{#3}{#6}}\mathstrut$\cr%
          }%
        }\cr
      }%
    }%
  }%
}
\makeatother

\makeatletter
\NewDocumentCommand{\tmbc}{O{0.1mm} O{0.1mm} O{0.88} m m m}{%
  \mathrel{%
    \vbox{\offinterlineskip\m@th
      \ialign{%
        \hfil##\hfil\cr
        $\scriptscriptstyle\mathclap{\text{\scalebox{#3}{#4}}}\mathstrut$\cr%
        \noalign{\vspace{#1}}%
        \vtop{%
          \ialign{%
            \hfil##\hfil\cr
            $#5$\cr\noalign{\vspace{#2}}%
            $\scriptscriptstyle\mathclap{\text{\scalebox{#3}{#6}}}\mathstrut$\cr%
          }%
        }\cr
      }%
    }%
  }%
}
\makeatother

\usetikzlibrary{shapes.geometric}

\def\IR{\mathbb{R}}
\def\IN{\mathbb{N}}
\def\Prob{\mathbb{P}}
\def\E{\mathbb{E}}

\def\bfY{\bm{Y}}

\def\d{\,\mathrm{d}}

\newcommand{\abs}[1]{|#1|}

\newcommand{\ceil}[1]{\left\lceil#1\right\rceil}

\def\gen{\psi}
\def\invgen{\gen^{-1}}

\def\Copula{C}

\def\indist{\overset{\text{\tiny d}}{\longrightarrow}}

\newcommand{\eqindist}{\stackrel{\mbox{\tiny d}}{=}}

\def\cn{c}
\def\dn{d}
\def\an{\cn^*}
\def\bn{\dn^*}

\def\pF{F}
\def\pMargin{\pF}

\def\pG{G}
\def\pD{D}
\def\dD{d}
\def\qD{\pD^{-1}}

\def\EVDist{H}

\def\EVI{\xi}

\def\STDF{\ell}

\def\JointDist{\pMargin}

\newcommand*{\U}{\operatorname{U}}

\newcommand*{\I}{\mathds{1}}

\DeclareMathOperator{\RV}{RV}
\DeclareMathOperator{\MDA}{MDA}
\DeclareMathOperator{\cov}{cov}

\def\STDFone{\eta}
\def\Cdiag{\delta}
\def\Normal{\mathrm{N}}

\DeclareMathOperator{\opFN}{FN}
\def\FN{\opFN}
\def\rate{r}
\def\iidrate{\beta^*}

\DeclareMathOperator{\GEV}{GEV}

\usepackage{dsfont}

\def\opCondD{\mathcal{D}}
\def\CondDu{\opCondD(u_n)}

\newcommand\ct[1]{\text{\rmfamily\upshape #1}}

\def\e{\ct{e}}

\begin{document}
\thispagestyle{plain}
\begin{center}
  \sffamily
  {\bfseries\LARGE Limiting Behavior of Maxima under Dependence\par}
  \bigskip\smallskip
  {\Large Klaus Herrmann\footnote{D\'epartement de Math\'ematiques, Universit\'e de Sherbrooke, \href{mailto:klaus.herrmann@usherbrooke.ca}{\nolinkurl{klaus.herrmann@usherbrooke.ca}}}, Marius Hofert\footnote{Department of Statistics and Actuarial Science, The University of Hong Kong, \href{mailto:mhofert@hku.hk}{\nolinkurl{mhofert@hku.hk}}}, Johanna G.\ Ne\v{s}lehov\'a\footnote{Department of Mathematics and Statistics, McGill University, \href{mailto:johanna.neslehova@mcgill.ca}{\nolinkurl{johanna.neslehova@mcgill.ca}}}%
    \par
    \bigskip
    \today\par}
\end{center}
\par\smallskip
\begin{abstract}
  Weak convergence of maxima of dependent sequences of identically distributed
  continuous random variables is studied under normalizing sequences arising as
  subsequences of the normalizing sequences from an associated
  iid sequence. This general framework allows one to derive several
  generalizations of the well-known Fisher--Tippett--Gnedenko theorem under
  conditions on the univariate marginal distribution and the dependence
  structure of the sequence. The limiting distributions are shown to be
  compositions of a generalized extreme value distribution and a
  distortion function which reflects the limiting behavior of the diagonal of
  the underlying copula. Uniform convergence rates for the weak convergence
  to the limiting distribution are also derived. Examples covering
  well-known dependence structures are provided. Several existing results, e.g.\
  for exchangeable sequences or stationary time series, are embedded in the
  proposed framework.
\end{abstract}
\minisec{Keywords}
copula,
dependent sequences,
distortion function,
extremes,
maxima,
time series,
uniform convergence rate
\minisec{MSC2010}
60G70, 62H05%

\section{Introduction}
Consider a sequence $(X_i, i\in\IN)$ (in short: $(X_i)$) of random variables
with common distribution function $\pMargin$, i.e.\
$X_1,X_2,\ldots\sim\pMargin$. The purpose of this article is to study the
limiting behavior of the maximum of the first $n$ variables of this sequence
under suitable normalization. That is, we seek to identify conditions under
which there exists a nondegenerate distribution function $\pG$ and normalizing
sequences $(\cn_n)$, $\cn_n>0$, and $(\dn_n)$ of constants such that
$M_n=\max(X_1,\dots,X_n)$ satisfies the convergence in distribution
\begin{align}\label{eq:central}
  \frac{M_n - \dn_n}{\cn_n} \indist \pG.
\end{align}
If \eqref{eq:central} holds, we say that $\pMargin$ is in the maximum domain of
attraction of $\pG$, denoted by $\pMargin\in\MDA(\pG)$. Our further interest
lies in identifying the limit $\pG$. Answering these questions is essential for applications that require
the extrapolation into the tail of the unknown underlying distribution function
$\pMargin$; this is a cornerstone of extreme value analysis as described in
standard textbooks such as
\cite{BeirlantGoegebeurTeugelsSegers2004,deHaanFerreira2006,EmbrechtsKluppelbergMikosch1997,Resnick1987}.

  When the random variables $X_i$, $i \in \IN$, are iid the problem has been
solved in \cite{FisherTippett1928} and \cite{Gnedenko1943}.  The
  so-called Fisher--Tippett--Gnedenko theorem, given, e.g.\ in \cite[Proposition~0.3]{Resnick1987}, states that if
  $\pMargin\in\MDA(\pG)$, then $\pG$ necessarily belongs to the class of
  generalized extreme value distributions, denoted by $\pG\in
  \GEV$. Distribution functions in the $\GEV$ class have
  parameters $\xi,\mu\in\IR$ and $\sigma>0$, and they are given by
  \begin{align*}
    \EVDist_{\EVI,\mu,\sigma}(x) = \begin{cases}
                                     \exp(-(1+\xi (x-\mu)/\sigma)^{-1/\xi}),&\text{if}\,\ \xi\neq 0,\\
                                     \exp(-\e^{-(x-\mu)/\sigma}),&\text{if}\,\ \xi=0,
                                   \end{cases}
  \end{align*}
  for all $x \in \mathbb{R}$ with $1+\xi (x-\mu)/\sigma>0$ and by the limiting
  values $\EVDist_{\EVI,\mu,\sigma}(x)=0$ or $\EVDist_{\EVI,\mu,\sigma}(x)=1$
  otherwise. Results about the construction of the normalizing sequences in the
  iid case are also available, see, e.g.\
  \cite{Resnick1987,EmbrechtsKluppelbergMikosch1997,deHaanFerreira2006}.

  Our main contribution is to identify the limiting
  behavior of appropriately centered and scaled maxima of $X_1,X_2,\ldots\sim F$
  under dependence. We derive conditions on $\pMargin$ and on the dependence
  structure of $(X_i)$ under which the limiting distribution $\pG$ in \eqref{eq:central} is
  nondegenerate and are able to identify its form, thereby generalizing the
  Fisher--Tippett--Gnedenko theorem to the case of dependent sequences.  In
  particular, we will show that $\pG$ is a
  composition of a generalized extreme value distribution and a univariate
  function driven by the
  dependence structure of the sequence $(X_i)$.

  Maxima of dependent sequences are of particular interest in actuarial science,
  where the principal idea of insurable exposure is that a large number of
  similar risks are pooled into homogeneous portfolios; see, e.g.\
  \cite[Chapter~2]{MehrCammack1980}. The behavior of the largest loss in such a
  portfolio then plays a key role for risk assessment. In such situations and
  with a large portfolio approximation in mind, one often assumes that risks
  are identically distributed. Also, such risks are typically dependent due
  to common factors affecting all risks.

  The classical univariate Fisher--Tippett--Gnedenko theorem was extended early
  on to incorporate specific types of dependence between the random variables
  $X_i$, $i \in \IN$. Berman \cite{Berman1962b} derived general limiting results
  for maxima of exchangeable sequences.  The special case where the dependence
  between the random variables of the exchangeable sequence is given by an
  Archimedean copula was considered in \cite{Ballerini1994b,Wuthrich2004};
  extremal properties under Archimedean and related dependence structures were
  crucial in modeling collateralized debt obligations, e.g.\ in
  \cite{SchoenbucherSchubert2001} and \cite{hofertscherer2011}.  Almost sure
  limit theorems for maxima under Archimedean dependence were recently studied
  in \cite{DudzinskiFurmanczyk2017}, while \cite{HuangNorthZewotir2017} proposed
  adaptations of the block maxima and peaks-over-threshold methods for
  exchangeable sequences.

  Furthermore, an extensive body of literature exists when $(X_i)$ is a
  stationary time series. The early work of \cite{Watson1954} under
  $m$-dependence, and \cite{Loynes1965} under $\alpha$-mixing was substantially
  extended and refined in \cite{Leadbetter1974} and \cite{Leadbetter1983},
  culminating in \cite{LeadbetterLindgrenRootzen1983}. Numerous statistical
  inference procedures for tail extrapolation in the context of stationary time
  series have also been developed, see e.g.\
  \cite{BeirlantGoegebeurTeugelsSegers2004, FawcettWalsh2012}; a recent review
  with extensions is provided by
  \cite{BuriticaMeyerMikoschWintenberger2021}. Alternatively,
  \cite{FerreiraFerreira2018} consider a local dependence condition for
  stationary time series, while \cite{deHaanMercadierZhou2016} and
  \cite{ChavezGuillou2018} work under a $\beta$-mixing
  framework. %
  Finally, \cite{RussellHuang2021} use a bivariate Gumbel copula to model the
  dependence between consecutive block maxima in a first order Markov
  process.%

Contrary to the existing literature, we study the limiting behavior of $M_n$ in general, i.e.\ without assuming exchangeability or stationarity. This context is akin to the order statistics literature, which studies the distribution of maxima for various multivariate distributions in finite samples, see e.g.\ \cite{BalakrishnanBendreMalik1992,Rychlik1994,ArellanoValleGenton2008}.
We only require throughout the classical assumption that the common distribution
function $\pMargin$ be continuous. The dependence between the random variables
$(X_i, i \in \IN)$ can then be uniquely described using copulas following
Sklar's theorem \cite{Sklar1959}. Specifically, for each $n\in\IN$, there exists
a unique copula $C_n$, i.e.\ a distribution function on $[0,1]^n$ with standard
uniform univariate margins, such that
$\Prob(X_1\le x_1,\dots, X_n\le x_n)=C_n(\pMargin(x_1),\dots,\pMargin(x_n))$,
$(x_1,\dots,x_n)\in\IR^n$. With this representation, we obtain
$\Prob(M_n\le x) = \Cdiag_n(\pMargin(x))$ where $\Cdiag_n(u)=C_n(u,\dots,u)$,
$u\in[0,1]$, is the \emph{diagonal} of $C_n$. As we show in
Section~\ref{section_extremes_and_copula_diagonals}, the limiting behavior of
$M_n$ under dependence is determined by the tail behavior of $\pMargin$ and its
interplay with the properties of the copula diagonal $\delta_n$.  Investigating
this relationship allows us to derive a rigorous theoretical framework for the
study of the limiting behavior of maxima under dependence.

Our first main result (Theorem~\ref{theorem_continuous_limit}) provides
conditions under which maxima of identically distributed dependent random
variables $(X_i)$ can be normalized using the same sequences $(\cn_n)$,
$\cn_n>0$, and $(\dn_n)$ of normalizing constants as in the iid case. Our second
main result (Theorem~\ref{theorem-FTG2}) is then a generalization of the
Fisher--Tippett--Gnedenko theorem under dependence. In both cases the limiting
distribution of appropriately stabilized maxima is of the form
$\pD\circ \EVDist$ for $\EVDist\in\GEV$ and $\pD$ being a \emph{distortion
  function}, i.e.\ a distribution function satisfying $\pD(0) = 0$ and
$\pD(1) = 1$. Uniform convergence rates for the weak convergence of normalized
maxima to the limit $\pD\circ \EVDist$ are derived in
Section~\ref{sec:unif:conv:rate}(Theorem~\ref{theorem_convergence_rate}).

The results from Section~\ref{section_extremes_and_copula_diagonals}
are illustrated through various examples. In Section~\ref{sec:3}, we first
consider so-called power diagonals which still lead to limiting distributions of
maxima that are generalized extreme value. We also establish connections between
our findings and well-known results about extremes of stationary time series
with restricted short-range dependence at extreme levels. Complementing these
results, we also provide examples that illustrate how our methodology can be
used for time series with long-range dependence. Section~\ref{sec:examples} then
presents examples where the limiting distribution of maxima is no longer
generalized extreme value. Notably, we explore the case where $\Copula_n$ is
Archimedean or Archimax and provide a generalization of multiplicative frailty
models linked to the Fr\'echet domain of attraction. Conclusions are given in
Section~\ref{section_conclusion}. Proofs of all results are provided in the
Supplementary Material \cite{HerrmannHofertNeslehova2024}.

\section{Weak convergence of maxima of dependent sequences}\label{section_extremes_and_copula_diagonals}
In this section, we derive a generalization of the Fisher--Tippett--Gnedenko
theorem when $X_1,X_2, \ldots \sim \pMargin$ are identically
  distributed but dependent. As stated before, we
  assume that $\pMargin$ is continuous and denote by $C_n$ and $\Cdiag_n$ the unique copula of
  $(X_1,\ldots, X_n)$ and its diagonal, respectively.

For a fixed $n$, general properties of copula diagonals have
  been studied, e.g., by \cite{Jaworski2009}; the connection between the
  maximum $M_n$ of $X_1,\ldots, X_n$ and $\delta_n$ is also made in
  \cite{HofertMachlerMcNeil2013} who use it to construct a maximum
  likelihood estimator for the parameters of Archimedean
  copula models. In order to exploit the
latter connection in an asymptotic context when $n \to \infty$, we first
define the so-called diagonal power distortion of $\Cdiag_n$.%
\begin{definition}%
\label{def:Dnr} Let
  $X_1,X_2,\ldots\sim\pMargin$ and $\Copula_n$ be the unique copula of
  $(X_1,\ldots,X_n)$ with diagonal $\Cdiag_n$. Let the rate
  $\rate \colon \IN \to (0,\infty)$ be an arbitrary strictly positive
  function; for simplicity, we write $\rate_n$
  for $\rate(n)$. The \emph{diagonal power distortion} with respect to the
  rate $\rate$ is $\pD_n^{\rate} \colon [0,1]\to [0,1]$ given, for all $u \in [0,1]$, by
  $\pD^{\rate}_n(u) = \Cdiag_n(u^{1/\rate_n})$.
\end{definition}
Because $C_n$ is componentwise non-decreasing and uniformly continuous \cite{Nelsen2006}, $\pD^{\rate}_n$ is a continuous distribution
function on $[0,1]$ for any positive rate function $\rate$. %
We begin our investigation towards a generalization of the Fisher--Tippett--Gnedenko theorem for dependent sequences  by describing the
precise impact of $\delta_n$ on the limiting distribution of maxima.
\begin{theorem}%
  \label{theorem_continuous_limit} %
  Consider $X_1,X_2,\ldots\sim\pMargin$ and
  let %
  $M_n^*=\max(X_1^*,\ldots, X_n^*)$ denote the maximum corresponding to an iid
  sequence $X_1^*,X_2^*,\ldots \sim\pMargin$.  Assume that
  $\pMargin\in\MDA(\EVDist)$ for a $\GEV$ distribution $H$ with normalizing
  sequences $(\bn_n)$ and $(\an_n)$, $\an_n > 0$, i.e.\ for any $x \in \IR$,
  $\lim_{n\to\infty} \Prob\left(\frac{M_n^*-\bn_n}{\an_n} \leq x\right) = \EVDist(x)$.
  \begin{enumerate}[label=(\roman*), labelwidth=\widthof{(iii)}]
  \item\label{mainThm:1}
    If there exists a function $\rate \colon \IN \to (0,\infty)$ such that $\lim_{n\to\infty} \rate_n = \infty$ and $\pD^{\rate}_n$ converges pointwise to a continuous function $\pD$ on $[0,1]$, then, for all $x \in \mathbb{R}$,
    \begin{align}
      \lim_{n\to\infty} \Prob\left(\frac{M_n - \bn_{\ceil{\rate_n}}}{\an_{\ceil{\rate_n}}} \leq x \right) = \pD \circ \EVDist (x)\label{eq_generalized_FTP}
    \end{align}
    and $\pD$ is a continuous distribution function on $[0,1]$.
  \item\label{mainThm:2}
    If there exists a function $\rate \colon \IN \to (0,\infty)$ such that
    $1/r_n = O(1/n)$ and a distribution function $\pD$ on $[0,1]$ so
    that, %
    for all continuity points $x$ of $\pD\circ\EVDist$, \eqref{eq_generalized_FTP} holds,
    then %
    $ \lim_{n\to\infty} \pD_n^{\rate}(u) = \pD(u)$
    for $u\in\{0,1\}$ and all continuity points $u\in(0,1)$ of $\pD$.
  \end{enumerate}
\end{theorem}

\begin{remark}%
  Although some of the results in Theorem~\ref{theorem_continuous_limit} can
  possibly be extended to the case when $\pMargin$ is not continuous, we do not
  consider such a generalization here. First, $\pMargin\in\MDA(\EVDist)$ does
  not hold for several well-known discrete distributions even in the iid case
  \cite[Section~3.1]{EmbrechtsKluppelbergMikosch1997} and other approaches,
  such as maxima of triangular arrays, are typically considered
  (\cite{AndersonColesHusler1997}). Second, the copula is no longer unique when
  $\pMargin$ is not continuous and this complicates matters in the present
  context; see e.g.\ \cite{FeidtGenestNeslehova2010}.
\end{remark}

\begin{remark}%
$\pD_n^{\rate}(u)$ can be interpreted as a probability.
  Indeed, for
  $(U_1,\dots, U_n)\sim\Copula_n$, %
  $\pD_n^{\rate}(u) =
  \Cdiag_n(u^{1/\rate_n}) %
  = \Prob(U_1 \leq u^{1/\rate_n},\dots,U_n \leq u^{1/\rate_n})=
  \Prob(\{\max(U_1,\dots,U_n)\}^{\rate_n} \leq u)$.  Its limiting behavior as
  $n\to \infty$ can thus be rephrased in terms of weak convergence of the
  maximum of dependent standard uniform random variables under power
  normalization to a continuous limit.
  Theorem~\ref{theorem_continuous_limit}~\ref{mainThm:1} thus complements and
  generalizes results of limiting distributions under power normalization in the
  iid case, pioneered in \cite{Pancheva1985}. Due to the positivity of the
  standard uniform random variables we also have a direct connection to scale
  transformations via logarithms. If $W_i$ denotes a standard reverse Weibull
  random variable, i.e.\ $\Prob(W_i\leq x) = \exp(x)$ for $x\le 0$, taking
  logarithms leads to
  $\pD_n^{\rate}(u) = \Prob(\max(W_1,\ldots,W_n) \leq \log(u)/\rate_n)$, where
  the copula of $(W_1,\ldots,W_n)$ is again $\Copula_n$ because of the
  invariance of copulas with respect to strictly increasing transformations
  \cite[Theorem~2.4.3]{Nelsen2006}. %
\end{remark}

Theorem~\ref{theorem_continuous_limit} shows that, under suitable conditions on
$r$, the convergence of the diagonal power distortion
$\pD_n^{\rate}$ to a limit $\pD$ is necessary and sufficient for
\eqref{eq:central} to hold with $\pG=\pD \circ \EVDist$, provided that the
 sequences $(c_n)$ and $(d_n)$ are suitable subsequences of $(c_n^*)$
and $(d_n^*)$. In this setup,
we necessarily have that $\EVDist \in \GEV$ by the classical Fisher--Tippett--Gnedenko theorem.
The requirement that
$\rate_n \to \infty$ is central for these statements to be true. To see this, consider a
sequence with $X_i = X_1$ for all $i \ge 2$ almost surely. In this case, for
every $n\in\IN$, $C_n$ is the comonotone copula given, for all
$u_1,\ldots, u_n \in [0,1]$, by $\min(u_1,\dots,u_n)$ and thus $\Cdiag(u) = u$ for
all $u \in [0,1]$. Then $M_n=X_1$ almost surely and \eqref{eq:central} holds
with $c_n=1$, $d_n=0$, $n \in \IN$, and $\pG=F$. This means that the limiting
distribution $\pG$ can be entirely arbitrary here.

We now consider what happens if (instead of assuming
$r_n \to \infty$ as in Theorem~\ref{theorem_continuous_limit}) $\rate_n \to \varrho$ for
$n \to \infty$ and $\varrho \in (0,\infty)$. In this
case, there is no need to stabilize $M_n$ or to assume that $\pMargin$ is in the
maximum domain of attraction of a nondegenerate distribution function; however,
the limit in \eqref{eq:central} is not necessarily generalized extreme value or
a distortion thereof anymore.

\begin{proposition}%
\label{cor:r(n)finite}
  Let $X_1,X_2,\ldots\sim\pMargin$ and $M_n = \max(X_1,\ldots,X_n)$, $n\in\IN$.
  \begin{enumerate}[label=(\roman*), labelwidth=\widthof{(iii)}]
  \item\label{cor:case:1}
  If there exists
    $\rate \colon \IN \to (0,\infty)$ such that $\rate_n \to \varrho\in(0,\infty)$ as
    $n \to \infty$ and %
    $\pD^{\rate}_n$ %
    converges pointwise to a continuous function
    $\pD$ on $[0,1]$, then, for all $x \in \IR$,
    $\Prob(M_n \leq x) \to \pD (\pMargin^{\varrho}(x))$ as $n \to \infty$.
  \item\label{cor:case:2}
  Conversely, if $\rate_n\to\varrho\in(0,\infty)$ such
    that $\abs{\rate_n-\varrho} = o(1/n)$, and
    $\Prob(M_n \leq x) \to \pD (\pMargin^{\varrho}(x))$ as $n \to \infty$ for
    all continuity points of $\pD\circ\pMargin^{\varrho}$, then
    $\lim_{n\to\infty}\pD^{\rate}_n(u) = \pD(u)$ for $u\in\{0,1\}$ and all
    continuity points $u\in(0,1)$ of $\pD$.
  \end{enumerate}
\end{proposition}
The following two examples illustrate Proposition~\ref{cor:r(n)finite} for nontrivial dependent sequences.
\begin{example}%
\label{example_gen_comonotone_sequence}
Consider again $X_i = X_1$ for
all $i \ge 2$ almost surely. Therefore, if we consider any strictly positive
function $\rate$ with the property that $\rate_n \to \varrho$ as $n \to \infty$
for some $\varrho\in (0,\infty)$, we have that
$\pD_n^{\rate}(u) \to u^{1/\varrho}$ and the limit satisfies
$\pD(F^\varrho(x)) = \pMargin(x)$ for all $x \in \IR$. This example readily
generalizes to an eventually almost surely constant sequence of the
form $X_1,\ldots, X_K, X_K, X_K, \dots$. If $\Cdiag_K$ denotes the diagonal of the
copula $C_K$ of $(X_1,\ldots, X_K)$, then $\pD_n^r$
converges pointwise to $\pD(u) = \Cdiag_K(u^{1/\varrho})$ for all $u \in [0,1]$
for any $\rate$ that satisfies $\rate_n \to \varrho$ as $n \to \infty$ for some
$\varrho \in (0,\infty)$. The limiting distribution of $M_n$ is then
$\Cdiag_K\circ F$.
\end{example}
\begin{example}%
\label{example_Cuadras_Auge}
A special case of the multivariate Cuadras--Aug\'e copula
\cite{CuadrasAuge1981}, with parameter $\theta \in(0,1)$, is given, for all
$u_1,\ldots, u_n \in [0,1]$, by
$\Copula_n(u_1,\ldots, u_n) = \prod_{i=1}^n u_{(i)}^{(1-\theta)^{i-1}}$, where
$u_{(1)} \le \ldots \le u_{(n)}$ are the evaluation points sorted in
ascending order; see \cite{MaiScherer2009}. \cite[Example~5]{Mai2018} shows that
there indeed exists an infinite sequence of identically distributed random
variables with the property that for each $n \ge 2$, the copula $\Copula_n$ of
$(X_1,\dots,X_n)$ is a Cuadras--Aug\'e copula of this form with parameter $\theta \in (0,1)$. The diagonal of $\Copula_n$ has the form
$\Cdiag_n(u) = u^{(1-(1-\theta)^n)/\theta}$ for all $u\in [0,1]$. If
$\rate_n = (1-(1-\theta)^n)/\theta$, then $\rate_n \to 1/\theta$ as
$n \to \infty$. Furthermore, for all $u \in [0,1]$, $\pD_n^\rate(u) = \pD(u)=u$
for all $n \in \IN$. By Proposition~\ref{cor:r(n)finite}~\ref{cor:case:1}, the
limiting distribution of $M_n$ is therefore $\pMargin^{1/\theta}$.
\end{example}

Proposition~\ref{cor:r(n)finite}~\ref{cor:case:1} and the subsequent examples
show that any continuous distribution can arise as a limit of suitably
normalized maxima of some dependent sequence of identically distributed random
variables. In particular, there is no guarantee that the limiting distribution
is max-stable, which contrasts the iid
case. %
In general there is thus no hope of obtaining a single nontrivial class of all
possible nondegenerate limits of suitably normalized maxima of identically
distributed dependent random variables. Nonetheless, two analogues of the
Fisher--Tippett--Gnedenko theorem can be obtained under specific
assumptions. Assuming $\pMargin \in \MDA(\EVDist)$ for some $\EVDist\in \GEV$,
Theorem~\ref{theorem-FTG2} is a simple consequence of
Theorem~\ref{theorem_continuous_limit}~\ref{mainThm:1} and the convergence to
types theorem \cite[Proposition~0.2]{Resnick1987}.

\begin{corollary}%
\label{corollary_FTG}
Consider $X_1,X_2,\ldots\sim\pMargin$ and let, for $n\in\IN$, $\Cdiag_n$ be the
diagonal of the copula $C_n$ of $(X_1,\ldots,X_n)$.  Assume that
$\pMargin\in\MDA(\EVDist_{\EVI,\mu,\sigma})$ with
$\EVDist_{\EVI,\mu,\sigma}\in\GEV$ and that there exists
$\rate \colon \IN \to (0,\infty)$ such that $\lim_{n\to\infty} \rate_n = \infty$
and $\pD^{\rate}_n$ converges pointwise to a continuous function $\pD$ on
$[0,1]$.  If there exist sequences $(\cn_n)$, $\cn_n > 0$, and $(\dn_n)$ such
that \eqref{eq:central} holds for a nondegenerate $\pG$, then there exist $c>0$
and $d \in \mathbb{R}$ such that
$\pG = \pD\circ \EVDist_{\EVI,\tilde{\mu},\tilde{\sigma}}$, where
$\EVDist_{\EVI,\tilde{\mu},\tilde{\sigma}}\in\GEV$ with
$\tilde{\mu} = (\mu-d)/c$ and $\tilde{\sigma} = \sigma/c$.
\end{corollary}

\begin{remark} \label{rem_rate_uniqueness}
  While the rate $\rate$ in
  Corollary~\ref{corollary_FTG} is not unique, the speed at which it tends to
  $\infty$ is. Suppose that
  the assumptions of Corollary~\ref{corollary_FTG} hold and that
  $(\an_n)$ and $(\bn_n)$ are the normalizing sequences of the maxima of the
  associated iid sequence. %

  Suppose $\tilde \rate \colon \IN \to (0,\infty)$ is such that
  $r_n / \tilde \rate_n \to \theta$ as $n \to \infty$ where
  $\theta \in (0,\infty)$.  Obviously, $\tilde \rate_n \to \infty$ as
  $n \to \infty$. Because $\pD$ is continuous, $\pD^{\rate}_n \to \pD$ uniformly
  on $[0,1]$ by Lemma~\ref{lemma_uniform_convergence} and hence
  $\pD^{\tilde\rate}_n(u) \to \pD(u^\theta)=:\widetilde\pD(u)$ for all
  $u \in [0,1]$. From Theorem~\ref{theorem_continuous_limit}~\ref{mainThm:1}
  we obtain that \eqref{eq:central} holds with
  $\cn_n = \an_{\ceil{\tilde \rate_n}}$, $\dn_n = \bn_{\ceil{\tilde \rate_n}}$
  and
  $G=\widetilde \pD \circ \EVDist_{\EVI,\mu,\sigma} = \pD \circ
  \EVDist^\theta_{\EVI,\mu,\sigma}$. By direct calculation,
  $\EVDist^\theta_{\EVI,\mu,\sigma} = \EVDist_{\EVI,\tilde\mu,\tilde\sigma}$,
  where $\tilde \mu = \mu+\sigma\log(\theta)$ if $\EVI=0$ and $\tilde \mu = \mu$
  otherwise, while $\tilde \sigma = \sigma$ if $\EVI =0$ and
  $\tilde \sigma = \sigma\theta^{1/|\EVI|}$ otherwise. Applying the
  convergence to types theorem as in the proof of Corollary~\ref{corollary_FTG}
  yields
  $\an_{\ceil{\tilde \rate_n}}/\an_{\ceil{\rate_n}} \to \sigma/\tilde \sigma$
  and
  $(\bn_{\ceil{\tilde \rate_n}}-\bn_{\ceil{\rate_n}})/\an_{\ceil{\rate_n}} \to
  \mu - \tilde \mu \sigma/\tilde \sigma$.

Now suppose that $\tilde \rate \colon \IN \to (0,\infty)$ is such that $\rate_n / \tilde \rate_n \to 0$. In this case, $\tilde \rate_n \to \infty$, but this time the uniform convergence of $\pD^{\rate}_n$ to $\pD$ implies that $\pD^{\tilde \rate}_n(u) \to 1$ as $n \to \infty$ if $u \in (0,1]$, while  $\pD^{\tilde \rate}_n(0) = 0$ for all $n \in \IN$. Consequently, for any $x \in \IR$ such that $\EVDist_{\EVI,\mu,\sigma}(x) > 0$, $\Prob(M_n \le \an_{\ceil{\tilde \rate_n}} x + \bn_{\ceil{\tilde \rate_n}}) \to 1$ as $n \to \infty$. Thus if $\xi \in (-\infty, 0]$, the limiting behavior of $(M_n - \bn_{\ceil{\tilde \rate_n}})/\an_{\ceil{\tilde \rate_n}}$ is degenerate. The same is true when $\xi > 0$ and $D$ is such that $D(u) < 1$ for all $u \in [0,1)$, as can be argued from the convergence to types theorem.

The case of $\tilde \rate \colon \IN \to (0,\infty)$ with $\rate_n / \tilde \rate_n \to \infty$ and $\tilde \rate_n\to \infty$ as $n \to \infty$ is similar. We obtain that whenever $\EVDist_{\EVI,\mu,\sigma}(x) <1$, $\Prob(M_n \le \an_{\ceil{\tilde \rate_n}} x + \bn_{\ceil{\tilde \rate_n}}) \to 0$ as $n \to \infty$. If $\xi \in [0,\infty)$, this implies that the weak limit of $(M_n - \bn_{\ceil{\tilde \rate_n}})/\an_{\ceil{\tilde \rate_n}}$ is degenerate. For $\xi < 0$, the same follows from the convergence to types theorem if additionally $D(u) > 0$ for all $u \in (0,1]$.
\end{remark}

The second generalization of the Fisher--Tippett--Gnedenko theorem can be
derived without assuming that $\pMargin \in \MDA(\EVDist)$ for some
$\EVDist\in \GEV$, at the cost of stronger assumptions on the underlying
dependence structure than those in Corollary~\ref{corollary_FTG}.

\begin{theorem}%
  \label{theorem-FTG2}
  Let $X_1,X_2,\ldots\sim\pMargin$.
  Suppose that the diagonal $\Cdiag_n$ of the copula $\Copula_n$ of
  $(X_1,\dots,X_n)$ is strictly increasing for each $n$ and that there exists a
  function $\rate \colon \IN \to (0,\infty)$, and a bijection
  $\lambda\colon(0,\infty) \to (0,\infty)$ such that the following conditions
  hold:
\begin{enumerate}[label=(\roman*), labelwidth=\widthof{(iii)}]
  \item\label{FTG:1} $\lim_{n\to\infty} \rate_n = \infty$ and
    $\lim_{n\to\infty} \rate_{\ceil{tn}}/\rate_n = \lambda(t)$, $t > 0$;
  \item the diagonal power distortion $\pD^{\rate}_n$ with respect to $\rate$ converges pointwise to a continuous and
    strictly increasing bijection $\pD\colon[0,1] \to [0,1]$.
  \end{enumerate}
  If there exist sequences $(\cn_n)$, $\cn_n>0$ and $(\dn_n)$ such that \eqref{eq:central} holds for a
  nondegenerate $\pG$, then $\pG = \pD\circ\EVDist$ where $\EVDist \in \GEV$.
\end{theorem}

\begin{remark}%
An interesting special case of Theorem~\ref{theorem-FTG2} arises when $r_n$ satisfies $n/r_n \to \alpha$ as $n\to \infty$ for some $\alpha > 0$. Condition~\ref{FTG:1} of Theorem~\ref{theorem-FTG2} then holds with $\lambda(t) = t$ for all $t  > 0$. From Corollary~\ref{corollary_pointwise_convergence_power_limit} we have that $\pMargin^n(\cn x + \dn) \to H^\alpha(x)$ for all $x \in \IR$, i.e.\ $\pMargin \in \MDA(H^\alpha)$, where $H^\alpha$ is of the same type as $H$ given that $H$ is max-stable.
\end{remark}
We close this section with an example where all calculations can be done explicitly. %
\begin{example}\label{example_moving_max}%
  Consider the following case of a moving maximum process of \cite{Newell1964,Deheuvels1983}, see
  \cite[Chapter~10]{BeirlantGoegebeurTeugelsSegers2004} for an overview. For a fixed
  $k \geq 0$, let $(Z_i)_{i=-k+1}^{\infty}$ be a sequence of iid standard
  Fr\'{e}chet random variables and set
  $Y_i = (1/(k+1))\max_{0 \leq j \leq k} Z_{i-j}$.  Then $Y_i$ is also standard Fr\'{e}chet.
  For $n\geq k+1$ the joint distribution of $\bfY_n = (Y_1,\ldots,Y_n)$ can be calculated to be
  $\JointDist_{\bfY_n}(y_1,\ldots,y_n) = \prod_{i=-k+1}^{0}\Prob(Z_{i} \leq (k+1) \min_{1\leq \ell \leq i+k} y_{\ell})\times \prod_{i=1}^{n-k}\Prob(Z_{i} \leq (k+1) \min_{0 \leq \ell \leq k} y_{i+\ell}) \times \prod_{i=n-k+1}^{n}\Prob(Z_{i} \leq (k+1) \min_{i \leq \ell \leq n} y_{\ell})$.
  Taking into account the marginal distribution of $Y_i$ and using that
  $\min_{j\in J}(-1/\log(u_j))$ equals $-1/\log(\min_{j\in
    J} u_j)$ for any index set $J$, the copula of $\bfY_n$ is given by
\begin{align*}%
\Copula_n(u_1,\ldots,u_n) &=
 \prod_{i=-k+1}^{0}\min_{1\leq \ell \leq i+k} u_{\ell}^{1/(k+1)}  \prod_{i=1}^{n-k}\min_{0 \leq \ell \leq k} u_{i+\ell}^{1/(k+1)}  \prod_{i=n-k+1}^{n} \min_{i \leq \ell \leq n} u_{\ell}^{1/(k+1)}.
\end{align*}
Comparing $\Copula_n$ to \cite[Equation~(5.3)]{Mulinacci2015} (see also
  \cite{Li2008}), the copula can be identified as Marshall--Olkin copula, i.e.\ the survival copula of a
  multivariate Marshall--Olkin distribution with appropriate choices of
  within-group intensities to keep only certain terms in the product
  over all non-empty subsets of $\{1,\ldots,n\}$.
We now have
  $\Cdiag_n(u) = u^{(n+k)/(k+1)}$, leading to
  $\pD_n^n(u) = \Cdiag_n(u^{1/n}) = u^{(n+k)/(n(k+1))} \to u^{1/(k+1)} =
  \pD(u)$ as $n \to \infty$.

  Now pick an arbitrary continuous distribution function
  $\pMargin\in\MDA(\EVDist_{\EVI,\mu,\sigma})$ and set
  $X_i = \pMargin^{-1}(\pMargin_{Y_i}(Y_i))$ to obtain a sequence $(X_i)$ with
  univariate margins $\pMargin$ and the Marshall--Olkin copula $\Copula_n$
  as the copula of $(X_1,\ldots,X_n)$.  If $(\an_n)$
  and $(\bn_n)$ denote the normalizing sequences from the iid case and if the rate
  function is given by $\rate_n=n$, we recover from
  Theorem~\ref{theorem_continuous_limit}~\ref{mainThm:1} that
  $\Prob\{(M_n-\bn_n)/\an_n \leq x\} \to \EVDist^{\theta}_{\EVI,\mu,\sigma}(x)$, $x \in \IR$,
  where $\theta = 1/(k+1) \in (0,1]$ and $M_n = \max(X_1,\ldots, X_n)$.  The
  fact that the distortion $\pD(u) = u^{\theta}$ is a power function has the
  interesting effect that the weak limit of the maximum under dependence is
  again $\GEV$, as already addressed in Remark~\ref{rem_rate_uniqueness}.
 This result about the moving maximum process is well-known and $\theta=1/(k+1)$ is indeed its extremal index
  \cite[Example~10.5]{BeirlantGoegebeurTeugelsSegers2004}.

  Finally, notice that any rate
  function $\rate_n = n/\alpha$ for $\alpha > 0$ could have been used to obtain
  $\Cdiag_n(u^{1/(n/\alpha)}) \to u^{\alpha/(k+1)}$.  Because
  $\lim_{n\to\infty} \rate_{\ceil{nt}}/\rate_n = t$ for any $\alpha>0$,
  Theorem~\ref{theorem-FTG2} implies that changing the rate function will only
  lead to a location-scale transform of the limiting distribution but will not
  change its functional form otherwise. This observation is in line with the fact that
  the changed rate $\rate_n=n/\alpha$ also influences the utilized normalizing
  constants, leading to the location-scale transform of the limit.
\end{example}

\begin{remark}%
  Although it is a limit of rescaled copula diagonals, $\pD$ in
  Theorem~\ref{theorem-FTG2} is not necessarily a copula diagonal
  itself. Indeed, by the Fr\'echet--Hoeffding inequality
  \cite[Theorem~2.10.12]{Nelsen2006}, any copula diagonal must satisfy that for
  all $u\in(0,1)$, $\Cdiag_n(u) \leq u$. However, the limiting $\pD$ of the
  moving maximum process in Example~\ref{example_moving_max} satisfies
  $\pD(u) = u^\theta > u$ for $u\in(0,1)$ whenever $\theta < 1$.
\end{remark}

\section{Sequences with asymptotic power diagonals}\label{sec:3}
An appealing feature of the moving maximum process in
Example~\ref{example_moving_max} is that the limiting distribution of normalized
maxima is still generalized extreme value. In this section, we
explore the consequences of the results in
Section~\ref{section_extremes_and_copula_diagonals} for a broad class of
sequences which behave similarly in the sense that the weak limit of suitably normalized maxima is $\GEV$ because the diagonal power distortion
converges to a power.  We first formalize this property in a definition.
\begin{definition}\label{def:3.1}%
  A sequence of identically distributed random variables
  $X_1,X_2, \ldots\sim\pF$ with continuous marginal distribution function
  $\pMargin$ has a \emph{power diagonal} if, for all $n \in \IN$ and
  $u \in (0,1)$, $\delta_n(u) = u^{\eta_n}$ for some sequence $(\eta_n)$ such
  that $\eta_n \to \infty$ as $n \to \infty$. And
  $X_1,X_2, \ldots\sim\pF$ has an \emph{asymptotic power diagonal with index}
  $\theta >0$ if there exists a rate function $\rate : \IN \to (0,\infty)$ with
  $\rate_n \to \infty$ as $n \to \infty$ and such that the diagonal power
  distortion satisfies $\pD_n^{\rate}(u) \to u^\theta$ for all $u \in [0,1]$.
\end{definition}
Before proceeding, note that a sequence with a power diagonal necessarily has an asymptotic
power diagonal; it suffices to set $r_n = \eta_n$ as then
$\pD_n^{\rate_n}(u) = u$ for all $u \in [0,1]$. Furthermore, $\eta_n \to \infty$
and $\rate_n \to \infty$ as $n\to \infty$ are required in view of
Proposition~\ref{cor:r(n)finite}, so that all possible limiting distributions of
normalized maxima can be characterized.  Also, any copula is bounded above
by the comonotone copula \cite[Theorem~2.10.12]{Nelsen2006}, so that
$\pD_n^{\rate}(u)\le u^{1/\rate_n}$ for all $u \in (0,1)$. Given
that $u^{1/\rate_n}\to 1$ as $n \to \infty$ if $\rate_n \to \infty$, any
$\theta$ such that $\pD_n^{\rate}(u) \to u^\theta$, $u \in (0,1)$, must
satisfy $\theta \geq 0$. We excluded the case $\theta=0$ in
Definition~\ref{def:3.1} because the limiting distortion $\pD$ would then be
degenerate.%

If $X_1,X_2, \ldots\sim\pF$ has an asymptotic power diagonal and its marginal distribution
function satisfies $\pMargin \in \MDA(\EVDist)$, Theorem~\ref{theorem_continuous_limit}~\ref{mainThm:1}
guarantees that
$( M_n - \bn_{\ceil{\rate_n}})/\an_{\ceil{\rate_n}} \indist \EVDist^\theta$,
where $(\bn_n)$ and $(\an_n)$ are the normalizing sequences from the iid case, i.e.\ such that
$\pMargin^n(\an_n x + \bn_n) \to \EVDist(x)$ as $n \to \infty$ for any $x \in \IR$.
Because $\EVDist\in \GEV$, $H^\theta$ is generalized extreme value with the same shape parameter as $\EVDist$, as detailed in Remark~\ref{rem_rate_uniqueness}. The latter also implies that if $\rate_n/n \to 0$ or $\rate_n/n \to \infty$, the weak limit of $(M_n-\bn_n)/\an_n$ is degenerate.

If $X_1,X_2, \ldots\sim\pF$ has an asymptotic power diagonal but not necessarily $\pMargin \in \MDA(\EVDist)$,  any non-degenerate $G$ in \eqref{eq:central} is still $\GEV$ by Theorem~\ref{theorem-FTG2} as long as $\delta_n$ is strictly increasing for each $n$ and $\rate_{\ceil{tn}}/\rate_n \to \lambda(t)$
for all $t > 0$ and some bijection $\lambda\colon(0,\infty) \to (0,\infty)$.

We discuss sequences with power diagonals in Section~\ref{sec:3.1} and relate to sequences that satisfy the so-called distributional mixing condition in Section~\ref{sec:extremes:for:time:series}.

\subsection{Sequences with power diagonals}\label{sec:3.1}
Apart from an iid sequence where $C_n$ is the independence copula
$\Pi_n(u_1,\dots,u_n)=\prod_{i=1}^nu_i$ with $\delta_n(u) = u^n$, $u\in (0,1)$,
an example of a sequence with a power diagonal is the moving maximum process in
Example~\ref{example_moving_max}. We can generalize it if we notice that the
copula $C_n$ in the latter example is in fact extreme value. From, e.g.\
\cite{Huang1992,BeirlantGoegebeurTeugelsSegers2004,GudendorfSegers2010}, this
means that $C_n$ is of the form
$C_n(u_1,\dots, u_n) = \exp[-\STDF_n\{-\log (u_1),\dots,-\log(u_n)\}]$
  for all $u_1,\dots, u_n \in (0,1)$, where $\ell_n: [0,\infty)^n\to[0,\infty)$ is a \emph{stable tail dependence function (stdf)}, i.e.\ a map which is homogeneous of order one and has further analytical properties identified in \cite{Ressel2013}.
 This motivates the following definition.

 \begin{definition}\label{def_meta_extreme_sequence}%
   We call any sequence $(X_i, i\in\IN)$ of identically distributed random
   variables with common continuous marginal distribution $\pF$ a
   \emph{meta-extreme sequence}, if the copula $\Copula_n$ of $(X_1,\dots,X_n)$
   is an extreme value copula for any $n\in\IN$.
 \end{definition}

 Because any stdf $\ell_n$ is homogeneous of order
 one, the diagonal of an extreme value copula satisfies
 $\Cdiag_n(u) = \exp\{(\log u)\STDF_n(1,\ldots,1)\}=u^{\eta_n}$, where
 $\eta_n = \STDF_n(1,\ldots, 1)$ is the extremal coefficient of Smith
 \cite{Smith1990}; see also \cite[Chapter~2.3]{Falk2019}.  This means that
 meta-extreme sequences have a power diagonal provided that $\eta_n \to \infty$
 as $n \to \infty$. Example~\ref{example_Cuadras_Auge} shows that this latter
 condition is indeed a restriction; the Cuadras--Aug\'e copula appearing therein
 is extreme value but such that $\eta_n$ has a finite limit as $n\to\infty$.

 By the characterization of multivariate max-stable distributions in terms of
 D-norms described in \cite[Section~2.3, Theorem~2.3.3.]{Falk2019}
 and the Takahashi characterization (\cite[Corollary~1.3.2]{Falk2019}), we have $\eta_n=n$ for all
 $n \in \IN$ if and only if the elements of the meta-extreme sequence $(X_i)$
 are iid. Next, we discuss two specific examples of meta-extreme sequences.

\begin{example}\label{ex:Mai}
  Exchangeable meta-extreme sequences are fully characterized in \cite{Mai2019},
  where it is shown on p.~167 that the na\"ive, bottom-up construction approach
  of selecting an exchangeable STDF $\tilde \ell$ in
  a fixed dimension $n$ in the hope that there exists an exchangeable
  meta-extreme sequence with $\ell_n = \tilde \ell$ fails in
  general. Generalizing de Haan's spectral representation \cite{DeHaan1984}, it
  is shown in \cite{Mai2019} that $(X_i)$ is an exchangeable meta-extreme
  sequence if and only if there exists $b \in [0,1]$ and an exchangeable
  sequence $(W_i)$ of non-negative random variables with unit mean so that for
  each $n \ge 2$ and $t_1,\dots, t_n \ge 0$,
  $\STDF_n(t_1,\dots, t_n) = b \sum_{j=1}^n t_j + (1-b) \E[\max_{1 \le i \le n}
  (t_i W_i)]$.  Clearly, $b=1$ corresponds to an iid sequence $(X_i)$. When
  $b > 0$, the sequence $(X_i)$ has a power diagonal, while when $b\in[0,1)$,
  the behavior of the extremal coefficient $\eta_n = \STDF_n(1,\dots, 1)$
  depends on $(W_i)$. Also when $b=0$,
  $\ell_n(t_1,\dots,t_n) = \|(t_1,\dots, t_n)\|_D$, where $\|\cdot \|_D$ is the
  so-called D-norm generated by $(W_1,\dots, W_n)$ \cite[Lemma~1.1.3]{Falk2019}.

  The special case when $b=0$ and $(W_i)$ is an iid sequence is treated in
  \cite{Mai2018}, where it is also shown how to construct meta-extreme sequences
  starting from the distribution function $F_W$ of $W_1$, provided that
  $F_W(0) < 1$. This construction first defines a stochastic process $H_t$,
  $t \ge 0$, by
  $H_t = -\ln [\prod_{k=1}^\infty F_W\{(\varepsilon_1 + \dots +
  \varepsilon_k)/t-\}]$, where $F_W(w-)$ is the left limit of $F_W$ at $w$ and
  $(\varepsilon_i)$ is an iid sequence of unit exponentials.  The meta-extreme
  sequence $(X_i)$ with univariate margin $F$ is then obtained by setting
  $X_i = F^{-1}( e^{-Y_i})$, where $Y_i = \inf\{ t > 0 : H_t > \xi_i\}$ and
  $(\xi_i)$ is an independent copy of $(\varepsilon_i)$. From \cite[Lemma~2
  and~3]{Mai2018}, $(X_i)$ has a power diagonal if and only
  if %
  $\inf\{t : F_W(t) = 1\} = \infty$.

  Specific choices for $F_W$ lead to meta-extreme sequences with well-known
  stdfs. From \cite[Example~1 and Example~2]{Mai2018} and
  \cite[Section~6]{Belzile/Neslehova:2017}, we can take $W$ to be scaled gamma
  with parameters $\alpha > 0$, $\rho > -\alpha$ and density
  $f_W(w) = \{|1/\rho|/\Gamma(\alpha)\} a^{-\alpha/\rho} x^{\alpha/\rho-1}
  e^{-(w/a)^{1/\rho}}$, $w > 0$, where $\Gamma(\cdot)$ is Euler's gamma
  function and $a=\Gamma(\alpha)/\Gamma(\rho + \alpha)$. Setting $\rho=1$ yields
  the (symmetric) Coles--Tawn extremal Dirichlet sequence, while when
  $\alpha=1$, then $\rho > 0$ and $\rho \in (-1,0)$ lead to the negative and
  positive logistic sequences, respectively. The positive logistic (or
  Gumbel--Hougaard) stdf is usually parametrized in
  terms of $\theta = -1/\rho$ and is given, for for all $t_1,\dots, t_n \ge 0$,
  by
  $ \ell_{n,\text{Gu}}(t_1,\dots, t_n) = (t_1^{\theta} + \dots + t_n^{\theta})^{1/\theta}$.
\end{example}

\begin{example}%
\label{ex:evds}
Another approach to construct meta-extreme sequences is through a more general
class of simple max-stable processes on $[0,1]$, where we follow
\cite[Chapter~9]{deHaanFerreira2006}. %
Let $C[0,1]$ denote the space of continuous functions on $[0,1]$ equipped with the supremum norm and $C^{+}[0,1]$ its subspace of strictly positive functions. The process $S$ on $C^+[0,1]$ is \emph{simple max-stable} if for all $t \in [0,1]$, $\Prob(S(t) \le z) = e^{-1/z}$ for all $z \ge 0$ and if for all $k \in \IN$, $(1/k) \vee_{i=1}^k S_i \eqindist S $, where $S_1,S_2,\dots$ are iid copies of $S$, $\vee$ is the pointwise maximum operator, and $\eqindist$ denotes equality in distribution.

Starting with a simple max-stable process and a strictly increasing sequence
$(t_i)$ in $[0,1]$, %
e.g.\ $t_i = 1 - 1/(i+1)$ for $i\geq 1$, we can define $(X_i)$ by
$X_i = F^{-1}(e^{-1/S(t_i)})$. Because $S(t_i)$ is standard Fr\'echet,
$X_i \sim F$. The invariance principle \cite[Theorem~6.5.6]{schweizersklar1983}
implies that for each $n \in \IN$, the copula $C_n$ of $(X_1,\dots, X_n)$ is the
same as that of $(S(t_1),\dots, S(t_n))$ and consequently extreme value, so that
$(X_i)$ is a meta-extreme sequence. Following \cite{DeHaan1984}, $S$ admits the
stochastic representation $S \eqindist \vee_{k\geq 1}(\xi_k W_k)$, where
$(\xi_k)$ is an enumeration of points of a Poisson point process $\xi$ on
$[0,1]$ with intensity $t^{-2} \d t$ and which is independent of the iid copies
$W_1, W_2, \dots$ of a stochastic process $W$ on $C^{+}[0,1]$ such that
$\E[W(t)] = 1$ for all $t \in [0,1]$ and
$\E\big[\sup_{t \in [0,1]} W(t)\big] < \infty$
\cite[Corollary~9.4.5]{deHaanFerreira2006}. This implies that the stdf $\ell_n$
of $C_n$ satisfies
$\STDF_n(x_1,\dots, x_n) = \E[ \max_{1 \le i \le n} x_i W(t_i)]$ for all
$x_1,\dots, x_n \ge 0$, so that $\ell_n(x_1,\dots,
x_n)$ %
is the D-norm generated by $(W(t_1),\dots,
W(t_n))$ \cite[Lemma~1.1.3]{Falk2019}.
Note however that the meta-extreme sequence $(X_i)$ constructed this way cannot have a power diagonal %
because $\eta_n \not \to \infty$. Indeed, for each $n \in \IN$,
$\eta_n =\ell_n(1,\dots, 1) = \E\{ \max_{1 \le i \le n} W(x_i)\} \le \E \{\sup_{t \in [0,1]} W(t)\} < \infty$.
Since $\eta_n \ge 1$ and the sequence $(\eta_n)$ is non-decreasing, there exists $\varrho > 0$ so that $\eta_n \to \varrho$. Proposition~\ref{cor:r(n)finite}~\ref{cor:case:1} thus implies that $M_n = \max(X_1,\dots, X_n)$ converges weakly to $F^\varrho$.
\end{example}
The next example shows that power diagonals arise not
only from extreme value copulas, demonstrating that the previous
discussion is not limited to meta-extreme sequences.
 \begin{example}%
 \label{ex:FN}
 Consider a sequence of iid bivariate random vectors $(U_i,V_i)$, $i \in
 \{0,1,\dots\}$ such that $(U_i,V_i)\sim\Copula_{\FN}$ where $\Copula_{\FN}(u,v)
 = \min\{u,v,(u^2+v^2)/2\}$,
 $u,v\in[0,1]$. This copula is a special case of the so-called bivariate
 diagonal copulas introduced and studied in \cite{FredricksNelsen1997}. For each
 $i \in \IN$, set $Y_i =
 \max(V_{i-1}^2,U_i^2)$. This construction is similar to the
 (finite-dimensional) product type copula construction of \cite{Liebscher2008}
 and to the copulas discussed in \cite{MazoGirardForbes2015}. Given that
 $V_{i-1}$ and $U_i$ are independent, the random variables
 $Y_i$ are standard uniform and the joint distribution function of $(Y_1,\dots,
 Y_n)$ is a copula given by $\Copula_n(u_1,\ldots,u_n) = \sqrt{u_1} \times
 \prod_{j=1}^{n-1} \Copula_{\FN}(\sqrt{u_{j}},\sqrt{u_{j+1}}) \times
 \sqrt{u_n}$, $u_1,\dots, u_n \in [0,1]$.  Clearly, $\Cdiag_n(u)=u^n$ for all $u
 \in [0,1]$, i.e.\ that $C_n$ has the same power diagonal as
 $\Pi_n$. This means that if we take any continuous distribution function
 $F$, the maxima of the sequence $X_i = F^{-1}(Y_i)$, $i \in
 \IN$, behave in the same way as the maxima of the associated iid sequence. Yet,
 $\Copula_{\FN}$ and hence also
 $\Copula_n$ is not an extreme value copula, so that
 $(X_i)$ is not meta-extreme.
  \end{example}

  \subsection{Stationary sequences with short-range extremal dependence}\label{sec:extremes:for:time:series}

  We now show that asymptotic power diagonals are inherent to strictly
  stationary sequences with a certain form of asymptotic independence in the
  tail. Conditions that formalize the latter property are regularly used to
  study maxima of strictly stationary time series, see e.g.\
  \cite[Chapter~3]{LeadbetterLindgrenRootzen1983}, or
  \cite[Chapter~10]{BeirlantGoegebeurTeugelsSegers2004} and
  \cite[Chapter~4]{EmbrechtsKluppelbergMikosch1997}. %

 \begin{definition}\label{def:3.3}
   A strictly stationary sequence $X_1,X_2, \ldots \sim F$ is said to satisfy:
   \begin{enumerate}[label=(\roman*), labelwidth=\widthof{(iii)}]
   \item The \emph{distibutional mixing condition} $\CondDu$ if for any integers
     $p,q,n$ and indices $1 \le i_1 < \ldots < i_p < j_1 < \ldots < j_q \le n$
     such that $j_1 -i_p \ge s$,
     $\abs{\Prob(\max_{i \in A \cup B} X_i \le u_n) - \Prob(\max_{i \in A} X_i
       \le u_n)\Prob(\max_{i \in B} X_i \le u_n)} \le \alpha(n,s)$, where
     $A=\{i_1,\ldots, i_p\}$, $B=\{j_1,\ldots,j_q\}$, and $\alpha(n,s_n) \to 0$
     as $n\to \infty$ for some positive integer sequence $s_n = o(n)$.
   \item The \emph{anticlustering condition} $\mathcal{D}^\prime(u_n)$ if
     $\lim_{k \to \infty} \limsup_{n\to \infty} n \sum_{j=2}^{\lfloor n/k
       \rfloor} \Prob(X_1 > u_n, X_j > u_n) = 0$.
   \end{enumerate}
 \end{definition}

 The following result, relating $\CondDu$ and $\mathcal{D}^\prime(u_n)$ to the
 behavior of $\delta_n$, is a consequence of
 Theorem~\ref{theorem_continuous_limit}~\ref{mainThm:2} and limit theorems for
 maxima of stationary series in
 \cite{Leadbetter1974,Leadbetter1983,LeadbetterLindgrenRootzen1983}.

 \begin{corollary}\label{corollary_copula_diagonal_extremal_index}
   Let $X_1,X_2,\ldots \sim \pMargin$ be a strictly stationary sequence. Suppose
   that $\pMargin$ is continuous and satisfies $\pMargin \in \MDA(\EVDist)$ with
   normalizing sequences $(\cn_n)$, $\cn_n > 0$ and
   $(\dn_n)$. %
   \begin{enumerate}[label=(\roman*), labelwidth=\widthof{(iii)}]
   \item\label{item:diag:1} If the condition $\mathcal{D}(u_n)$ is satisfied
     with $u_n = \cn_n x + \dn_n$ for each $x$ such that $\EVDist(x)>0$ and
     there exists a $u \in (0,1)$ such that $\pD_n^{\rate}(u) \to \gamma$ where
     $\gamma \in (0,1)$ and $r_n = n$ for each $n \in \IN$, then the sequence
     $X_1,X_2,\ldots \sim \pMargin$ has an asymptotic power diagonal with index
     $\theta \in (0,1]$ and $\theta$ is the extremal index of $(X_i)$, meaning
     that $(M_n-\dn_n)/\cn_n \indist \EVDist^\theta$.
   \item\label{item:diag:2} If the conditions $\mathcal{D}(u_n)$ and
     $\mathcal{D}^\prime(u_n)$ hold with $u_n = \cn_n x + \dn_n$ for each
     $x \in \IR$, then the sequence $X_1,X_2,\ldots \sim \pMargin$ has an asymptotic
     power diagonal with index $\theta=1$ and the rate may be chosen as
     $r_n = n$ for each $n \in \IN$.
   \end{enumerate}
 \end{corollary}

 \begin{remark}\label{rem:3.1}
   When the rate is $\rate(n) = n$, $n \in \IN$, Lipschitz continuity of
   copulas implies that $\abs{\pD_n^{\rate}(v_n) - \pD_n^{\rate}(u)} \to 0$ as
   $n\to \infty$ whenever $v_n \to u$ where $u \in (0,1)$, as argued in the
   proof of
   Theorem~\ref{theorem_continuous_limit}~\ref{mainThm:2}. Consequently,
   provided that $\pMargin \in \MDA(\EVDist)$ with normalizing sequences
   $(\cn_n)$, $\cn_n > 0$ and $(\dn_n)$, the condition that $\pD_n^{\rate}(u)$
   converges for some $u \in (0,1)$ is equivalent to
   $\Prob(M_n \le \cn_n x + \dn_n)$ converging for some $x$ such that
   $\EVDist(x) \in (0,1)$.

   Finally, if $F \in \MDA(\EVDist)$ and $(X_i)$ has extremal index
   $\theta \in (0,1]$, meaning that $(M_n-\dn_n)/\cn_n$ converges weakly to
   $H^\theta$ with the same normalizing sequences $(\cn_n)$ and $(\dn_n)$ that
   stabilize the maxima of the associated iid sequence
   $X_1^*, X_2^*,\ldots \sim \pMargin$, then
   Theorem~\ref{theorem_continuous_limit}~\ref{mainThm:2} implies that $(X_i)$
   has an asymptotic power diagonal with index $\theta$ and rate $r_n = n$,
   $n\in \IN$. Thus, for large enough $n$,
   $\pD_n^{\rate}(u) = \Cdiag_n(u^{1/n}) \approx \pD(u) = u^{\theta}$, so that
   $\Cdiag_n(u) \approx u^{\theta n}$, where $u^{\theta n}$ may be interpreted
   as the copula diagonal of $\theta n$ independent variables.
 \end{remark}

 The moving maximum process in Example~\ref{example_moving_max} is a case in
 point where Corollary~\ref{corollary_copula_diagonal_extremal_index}~\ref{item:diag:1}
 applies. The
 $\mathcal{D}(u_n)$ condition holds and the extremal index equals
 $\theta = 1/(k+1)$, $k\geq 0$
 \cite[Example~10.5]{BeirlantGoegebeurTeugelsSegers2004}. The fact that the
 moving maximum process has an asymptotic power diagonal could thus have been
 alternatively obtained from
 Corollary~\ref{corollary_copula_diagonal_extremal_index}~\ref{item:diag:1}.

 As we illustrate in the next example,
 Corollary~\ref{corollary_copula_diagonal_extremal_index} reveals facts about
 the limiting behavior of copula diagonals of stationary sequences whose copulas
 may be intractable or not even explicit. This is the case for most classical
 time series models, such as ARMA or GARCH processes (although we note that copulas have
 also been used explicitly to construct time series models, examples of the
 latter are Markov processes \cite{DarsowNguyenOlsen1992}, models that use
 neural networks to capture cross-sectional dependence
 \cite{hofertprasadzhu2022}, or vine copula models \cite{NaglerKrugerMin2022}).

 \begin{example}\label{ex:3.4}
   Consider a stationary Gaussian sequence $(Y_i)$, meaning that all
   finite-dimensional distributions are multivariate normal. From Sklar's
   theorem, the copula $\Copula_n$ of $(Y_1,\dots, Y_n)$ is a Gaussian
   copula. Although the univariate normal distribution is in the maximum domain
   of attraction of the Gumbel extreme value distribution $\Lambda$, the tail
   properties of $\Cdiag_n$ may not be easy to investigate. For example, take
   the Gaussian AR(1) process given by $Y_n = \phi Y_{n-1} + Z_n$ for
   $n\in \mathbb{Z}$, where $\phi\in(-1,1)$ and $(Z_n)$ is an iid sequence with
   $Z_n \sim \Normal(0,\sigma^2)$. It is easily shown that
   $(Y_1,\ldots,Y_n)\sim\Normal(\bm{0},\Sigma_n)$ with
   $(\Sigma_n)_{ij} = (\sigma^2 \phi^{\abs{i-j}})/(1-\phi^2)$.
   Denote by $\Phi$ and $\Phi_{\Sigma_n}$ the distribution function of the
   $\Normal(0,1)$ and $\Normal(\bm{0}, \Sigma_n)$, respectively. The diagonal of
   $\Copula_n$ is then
   $\Cdiag_n(u) = %
   \Phi_{\Sigma_n}\bigl(\sigma\Phi^{-1}(u)/\sqrt{1-\phi^2},\ldots,\sigma\Phi^{-1}(u)/\sqrt{1-\phi^2}\bigr)$, $u \in (0,1)$.

   It is not easy to investigate the limiting behavior of
   $\Cdiag_n(u^{1/\rate_n})$ for some suitable rate $r$.  However, for a
   stationary Gaussian sequence $(V_i)$, the so-called Berman condition
   \begin{align}\label{eq:Berman}
     \lim_{n\to\infty} \cov(V_1,V_n) \ln(n) = 0
   \end{align}
   ensures that the $\mathcal{D}(u_n)$ and $\mathcal{D}^\prime(u_n)$ conditions
   hold for any $u_n = \an_n x + \bn_n$ and $x \in \IR$, where $(\an_n)$ and
   $(\bn_n)$ are the normalizing sequences of the associated iid sequence
   \cite[Lemma~4.4.7]{EmbrechtsKluppelbergMikosch1997}. Corollary~\ref{corollary_copula_diagonal_extremal_index}~\ref{item:diag:2}
   thus ensures that $(V_i)$ has an asymptotic power diagonal with index
   $\theta=1$ and rate $r_n=n$. This is the case for the above AR(1) process,
   because $\cov(Y_1,Y_n) = \sigma^2 \phi^n/(1-\phi^2)$ \cite[Example~4.4.9 and
   Example~7.1.1]{EmbrechtsKluppelbergMikosch1997}.

   Our framework now allows us to derive the limiting behavior of $M_n$ for any
   sequence $(X_i)$ of the form $X_i = F^{-1}[\Phi\{(Y_i-\mu)/\sigma\}]$ where
   $F$ is a continuous distribution function and $(Y_i)$ is a stationary
   Gaussian sequence with $Y_i \sim \Normal(\mu,\sigma^2)$ that
   satisfies~\eqref{eq:Berman}. Indeed, because the copulas of
   $(X_1,\dots, X_n)$ and $(Y_1,\dots, Y_n)$ are identical,
   $\pD_n^{\rate}(u) \to u$ as $n\to \infty$ with rate $r_n=n$. If
   $F \in \MDA(\EVDist)$,
   Theorem~\ref{theorem_continuous_limit}~\ref{mainThm:1} implies that
   \eqref{eq:central} holds with $\pG=\EVDist$ and the normalizing sequences of
   the associated iid sequence.
 \end{example}

 Example~\ref{ex:3.4} can be generalized to a number of stochastic processes for
 which the extremal index is known to be non-zero. This includes the
 uniform autoregressive process or the GARCH(1,1) process; see
 \cite[p.~142]{mcneilfreyembrechts2015} or \cite[Section~3]{Ferreira2018} for a
 list of suitable examples.

 While the distributional mixing condition rules out long-range dependence in
 the tail, this is not the case for Theorem~\ref{theorem_continuous_limit} and
 Theorem~\ref{theorem-FTG2}.  The final result in this section characterizes sequences
 with asymptotic power diagonals for which the results in
 Section~\ref{section_extremes_and_copula_diagonals} imply that the limit in
 \eqref{eq:central} is generalized extreme value and yet the $\mathcal{D}(u_n)$
 condition is violated.

 \begin{proposition}\label{prop:Dun-violated}
   Let $X_1,X_2, \ldots \sim \pMargin$ be an exchangeable sequence with an
   asymptotic power diagonal with index $\theta > 0$ and rate
   $\rate : \IN \to (0,\infty)$, $\lim_{n\to\infty}\rate(n)=\infty$. Assume that
   there exists a bijection $\lambda:(0,\infty) \to (0,\infty)$ such that
   $\rate_{\lceil nt \rceil}/\rate_n \to \lambda(t)$ for all $t > 0$ and suppose
   that $\lambda$ restricted to $(0,1)$ is not linear, i.e., not of the form
   $\lambda(t) = \alpha t$, $t \in (0,1)$, $\alpha > 0$. Suppose
   further that $\delta_n$ is strictly increasing for all $n \in \IN$ and that
   \eqref{eq:central} holds for some normalizing sequences $(\cn_n)$,
   $\cn_n > 0$, and $(\dn_n)$ and non-degenerate distribution $\pG$. Then
   $\pG \in \GEV$ while $\mathcal{D}(u_n)$ is violated for all thresholds of the
   form $u_n = \cn_n x + \dn_n$ with $x$ such that $\pG(x) \in (0,1)$.
 \end{proposition}

 Prime examples of sequences which meet the conditions of
 Proposition~\ref{prop:Dun-violated} are exchangeable sequences with power
 diagonals, where the power $\eta_n$ grows slower than $n$. For example, we
 can consider the meta-extreme sequence $(X_i)$ from Example~\ref{ex:Mai} with
 continuous margin $\pMargin \in \MDA(\EVDist)$ for some $\EVDist \in \GEV$ and
 the logistic stdf $\ell_{n,\text{Gu}}$ with
 $\theta > 1$. As explained in Section~\ref{sec:3.1}, the Gumbel--Hougaard
 copula $\Copula_n$ has a power diagonal with
 $\eta_n =\ell_{n,\text{Gu}}(1,\dots,1) = n^{1/\theta}$. We then have $r_n = \eta_n$ and $\lambda(t) = t^{1/\theta}$, which is not
 linear. %
 \section{Sequences with non-GEV limits}\label{sec:examples}
 In this section, we study sequences for which the limiting distribution in
 \eqref{eq:central} is no longer generalized extreme value. We first investigate
 meta-Archimax sequences in Section~\ref{section_Archimedean_Archimax} and two
 generalizations in Section~\ref{section_extensions}, one to arbitrary
 exchangeable sequences and one to mixtures of not necessarily exchangeable
 sequences.

 \subsection{Meta-Archimax sequences}\label{section_Archimedean_Archimax}
 In this section, we treat sequences defined as follows.

 \begin{definition}\label{def:4.1}
   A continuous function $\gen : [0,\infty) \to [0,1]$ which satisfies
   $\gen(0) = 1$, $\gen(t) \to 0$ as $t\to \infty$ and which is strictly
   decreasing on $[0, \inf\{x : \gen(x) = 0\}]$ is called an \emph{(Archimedean)
   generator}. A sequence $(X_i)$ of identically distributed random variables is
   called a \emph{meta-Archimax sequence} with Archimedean generator $\gen$ and
   a sequence $(\STDF_n)$ of $n$-variate stdfs, if,
   for each $n \ge 2$, the copula $\Copula_n$ of $(X_1,\dots, X_n)$ is an
   Archimax copula with generator $\gen$ and stdf
   $\STDF_n$, i.e.\
   \begin{align}\label{eq_Archimax_copula}
     \Copula_{n}(u_1,\ldots,u_n) = \gen[\STDF_n\{\invgen(u_1),\ldots,\invgen(u_n)\}],
   \end{align}
   for all $u_1,\dots, u_n \in [0,1]$. If, for all $n \ge 1$ and $x_1,\dots, x_n \in [0,\infty)$, $\ell_n(x_1,\dots, x_n) = x_1 + \ldots + x_n$, the sequence is called \emph{meta-Archimedean}.
 \end{definition}

 We emphasize that the generator of a meta-Archimax sequence does not depend on
 $n$; its inverse $\invgen$ is well-defined on $(0,1]$ and
 $\invgen(0):=\inf\{x : \gen(x) = 0\}$ by
 convention. Table~\ref{table_generator_limit} provides several well-known
 parametric families of generators. For additional examples, see e.g.\
 \cite[Chapter~4.6]{Nelsen2006} and \cite{Hofert2011}; these generators can be
 further transformed to obtain richer classes of models as in
 \cite[Table~2]{Charpentier/Segers:2009} or \cite{hofertscherer2011}. When we
 say that $(X_i)$ is a meta-Archimax sequence with generator $\gen$ and stdfs
 $(\STDF_n)$, we implicitly assume that $\Copula_n$ in
 \eqref{eq_Archimax_copula} is a bona-fide copula for each $n$ and that
 $(\STDF_n)$ satisfies
 $\STDF_{n+1}(x_1,\dots,x_{n},0) = \STDF_n(x_1,\dots, x_n)$ for all $n \in \IN$
 and $x_1,\dots, x_n \in [0,\infty)$, where $\ell_1(x) = x$ by convention. The
 question when a given $\gen$ and $(\STDF_n)$ give rise to a sequence of
 Archimax copulas can only be answered in special cases; we elaborate on this in
 Example~\ref{ex_Archimedean_sequence} and Example~\ref{ex_Archimax_sequence}.

 Archimax and notably Archimedean copulas have been studied extensively, viz.\
 \cite{Caperaa/Fougeres/Genest:2000,CharpentierFougeresGenestNeslehova2014,Chatelain/Fougeres/Neslehova:2020,mcneilneslehova2009,Nelsen2006}. It
 is easily seen that when $\gen(t) = e^{-t}$ for all $t \ge 0$, the
 meta-Archimax sequence with generator $\gen$ reduces to a meta-extreme sequence
 with stdfs $(\STDF_n)$ that we investigated in
 Section~\ref{sec:3.1}. Example~\ref{ex_Archimedean_sequence} and
 Example~\ref{ex_Archimax_sequence} show explicit constructions of meta-Archimedean and
 meta-Archimax sequences for arbitrary generators, relating them to scale
 mixtures of certain iid or meta-extreme sequences.

 \begin{table}
   \centering\small
   \begin{tabular}{l @{\hskip -0.5cm} c c c c}
     \toprule
     \multicolumn{1}{l}{Copula} & \multicolumn{1}{c}{Generator $\gen(t)$} & \multicolumn{1}{c}{Inverse $\invgen(u)$} & \multicolumn{1}{c}{$\rho$} & \multicolumn{1}{c}{$-\gen'(0)$}\\
     \midrule
     Independence & $\exp(-t)$ & $-\log(u)$ & $1$ & $1$\\
     Ali-Mikhail-Haq $(\theta \in (0,1))$ & $\frac{1-\theta}{\exp(t)-\theta}$ & $\log\left( \frac{1-\theta(1-u)}{u}\right)$ & $1$ & $\frac{1}{1-\theta}$\\
     Clayton $(\theta > 0)$ & $(1+ t)^{-1/\theta}$ & $u^{-\theta}-1$ & $1$ & $1/\theta$\\
     Frank $(\theta > 0)$ & $-\frac{1}{\theta}\log\left(1+\exp(-t)\left(\e^{-\theta}-1\right)\right)$ &  $-\log\left(\frac{\exp(-\theta u)-1}{\exp(-\theta)-1}\right)$ &  $1$ & $\frac{\e^{\theta}-1}{\theta}$\\
     Example~\ref{ex:Ballerini} & $1/\left(t(1+1/t)^{1+t}\right)$ & (no closed form) & $1$ & $\infty$\\
     Gumbel-Hougaard $(\theta > 1)$ & $\exp\left(-t^{1/\theta}\right)$ & $(-\log(u))^{\theta}$ & $1/\theta$ & $\infty$\\
     Joe $(\theta > 1)$ & $1 - \left(1-\exp(-t)\right)^{1/\theta}$ & $-\log\left(1-(1-u)^{\theta}\right)$ & $1/\theta$ & $\infty$\\
     \bottomrule
   \end{tabular}
   \vspace{0.4em}
   \caption{Completely monotone generators, their inverses, coefficients of
     regular variation $\rho$ such that $1-\gen(1/x) \in \RV_{-\rho}$ and
     negative right-hand side generator derivatives at $0$ for selected
     Archimedean copulas. Note that $-\gen'(0) < \infty$ implies $\rho=1$ as
     discussed in Example~\ref{Example_Archimax_no_tail_dep}.}
   \label{table_generator_limit}
 \end{table}

 \begin{example}\label{ex_Archimedean_sequence}
   From \cite{Kimberling:1974} it is well-known that $\gen$ is an Archimedean
   generator of a meta-Archimedean sequence if and only if it is completely
   monotone, i.e.\ differentiable on $(0,\infty)$ of all orders with $k$-th
   derivative satisfying $(-1)^k \gen^{(k)}(\cdot) \ge 0$. This result allows us
   to construct an arbitrary meta-Archimedean sequence
   $X_1,X_2,\ldots\sim\pMargin$ explicitly as follows.
   The Bernstein--Widder theorem \cite[p.~439]{feller1971} implies that a
   completely montone $\gen$ must be a Laplace--Stieltjes transform of a
   positive random variable $V$ (also called \emph{frailty}), i.e.\
   $\gen(t) = \E[e^{-tV}]$, $t \ge 0$. Let $(E_i)$ be a sequence of iid
   unit exponential random variables independent of $V$. As observed in
   \cite{Marshall/Olkin:1988}, the survival copula of the multiplicative hazard
   (or frailty) model $(E_1/V,\dots,E_n/V)$ is an Archimedean copula with
   generator $\gen$, viz.\ $\Copula_{n}(u_1,\ldots,u_n) = \gen\{\invgen(u_1) + \dots + \invgen(u_n)\}$, $u_1,\dots, u_n \in (0,1)$. Consequently, the sequence $(Y_i)$ given by $Y_i = V/E_i$ is then meta-Archimedean with
   generator $\gen$ and a continuous univariate margin given by $\gen(1/x)$ for $x > 0$ and
   by $0$ otherwise. To obtain a meta-Archimedean sequence with generator $\gen$ and an
   arbitrary univariate margin $F$, it suffices to set $X_i = F^{-1}\{\gen(E_i/V)\}$ for all $i \in \IN$.
 \end{example}

 \begin{example}\label{ex_Archimax_sequence}
   When $\gen$ is a completely monotone Archimedean generator,
   \cite{CharpentierFougeresGenestNeslehova2014} show that
   \eqref{eq_Archimax_copula} is a bona-fide copula for any $n \ge 2$ and any
   $n$-variate stdf $\STDF_n$. Complete monotonicity
   of $\gen$ is not necessary for certain fixed sequences of $n$-variate stdfs
   \cite{mcneilneslehova2009,CharpentierFougeresGenestNeslehova2014}, but when
   it holds, it again allows us to use the Bernstein--Widder theorem to
   construct a meta-Archimax sequence explicitly, as follows.

   As in Example~\ref{ex_Archimedean_sequence}, let $V$ be a positive random
   variable with Laplace--Stieltjes transform $\gen$. Let also $(Z_i)$ be a
   meta-extreme sequence independent of $V$, with unit Fr\'echet margins and
   stdfs $\STDF_n$, as in
   Definition~\ref{def_meta_extreme_sequence}. Define the sequence $(Y_i)$ via
   $Y_i = VZ_i$ for all $i \ge
   1$. \cite[Remark~3.2]{CharpentierFougeresGenestNeslehova2014} implies that
   $(Y_i)$ is meta-Archimax with generator $\gen$ and stdfs $(\STDF_n)$. Its univariate margin is again given by $\gen(1/x)$
   for $x > 0$ and by $0$ otherwise. To generalize this construction to a
   meta-Archimax sequence with an arbitrary margin $F$, we set
   $X_i = F^{-1}\{\gen(E_i/V)\}$ for all $i \in \IN$, as in
   Example~\ref{ex_Archimedean_sequence}.
 \end{example}

 We now describe the limiting behavior of maxima of meta-Archimax sequences
 using the theory we developed in
 Section~\ref{section_extremes_and_copula_diagonals}. To this end, observe that
 the diagonal of an Archimax copula is of the form
 $\Cdiag_n(u) = \gen\{ \invgen(u) \eta_n\}$, where
 $\eta_n = \STDF_n(1,\dots, 1)$ is Smith's extremal coefficient as in
 Section~\ref{sec:3.1}. Therefore, the asymptotic properties of $\Cdiag_n$ will
 depend on the behavior of $\eta_n$ and of the Archimedan generator, notably its
 regular variation. We recall that a measurable function $f>0$ is regularly
 varying at $\infty$ with index $\rho\in\IR$, in notation $f \in \RV_{\rho}$, if
 it satisfies $\lim_{x\to\infty} f(t x)/f(x) = t^{\rho}$ for all $t > 0$, see
 e.g.\ \cite[Chapter~2]{BinghamGoldieTeugels1987}.

 \begin{theorem}\label{theorem_Archimedean_distortion_limit}
   Let $(X_i)$ be a meta-Archimax sequence with generator $\gen$ and stdfs $(\STDF_n)$.  For each $n \ge 1$, set
   $\STDFone_n = \STDF_n(1,\ldots,1)$ and define the rate
   $\rate : \IN \to (0,\infty)$ by
   $\rate(n) =\rate_n = 1/(1-\gen(1/\STDFone_n))$, $n \ge 1$.
   Furthermore assume:
   \begin{enumerate}[label=(\roman*), labelwidth=\widthof{(iii)}]
   \item\label{thm:A:dist:limit:1}$ \STDFone_n \to \infty$ as $n\to\infty$;
   \item\label{thm:A:dist:limit:2} $1 - \gen(1/\cdot) \in \RV_{-\rho}$ for $\rho \in (0,1]$.
   \end{enumerate}
   Then, for all $u \in [0,1]$, the diagonal power distortion
   $\pD_n^{\rate}(u) = \gen\{\STDFone_n \invgen(u^{1/\rate_n})\}$ converges to
   $\pD(u)=\gen\{(-\log u)^{1/\rho}\}$ as $n \to \infty$. If also the univariate
   marginal distribution $F$ of $(X_i)$ satisfies $F \in \MDA(\EVDist)$ for some
   $\EVDist \in \GEV$ and normalizing sequences $(\an_n)$, $\an_n>0$, and
   $(\bn_n)$, i.e.\ $F^n(\an_n x + \bn_n)\to \EVDist(x)$, $x \in \IR$,
   then \eqref{eq:central} holds with $\pG = \pD \circ \EVDist$ and normalizing
   sequences given by $\cn_n = \an_{\ceil{\rate_n}}$ and
   $\dn_n =\bn_{\ceil{\rate_n}}$.
 \end{theorem}

 The assumptions in Theorem~\ref{theorem_Archimedean_distortion_limit} are not
 restrictive. We already discussed assumption~\ref{thm:A:dist:limit:1} in
 Section~\ref{sec:3}: If $\eta_n$ had a finite limit as in
 Example~\ref{example_Cuadras_Auge}, we would be in the scope of
 Proposition~\ref{cor:r(n)finite} and the possible limits in \eqref{eq:central}
 would be too broad.  Assumption~\ref{thm:A:dist:limit:2} is satisfied by
 nearly all known Archimedean generators, including all listed in
 Table~\ref{table_generator_limit}, see
 \cite{Charpentier/Segers:2009,LarssonNeslehova2011}.

 \begin{remark}
   The limiting distribution $\pG$ obtained in
   Theorem~\ref{theorem_Archimedean_distortion_limit} can also be expressed in a
   different way. To see this, first write $\EVDist = \EVDist_{\EVI,\mu,\sigma}$
   for some $\EVI, \mu \in \IR$ and $\sigma > 0$. Using the fact that for all
   $x \in \IR$,
   $\{-\log\EVDist_{\EVI,\mu,\sigma}(x)\}^{1/\rho}=-\log
   \EVDist_{\rho\EVI,\mu,\rho\sigma}(x)$, we also have, for all
   $x \in \IR$,
   \begin{align}\label{eq:4alternative}
     \pG(x) = \gen \bigl\{ -\log \EVDist_{\rho\EVI,\mu,\rho\sigma}(x) \bigr\}.
   \end{align}
   Viewing $\gen(t^{1/\rho})$, $t\in[0,\infty)$, as an outer power
   transformation of the Archimedean generator $\gen$ when $\rho < 1$, the
   probabilistic interpretation in \cite{Hofert2011} does not apply here,
   because $\gen(t^{1/\rho})$ is a valid Archimedean generator only for
   $\rho \geq 1$ \cite[Theorem~4.5.1]{Nelsen2006}.
 \end{remark}

 \begin{remark}\label{rem:4wuethrich1}
   The alternative expression \eqref{eq:4alternative} also allows us to connect
   Theorem~\ref{theorem_Archimedean_distortion_limit} with the results of
   W\"uthrich \cite{Wuthrich2004}, who investigated maxima of meta-Archimedean
   sequences. The assumptions of Proposition~5.6 in the latter paper are the
   same as in Theorem~\ref{theorem_Archimedean_distortion_limit}, and the limit
   is precisely as in \eqref{eq:4alternative}, although the normalizing
   constants are different. Rather than investigating copula diagonals as done
   here, the approach in \cite{Wuthrich2004} relates $M_n$ to the maximum
   $\widetilde M_n$ of iid observations drawn from the univariate distribution
   function $\exp\{-\psi^{-1} \circ F\}$. While the normalizing sequences used
   in Theorem~\ref{theorem_Archimedean_distortion_limit} are subsequences of
   $(\an_n)$ and $(\bn_n)$ used to stabilize the iid maximum $M_n^*$,
   \cite{Wuthrich2004} uses the normalizing constants of $\widetilde M_n$. Clearly, the two sets of sequences must be related through the convergence to types theorem \cite[Proposition~0.2]{Resnick1987}.
   When $\pF \in \MDA(\EVDist_{\EVI,\mu,\sigma})$ with $\EVI > 0$, it is also
   possible to see this directly. In \cite{Wuthrich2004},
   $\cn_n = F^{-1}[ \gen \{\log n/(n-1)\}]$ and $\dn_n =0$, while the approach
   taken here gives
   $\cn_n = \an_{\ceil{\rate_n}} = F^{-1}\{(\ceil{\rate_n}-1)/\ceil{\rate_n}\}$
   and $\dn_n =\bn_{\ceil{\rate_n}}=0$. Assumption~\ref{thm:A:dist:limit:2} of
   Theorem~\ref{theorem_Archimedean_distortion_limit} implies that
   $1-\gen \{\log n/(n-1)\} \approx 1-\gen(1/n)$ so that if we set $n^*$ to be
   an integer such that $(n^*-1)/n^* \approx \gen \{\log n/(n-1)\}$, we obtain
   that $n^* \approx \ceil{1/\{1-\gen(1/n)\}}$. As $\eta_n = n$ for a meta-Archimedean sequence,
   $n^*\approx\ceil{\rate_n}$.
 \end{remark}

  As expected from Theorem~\ref{theorem_continuous_limit}, the distortion
  function $\pD$ identified in
  Theorem~\ref{theorem_Archimedean_distortion_limit} is a proper distribution
  function on $[0,1]$.  We formally state this result in the following
  corollary, while Figure~\ref{figure_density_shapes_D_Archimedean} illustrates
  the density of $\pD$ for various choices of $\gen$. As can be seen, the
  possible density shapes range from being constant to strictly increasing,
  unimodal or unbounded and thus cover a wide range of possible distortions.
  \begin{figure}
    \centering
    \includegraphics[width=\textwidth]{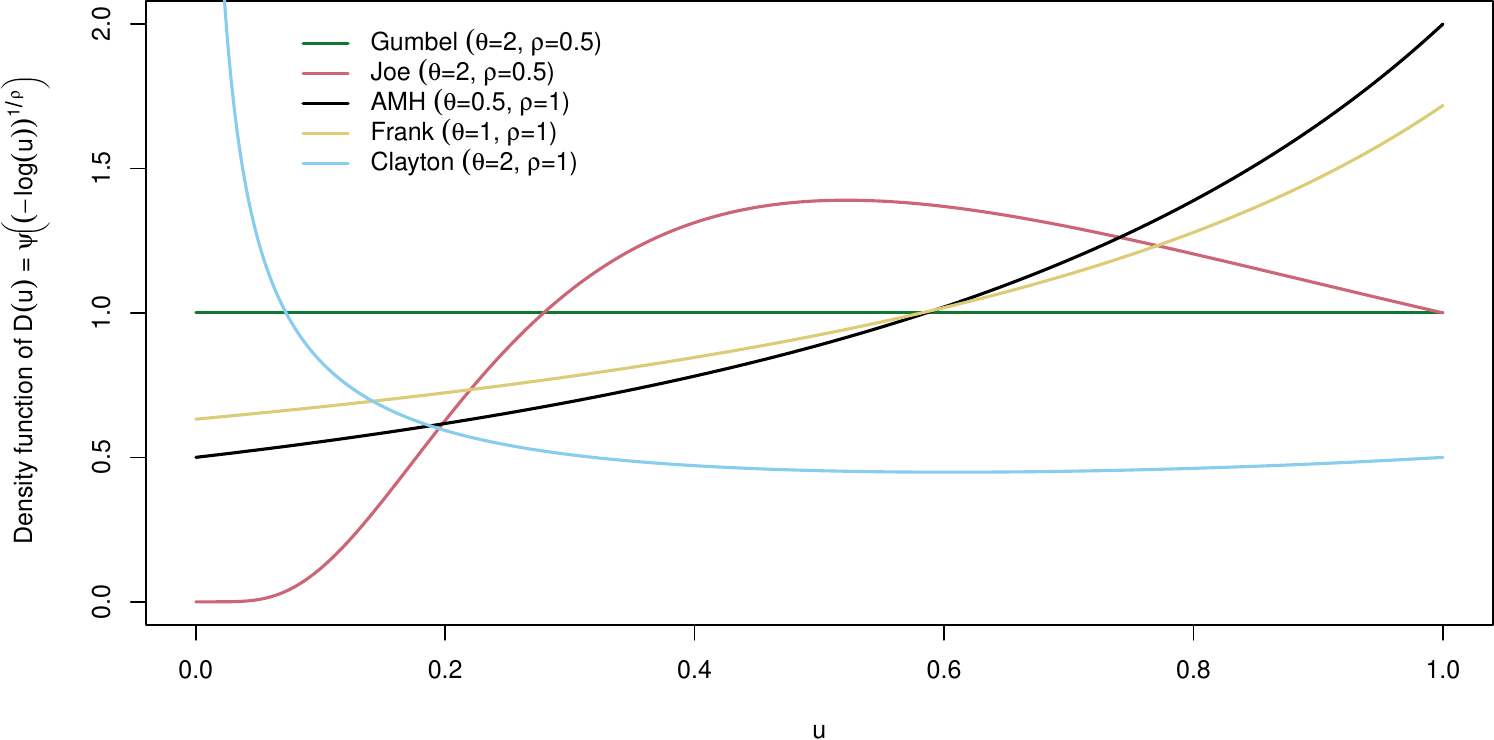}%
    \caption{Densities of $\pD(u) = \gen\{(-\log u)^{1/\rho}\}$ for different
      Archimedean generators from Table~\ref{table_generator_limit}.}
    \label{figure_density_shapes_D_Archimedean}
  \end{figure}
  \begin{corollary}\label{corollary_density_Archimedean_distortion}
    The function $\pD(u) = \gen\{(-\log u)^{1/\rho}\}$ is a distribution
    function on $[0,1]$, with corresponding quantile function
    $\qD(u)=\e^{-\{\invgen(u)\}^{\rho}}$, $u\in[0,1]$, and density
    \begin{align*}
      \dD(u) = \frac{-\gen'\{(-\log u)^{1/\rho}\} (-\log u)^{(1-\rho)/\rho}}{\rho u},\quad u\in(0,1).
    \end{align*}
    The values of $\dD$ at the boundary are given by
    $\dD(0) = \lim_{y \to \infty} -\gen'(y)y^{1-\rho}\e^{(y^{\rho})}/\rho$ and by
    $\dD(1) = -\gen'(0)$ if $\rho=1$ and $\dD(1) = \lim_{y \to 0} -\gen'(y)y^{1-\rho}/\rho$ if $\rho \in (0,1)$.
  \end{corollary}

We now draw conclusions from Theorem~\ref{theorem_Archimedean_distortion_limit}
and illustrate that the behavior of $\eta_n$ controls the speed of convergence
while the generator $\gen$ determines the functional form of the limit. We focus
on archetypal meta-Archimax
sequences %
with generators that induce different degrees of dependence in the upper tail of
the associated Archimedean copula.

\begin{example}\label{ex:4Independence}
  By \cite[Lemma~2.4]{Chatelain/Fougeres/Neslehova:2020},
  if the generator $\gen$ satisfies assumption~\ref{thm:A:dist:limit:2} of
  Theorem~\ref{theorem_Archimedean_distortion_limit}, the meta-Archimax sequence
  $(X_i)$ with generator $\gen$ and stdfs $(\STDF_n)$
  is meta-extreme if and only
  if %
  $\gen$ is the Gumbel--Hougaard generator $\gen(t) = \exp(-t^{1/\theta})$ with
  parameter $\theta \ge 1$; the constant $c$ in
  \cite[Lemma~2.4]{Chatelain/Fougeres/Neslehova:2020} can be set to $c=1$
  without loss of generality since the associated Archimax copula remains the
  same.  When $\theta=1$, $\gen$ is the independence generator
  $\gen(t) = e^{-t}$ and $(X_i)$ is a meta-extreme sequence with the same stdfs
  $(\STDF_n)$. When $\theta > 1$, expressing $(X_i)$
  as a meta-extreme sequence requires changing the stdfs to $(\STDF_n^{\star})$, where for all $n \ge 2$ and
  $x_1,\dots, x_n \in [0,\infty)$,
  $\STDF_n^{\star}(x_1,\dots, x_n) = \STDF_n^{1/\theta}(x_1^\theta,\dots,
  x_n^\theta)$. When $(X_i)$ is meta-Archimedean,
  $\ell_n(x_1,\dots, x_n) = x_1 + \dots + x_n$ so that $\ell_n^{\star}$ is the
  logistic stdf $\ell_{n,\text{Gu}}$ seen at the end
  of Example~\ref{ex:Mai}.

  Assuming that $\eta_n = \STDF_n(1,\dots, 1) \to \infty$ as $n\to \infty$, we
  can compare Theorem~\ref{theorem_Archimedean_distortion_limit} to the results
  in Section~\ref{sec:3}. To this end, suppose that the univariate margin $\pF$
  of $(X_i)$ satisfies $\pF \in \MDA(\EVDist_{\EVI,\mu,\sigma})$ and note that
  for $\theta \ge 1$, $\gen$ satisfies assumption~\ref{thm:A:dist:limit:2} of
  Theorem~\ref{theorem_Archimedean_distortion_limit} with $\rho=1/\theta$, viz.\
  Table~\ref{table_generator_limit}. Set
  $\rate^{\star}_n = \eta_n^{\star} = \eta_n^{1/\theta}$ with
  $\eta_n^{\star} = \STDF_n^{\star}(1,\dots, 1)$ so that
  $\lim_{n\to \infty} \rate_n^{\star} = \infty$. From
  $\lim_{t\to \infty}t(1-e^{-1/t}) = 1$, the rate $\rate_n$ from
  Theorem~\ref{theorem_Archimedean_distortion_limit} linked to $\eta_n^{\star}$
  satisfies
  $\lim_{n\to \infty}\rate_n/\rate_n^{\star}
  =\lim_{n\to \infty} [\eta_n^{\star}\{1-\gen(1/\eta_n^{\star})\}]^{-1}
  =\lim_{n\to \infty} \{\eta_n^{1/\theta} (1- e^{-1/\eta_n^{1/\theta}})\}^{-1} = 1$.
Remark~\ref{rem_rate_uniqueness} shows that the rate in
Theorem~\ref{theorem_Archimedean_distortion_limit} can thus be replaced by
$\rate_n^{\star} = \eta_n^{\star}$ used in Section~\ref{sec:3}, leading to the
same limit as in \eqref{eq:central}. The latter is
$\gen\{(-\log \EVDist_{\EVI,\mu,\sigma})^{\theta}\} = \exp[-\{(-\log
\EVDist_{\EVI,\mu,\sigma})^{\theta}\}^{1/\theta}]= \EVDist_{\EVI,\mu,\sigma}$,
which is indeed what we obtained in Section~\ref{sec:3.1}.
\end{example}

\begin{example}\label{Example_Archimax_no_tail_dep}
Let $\gen$ be an Archimedean generator with $-\gen^\prime(0) \in (0,\infty)$; see Table~\ref{table_generator_limit} for
examples. From \cite[Section~4.3]{Charpentier/Segers:2009},
the Archimedean copula generated by $\gen$ is upper tail independent and assumption~\ref{thm:A:dist:limit:2} of Theorem~\ref{theorem_Archimedean_distortion_limit} holds with $\rho=1$; see also \cite[Proposition~1]{LarssonNeslehova2011}. When $\gen$ is also completely monotone, the frailty $V$ in Example~\ref{ex_Archimedean_sequence} has a finite mean.

Let $(X_i)$ be a meta-Archimax sequence with generator $\gen$ such that
$-\gen^\prime(0) \in (0,\infty)$ and a sequence $(\STDF_n)$ of stdfs with
$\eta_n = \STDF_n(1,\dots, 1) \to \infty$ as $n\to \infty$. Suppose that the
univariate margin $\pF$ of $(X_i)$ satisfies
$\pF \in \MDA(\EVDist_{\EVI,\mu,\sigma})$.  If we set $\tilde \rate_n = \eta_n$
for all $n \in \IN$, we obtain via the formulation of $\rate_n$ provided in
Theorem~\ref{theorem_Archimedean_distortion_limit} that
$\lim_{n\to \infty} \rate_n/\tilde \rate_n = \lim_{n\to
  \infty}[\eta_n\{1-\gen(1/\eta_n)\}]^{-1} = \lim_{n\to
  \infty}[\eta_n\{\gen(0)-\gen(1/\eta_n)\}]^{-1} = -1/\gen^\prime(0)$.
Remark~\ref{rem_rate_uniqueness} thus implies that the rate $\tilde \rate_n$
could have been used instead of $\rate_n$ and that \eqref{eq:central} would then
hold with $(\bn_{\ceil{\STDFone_n}})$ and $(\an_{\ceil{\STDFone_n}})$ and a
limiting distribution of the form
$G(x) = \smash{\gen\{-\log \EVDist^{-1/\gen'(0)}_{\EVI,\mu,\sigma}(x)\}} =\gen(
- \log \EVDist_{\EVI,\tilde \mu, \tilde \sigma})$, $x \in \IR$, with altered
location $\tilde \mu$ and scale $\tilde \sigma$ for which the precises formulas
are provided in Remark~\ref{rem_rate_uniqueness}; as usual, $(\bn_n)$ and
$(\an_n)$ are the normalizing sequences corresponding to the iid sequence with
margin $\pMargin$.  Interestingly, when
$\STDF_n(x_1,\dots, x_n) = x_1 + \dots + x_n$ for all $n \ge 2$ so that $(X_i)$
is meta-Archimedean, $\eta_n = n$ and $(\bn_n)$ and $(\an_n)$ can be used to
stabilize $M_n$.
\end{example}

\begin{example}\label{Example_Archimax_tail_dep}
  Let $\psi$ be an Archimedean generator that satisfies
  assumption~\ref{thm:A:dist:limit:2} of
  Theorem~\ref{theorem_Archimedean_distortion_limit} and is such that
  $-\gen^\prime(0) = \infty$. In \cite{Charpentier/Segers:2009}, this case is
  referred to as \emph{asymptotic dependence in the upper tail} when $\rho < 1$
  and \emph{near asymptotic dependence} when $\rho=1$. Examples are again given
  in Table~\ref{table_generator_limit}.
  As in Example~\ref{Example_Archimax_no_tail_dep}, let $(X_i)$ be a meta-Archimax
  sequence with such generator $\gen$, a sequence $(\STDF_n)$ of stdfs with $\eta_n = \STDF_n(1,\dots, 1) \to \infty$ as
  $n\to \infty$, and univariate margin
  $\pMargin \in \MDA(\EVDist_{\EVI,\mu,\sigma})$. Let also $(\bn_n)$ and
  $(\an_n)$ be the normalizing sequences corresponding to the iid sequence with
  margin $\pMargin$.

  If we set $\tilde r_n = \eta_n$ for all $n \in \IN$ as in
  Example~\ref{Example_Archimax_no_tail_dep}, we can easily see that
  $\rate_n/\tilde \rate_n \to 0$, where $\rate_n$ is as in
  Theorem~\ref{theorem_Archimedean_distortion_limit}. Because any Archimedean
  generator is such that $\gen(t) < 1$ for all $t > 0$,
  Remark~\ref{rem_rate_uniqueness} and
  Theorem~\ref{theorem_Archimedean_distortion_limit} imply that the choice
  $\cn_n=\an_{\ceil{\STDFone_n}}$ and $\dn_n=\bn_{\ceil{\STDFone_n}}$ leads to a
  degenerate limit in \eqref{eq:central}. In the special case of a
  meta-Archimedean sequence, this means that the simplified choice of
  normalizing constants $(\an_n)$ and $(\bn_n)$ from the iid case cannot be used
  to stabilize $M_n$.  Theorem~\ref{theorem_Archimedean_distortion_limit}
  however still applies.
\end{example}

\begin{example}\label{ex:Ballerini}
  \cite{Ballerini1994b} explored maxima of meta-Archimedean sequences whose
  generator satisfies the so-called polynomial growth condition, viz.\
\begin{align}\label{eq:polygrowth}
\lim_{t \to \infty} t^{\rho}\{1-\gen(1/t)\} = c
\end{align}
for some $\rho \in (0,1]$ and $c \in (0,\infty)$. It is easily seen that $\gen$
then satisfies assumption~\ref{thm:A:dist:limit:2} of
Theorem~\ref{theorem_Archimedean_distortion_limit} with the same $\rho$,
because $\{1-\gen(1/t)\} = t^{-\rho} L(t)$ for $L(t) \to \theta$ as
$t \to \infty$. The latter property immediately renders $L$ slowly varying, but
also shows that the polynomial growth condition is more restrictive than
assumption~\ref{thm:A:dist:limit:2}, because slowly varying functions need not
tend to a positive constant at $\infty$. This is also apparent from
Example~\ref{Example_Archimax_no_tail_dep} and
Example~\ref{Example_Archimax_tail_dep}; in the special case $\rho=1$, the
polynomial growth condition implies that $-\gen^\prime(0) = 1/c \in
(0,\infty)$. One counterexample of an Archimedean generator for which
assumption~\ref{thm:A:dist:limit:2} holds with $\rho=1$ and yet
$-\gen^\prime(0) = \infty$ is provided by Family~23 of
\cite[Table~1]{Charpentier/Segers:2009}, another is given in
Table~\ref{table_generator_limit}.
We provide justification for the latter
  in the Supplementary Material \cite{HerrmannHofertNeslehova2024} and also
  explain therein that the family in \cite[Example~3]{Ballerini1994b} is
  actually not a counterexample.

Let $(X_i)$ be a meta-Archimax sequence with a generator that satisfies
\eqref{eq:polygrowth}, a sequence $(\STDF_n)$ of stdfs such that $\eta_n = \STDF_n(1,\dots, 1) \to \infty$ as $n\to \infty$,
and univariate margin $\pMargin \in \MDA(\EVDist_{\EVI,\mu,\sigma})$. Now set
$\tilde \rate_n = \eta_n^\rho$, $n \in \IN$. Because of
\eqref{eq:polygrowth}, $\rate_n$ in
Theorem~\ref{theorem_Archimedean_distortion_limit} satisfies
$\rate_n /\tilde\rate_n \to 1/c \in
(0,\infty)$. Remark~\ref{rem_rate_uniqueness} thus implies that $M_n$ could have
been alternatively normalized using $\cn_n=\an_{\ceil{\STDFone_n^\rho}}$ and
$\dn_n=\bn_{\ceil{\STDFone_n^\rho}}$, leading to the limit
$\pG(x) = \smash{\pD \{\EVDist_{\EVI,\mu,\sigma}^{1/c}(x)\}} = \gen[ c^{-\rho} \{ - \log \EVDist_{\EVI,\mu,\sigma}(x)\}^\rho]$, $x \in \IR$,
in \eqref{eq:central}. In the special case of a meta-Archimedean sequence, these
are precisely the normalizing constants and the limit derived in
\cite[Theorem~2]{Ballerini1994b}.
\end{example}
For meta-Archimax sequences it is possible to verify the conditions of
Theorem~\ref{theorem-FTG2} provided that the stdf
associated with the extreme value part is sufficiently regular. As shown next,
this guarantees the uniqueness of the limiting distributions in the sense of
Theorem~\ref{theorem-FTG2}.

\begin{proposition}\label{prop_unique_archimax_limit}
Let $(X_i)$ be a meta-Archimax sequence with stdfs $(\STDF_n)$ and generator $\gen$ which is strict, i.e.\ $\psi^{-1}(0) = \infty$.
Set $\STDFone_n = \STDF_n(1,\ldots,1)$ and define the rate $\rate : \IN \to (0,\infty)$ by
$\rate(n) =\rate_n = 1/(1-\gen(1/\STDFone_n))$, $n \ge 1$. Suppose that assumptions~\ref{thm:A:dist:limit:1} and \ref{thm:A:dist:limit:2} of Theorem~\ref{theorem_Archimedean_distortion_limit} hold and that
$\lim_{n\to\infty} \STDFone_{\ceil{nt}}/\STDFone_n = \kappa(t)$ for any $t>0$, where
$\kappa \colon (0,\infty) \to (0,\infty)$ is a bijection. Then the rate function
$\rate$ satisfies
$\lim_{n\to\infty} \rate_{\ceil{nt}}/\rate_n = (\kappa(t))^{\rho}$, $t>0$.
Furthermore, whenever $(M_n - \dn_n)/\cn_n$ converges
to a non-degenerate limit $\pG$, it holds that
$\pG = \pD\circ\EVDist$, where $\pD(u) = \gen\{(-\log u)^{1/\rho}\}$
and $\EVDist\in\GEV$.
\end{proposition}

An example of stdfs that satisfy
$\lim_{n\to\infty} \STDFone_{\ceil{nt}}/\STDFone_n = \kappa(t)$ for all $t > 0$
and some bijection $\kappa : (0,\infty) \to (0,\infty)$ is the logistic family
$(\ell_{n,\text{Gu}})$ from
Example~\ref{ex:Mai}. %

\begin{remark}\label{rem:4uniquenessAC}
  The assumptions of Proposition~\ref{prop_unique_archimax_limit} always hold
  when the sequence is meta-Archimedean, provided that $\gen$ satisfies
  assumption~\ref{thm:A:dist:limit:2} of
  Theorem~\ref{theorem_Archimedean_distortion_limit}. Indeed, $\gen$ is
  completely monotone as explained in Example~\ref{ex_Archimedean_sequence} and
  hence strict. Moreover, $\eta_n = n$, so that $\ceil{nt}/n \to t$ for any
  $t > 0$.  Hence, the limit is always of the form
  $\gen\{(-\log \EVDist_{\EVI,\mu,\sigma})^{1/\rho}\} = \gen(-\log
  \EVDist_{\rho\EVI,\mu,\rho\sigma})$ in view of \eqref{eq:4alternative}. A
  similar result has been derived in \cite{Wuthrich2004} albeit under stronger
  assumptions: \cite[Theorem~3.2]{Wuthrich2004} shows that if the limit in
  \eqref{eq:central} is of the form $\gen(-\log \EVDist)$ with a non-degenerate
  $\EVDist$, $\EVDist$ must be generalized extreme value.
\end{remark}

We close this section with a discussion about when the limiting distribution of
normalized maxima of a meta-Archimax sequence is in fact generalized extreme
value. From Example~\ref{ex:4Independence}, we already know that this is the
case when $\gen$ is the Gumbel--Hougaard generator. The
following result shows that this is the only possibility. To avoid assuming
\ref{thm:A:dist:limit:2} of
Theorem~\ref{theorem_Archimedean_distortion_limit}, we use the alternative
expression \eqref{eq:4alternative} of the limiting distribution in
\eqref{eq:central}.

\begin{lemma}\label{Lemma_Archimedean_GEV_Limit}
  Suppose that $\gen$ is an Archimedean generator. Then
  $\pG =\gen(-\log \EVDist)$ is generalized extreme value for all
  $\EVDist\in \GEV$ if and only if there exists $c > 0$ and $\theta \ge 1$ such
  that $\gen(t) = e^{-(ct)^{1/\theta}}$ for all $t \ge 0$.
\end{lemma}

An interesting consequence of Lemma~\ref{Lemma_Archimedean_GEV_Limit} is the
following observation. Suppose that $(X_i)$ is a stationary meta-Archimax
sequence so that \eqref{eq:central} holds with $\pG = \gen\{-\log \EVDist\}$ for
some $\EVDist \in \GEV$. At the same time, suppose that $(X_i)$ fulfills the
distributional mixing assumption $\mathcal{D}(u_n)$ from
Definition~\ref{def:3.3} for any threshold $u_n = \cn_n x + \dn_n$ with
$x \in \IR$ and $(\cn_n)$ and $(\dn_n)$ from \eqref{eq:central}. From
\cite[Theorem~3.3.3]{LeadbetterLindgrenRootzen1983}, we then necessarily have
that $\pG \in \GEV$, and from Lemma~\ref{Lemma_Archimedean_GEV_Limit} we know
that this only happens when $\gen$ is the Gumbel--Hougaard generator with
parameter $\theta \ge 1$. As argued in Example~\ref{ex:4Independence}, $(X_i)$
is then meta-extreme.  In the special case when $(X_i)$ is meta-Archimedean and
$\gen$ is regularly varying at $0$, $(X_i)$ is necessarily stationary and any
non-degenerate limit in \eqref{eq:central} is of the form
$\pG = \gen\{-\log \EVDist\}$ for some $\EVDist \in \GEV$, as seen in
Remark~\ref{rem:4uniquenessAC}. Moreover, Example~\ref{ex:4Independence} shows
that $(X_i)$ with the Gumbel--Hougaard generator is a meta-extreme sequence with
the logistic stdf $\ell_{n,\text{Gu}}$. As discussed
at the end of Section~\ref{sec:extremes:for:time:series},
Proposition~\ref{prop:Dun-violated} implies that $\theta =1$, so that $(X_i)$ is
in fact iid. As a consequence, there are many examples of stationary series
$(X_i)$ that violate the $\mathcal{D}(u_n)$ condition, but still can be analyzed
in our framework.  We summarize this as follows.
\begin{corollary}\label{theorem_contradiction_Du_n}
  Suppose that $(X_i)$ is a meta-Archimedean sequence with generator $\gen$ that
  fulfills assumption~\ref{thm:A:dist:limit:2} of
  Theorem~\ref{theorem_Archimedean_distortion_limit}, and assume that
  \eqref{eq:central} holds for some non-degenerate $G$. Then $\mathcal{D}(u_n)$
  from Definition~\ref{def:3.3} holds for all thresholds $u_n = \cn_n x + \dn_n$
  with $x \in \IR$ and $(\cn_n)$ and $(\dn_n)$ from \eqref{eq:central} if and
  only if $(X_i)$ is an iid sequence.
\end{corollary}

\subsection{Extensions}\label{section_extensions} %
We now discuss two generalizations of the results for meta-Archimax sequences
from Section~\ref{section_Archimedean_Archimax}: The first are arbitrary
exchangeable sequences and the second are mixtures of not necessarily
exchangeable sequences.

All meta-Archimedean sequences are exchangeable, as are meta-Archimax sequences
whose stdfs have the symmetry property that
$\STDF_n(x_1,\dots, x_n) = \STDF_n(x_{\pi(1)},\dots, x_{\pi(n)})$ for each
$n \in \IN$ and each permutation $\pi$ on $\{1,\dots, n\}$.  Maxima of
exchangeable sequences have first been investigated in \cite{Berman1962b}. From
de Finetti's theorem, there exists a real-valued random variable $W$ with the
property that $X_1,X_2,\ldots$ are conditionally independent and identically
distributed given $W$. If $\pG_W$ denotes the conditional distribution function
of $X_i$ given $W$, we obtain that for each $n \in \IN$,
$H_n(x_1,\dots, x_n) = \Prob(X_1 \le x_1, \dots, X_n \le x_n) = \E[\prod_{i=1}^n
\pG_W(x_i)]$.  The results in \cite{Berman1962b} can be broadly summarized as
follows. Suppose $\widetilde \pF$ is such that
$\widetilde \pF \in \MDA(\EVDist)$ with normalizing constants $(\cn_n)$,
$\cn_n>0$ and $(\dn_n)$, where $\EVDist$ is either the Weibull, Fr\'echet or
Gumbel distribution. Then \eqref{eq:central} holds with these normalizing
constants if and only if $\log \pG_W(x) / \log \widetilde \pF(x)$ converges in
distribution to some non-degenerate distribution function $A$ concentrated on
$[0,\infty)$ as $x \uparrow x_{\widetilde \pF}$, where
$x_{\widetilde \pF} = \sup\{ x : \widetilde \pF(x) < 1\}$.  The limit in
\eqref{eq:central} is then of the form $\mathcal{LS}_A(-\log \EVDist)$, where
$\mathcal{LS}_A$ is the Laplace--Stieltjes transform of $A$.

\begin{example}
  When $(X_i)$ is meta-Archimedean, $W$ is the frailty $V$ from
  Example~\ref{ex_Archimedean_sequence} and
  $\pG_W(x) = \e^{-W \gen^{-1}\{ F(x)\}}$. The approach of \cite{Berman1962b}
  can thus be adopted with $\widetilde \pF = \e^{-\gen^{-1} \circ \pF}$; $A$ is
  then the distribution function of $V$ and $\mathcal{LS}_A = \gen$. Proceeding
  this way recovers the results in \cite{Wuthrich2004} which we already related
  to ours in Remark~\ref{rem:4wuethrich1}. As we explained in
  Remark~\ref{rem:4uniquenessAC}, our theory allows us to show that
  $\gen(-\log \EVDist)$, $\EVDist \in \GEV$, are the only possible limits in
  \eqref{eq:central}.
\end{example}

The difficulty with the approach in \cite{Berman1962b} is that $\widetilde \pF$
is left unspecified, and it may not be clear whether it exists and how to choose
it. To instead use the results derived here, we can rely on
\cite{DuranteMai2010} according to which the copula $\Copula_n$ of the first $n$ elements
of an exchangeable sequence $(X_i)$ with continuous margin $\pMargin$ is given,
for all $u_1,\dots, u_n \in (0,1)$, by $\Copula_n(u_1,\dots, u_n) = \int_0^1 \prod_{i=1}^n\frac{\partial B(u_i,t)}{\partial t} \d t$,
where $B$ is the copula of $(X_1,W)$. This way,
$\pD^{\rate}_n(u) = \int_0^1 \{\frac{\partial}{\partial t}
B(u^{1/\rate_n},t)\}^n \d t$. Given that
$|\frac{\partial}{\partial t} B(v,t)| \le 1$ for all $v\in (0,1)$ and almost all
$t \in (0,1)$ \cite[Theorem~2.2.7]{Nelsen2006}, the dominated convergence
theorem allows us to conclude that $\pD^{\rate}_n \to \pD$ pointwise on $[0,1]$
with $\pD(u)=\int_0^1 b(u,t)\d t$, provided that
$\lim_{n\to \infty} \{ \frac{\partial}{\partial t}
B(u^{1/\rate_n},t)\}^n=b(u,t)$. Because $\pD^\rate_n(0)=0$ and
$\pD^\rate_n(1)=1$, we also have $\pD(0)=0$ and $\pD(1)=1$. The limiting
behavior of $M_n$ can then be deduced from
Theorem~\ref{theorem_continuous_limit}, if its conditions hold.

\begin{example}
  For a meta-Archimedean sequence with generator $\gen$ which is a Laplace
  transform of a frailty $V$ with distribution function $F_V$,
  $B(u,t) = \int_0^{\pF_V^{-1}(t)} \e^{-v \gen^{-1}(u)} \d \pF_V(v)$ and
  $\frac{\partial}{\partial t}B(u^{1/\rate_n},t) = \e^{-\pF^{-1}_V(t) \{
    \gen^{-1}(u^{1/\rate_n})\}}$.
When $1 - \gen(1/\cdot) \in \RV_{-\rho}$, we showed in the proof of
Theorem~\ref{theorem_Archimedean_distortion_limit} that
$n \gen^{-1}(u^{1/\rate_n}) \to (-\log u)^{1/\rho}$, so that the limiting
distortion is given by $\pD(u) = \E[\exp\{-(-\log u)^{1/\rho} F_V^{-1}(T)\}]$
where $T\sim\U(0,1)$. This expression readily simplifies to
$\pD(u) = \E[\e^{-V(-\log u)^{1/\rho}}] = \gen\{(-\log u)^{1/\rho}\}$, which is
precisely what we obtained in
Theorem~\ref{theorem_Archimedean_distortion_limit}.
\end{example}

\begin{example}
  From \cite{DuranteMai2010}, exchangeable sequences can also be constructed by
  choosing an arbitrary $B$, $F$, and a distribution of $W$.  If we set $B$ to
  be the Eyraud--Farlie--Gumbel--Morgenstern copula
  $B_{\theta}(u,t) = ut\{1+\theta(1-u)(1-t)\}$ \cite[Example~3.9]{Nelsen2006},
  then $\frac{\partial}{\partial t} B_{\theta}(u,t) = u + \theta u(u-1)(2t-1)$,
  leading to $b_\theta(u,t) = u^{\theta(2t-1)+1}$ upon setting $\rate_n =
  n$. For each $u \in (0,1)$, %
  $\pD_{\theta}(u) = \int_0^1 u^{\theta(2t-1)+1} \d t =
  \frac{u^{1+\theta}-u^{1-\theta}}{2\theta \log u}$.  If
  $\pMargin\in\MDA(\EVDist_{\mu,\sigma,\EVI})$,
  Theorem~\ref{theorem_continuous_limit}~\ref{mainThm:1} shows that
  \eqref{eq:central} holds with the non-degenerate limit
  $\pD_{\theta}\circ\EVDist_{\mu,\sigma,\EVI}$. The normalizing constants are
  those that stabilize the maximum of an iid sequence with margin
  $\pMargin$. The latter fact allows us to use the results in \cite{Berman1962b}
  and conclude that there must exist some non-degenerate distribution on
  $[0,\infty)$ with
  Laplace--Stieltjes transform $\pD_\theta \circ e^{-t}$. Indeed, we recognize
  it to be the uniform distribution on $[1-\abs{\theta}, 1+\abs{\theta}]$.
  From Corollary~\ref{corollary_FTG} we obtain that if \eqref{eq:central} holds
  with a non-degenerate limit, the latter must be of the form
  $\pD_{\theta}\circ\EVDist$, where $\EVDist$ is of the same type as
  $\EVDist_{\mu,\sigma,\EVI}$. Interestingly, $\pD_{\theta} = \pD_{-\theta}$,
  even though $B_{\theta}$ and $B_{-\theta}$ embody very different types of
  dependence.
\end{example}
The second extension is to mixtures of not necessarily exchangeable sequences.

\begin{example}
  Consider $X_1,X_2,\ldots \sim \pMargin$ with the property that for each
  $d \in \IN$ and distinct $t_1,\dots, t_d \in \IN$, the copula
  $\Copula_{t_1,\dots, t_d ; \theta}$ of $(X_{t_1},\dots, X_{t_d})$ depends on a
  parameter $\theta\in \Theta \subseteq\IR^k$. Suppose that the mapping
  $\theta\mapsto\Copula_{t_1,\dots, t_d ; \theta}(u_1,\dots,u_d)$ is
  Borel-measurable for every $u_1,\dots, u_d \in [0,1]$. These conditions are
  fulfilled for several sequences constructed in this paper, notably all
  meta-Archimedean sequences with parametric generators. Let $T$ be some
  distribution function on $\Theta$. From \cite[Corollary~1.4.7,
  p.~17]{DuranteSempi2016}, $\widetilde \Copula_{t_1,\dots, t_d}$ given by
  $\widetilde \Copula_{t_1,\dots, t_d} (u_1,\dots, u_d) = \int_{\Theta}
  \Copula_{t_1,\dots, t_d ; \theta}(u_1,\dots,u_d) \d T(\theta)$ is a copula.
  By the Daniell--Kolmogorov
  theorem, %
  there exists a sequence $(U_i)$ of standard uniform variables such that for
  each $n \in \IN$ and distinct $t_1,\dots, t_n \in \IN$,
  $(U_{t_1},\dots, U_{t_n}) \sim \widetilde \Copula_{t_1,\dots, t_n}$. A mixture
  sequence $(\widetilde X_i)$ with margin $\pMargin$ is then obtained by setting
  $\widetilde X_i = F^{-1}(U_i)$.

  To simplify the notation, write
  $\Copula_{n;\theta}=\Copula_{1,\dots, n ; \theta}$ and
  $\widetilde \Copula_n = \widetilde \Copula_{1,\dots, n}$, and let
  $\Cdiag_{n;\theta}$ and $\widetilde \Cdiag_n$ denote the diagonal of
  $\Copula_{n;\theta}$ and $\widetilde \Copula_n$. Suppose that there exists a
  rate function $\rate$, independent of $\theta$, such that for each
  $\theta \in \Theta$, the diagonal power distortion $\pD_{n;\theta}^{\rate}$ of
  $\Copula_{n;\theta}$ converges pointwise to a continuous map $\pD_\theta$. Then the
  dominated convergence theorem implies that for all $u \in [0,1]$,
  $\lim_{n\to\infty} \widetilde \pD_{n}^\rate (u) = \lim_{n\to \infty}\widetilde
  \Cdiag_n \{u^{1/\rate_n}\} = \lim_{n\to \infty} \int_{\Theta}
  \pD_{n;\theta}^{\rate}(u) \d T(\theta) = \int_{\Theta} \pD_{\theta}(u) \d
  T(\theta) = \widetilde D(u)$, i.e.\ the limiting distortion is a mixture with
  the same mixing distribution $T$. Moreover, if for some
  $\EVDist_{\EVI,\mu,\sigma} \in \GEV$ and sequences $(\an_n)$, $\an_n > 0$, and
  $(\bn_n)$, $\pF^n(\an_n x + \bn_n) \to \EVDist_{\EVI,\mu,\sigma}(x)$ for all
  $x \in \IR$, Theorem~\ref{theorem_continuous_limit}~\ref{mainThm:1} shows
  that \eqref{eq:central} holds with
  $\pG = \widetilde \pD \circ \EVDist_{\EVI,\mu,\sigma}$ and normalizing
  sequences given, for all $n \in \IN$, by $\cn_n = \an_{\lceil r_n \rceil}$ and
  $\dn_n = \bn_{\lceil r_n \rceil}$.  The same sequences can be used to
  stabilize the maximum of $(X_i)$ for each $\theta \in \Theta$.

  As a concrete example, let $(X_i)$ be the meta-Archimedean sequence with some
  margin $\pMargin \in \MDA(\EVDist)$ and the Ali--Mikhail--Haq generator
  $\gen_{\theta}$ with parameter $\theta \in (0,1)$, viz.\
  Table~\ref{table_generator_limit}. Because $-\psi^\prime_\theta(0) < \infty$,
  Example~\ref{Example_Archimax_no_tail_dep} shows that we can choose
  $\rate_n =n$; the limiting diagonal power distortion is then given, for all
  $u \in [0,1]$, by $\pD_\theta(u) = \gen_\theta(- \log u)$. Various mixtures
  can now be obtained by choosing a distribution $T$ on $(0,1)$. For example, if
  $T$ is uniform,
  $\widetilde \pD(u) = \int_0^1 \frac{1-\theta}{1/u - \theta} \d \theta = 1-\frac{u-1}{u} \log (1-u)$.
\end{example}

\section{Uniform convergence rates of maxima under dependence}\label{sec:unif:conv:rate}

We now discuss uniform convergence rates of the distribution of normalized
maxima to their weak limit, akin to Berry--Esseen type theorems for
sums. Numerous such results exist for maxima of iid sequences, a review is
provided in \cite[Section~2.8]{LeadbetterRootzen1988}, starting with the early
results for normal \cite{Hall1979,Nair1981} and exponential
\cite{HallWellner1979} random variables. Convergence rates for various classes
of distributions have been investigated since, see e.g.\
\cite{PengNadarajahLin2010}; also noteworthy are general results for the
Fr\'echet domain of attraction in \cite{Pereira1983}. Convergence rates for
maxima of dependent sequences have merely been obtained for specific
combinations of dependence patterns and marginal distributions, such as
stationary normal sequences \cite{Cohen1982,Rootzen1983}, or chain-dependent
sequences \cite{McCormickSeymour2001}; results for continuous-time stochastic
processes can be found in \cite{KratzRootzen1997}.

Here, we derive a general result for dependent sequences by exploiting the
interplay between their diagonal power distortions and the extremal properties
of the associated iid sequence. This allows us to transfer uniform convergence
rates from the iid case to dependent sequences by only slightly strengthening
the conditions in Theorem~\ref{theorem_continuous_limit}~\ref{mainThm:1}.

\begin{theorem}[Uniform convergence rate under dependence]\label{theorem_convergence_rate}
  Consider a sequence $X_1,X_2,\ldots\sim\pMargin$ with continuous margin $\pF$
  and assume the following conditions to hold.
\begin{enumerate}[label=(\roman*), labelwidth=\widthof{(iii)}]
  \item\label{itconvrate:1} There exist $\EVDist \in \GEV$ and sequences $(\an_n)$, $\an_n >0$, and $(\bn_n)$ and a map
  $\iidrate \colon \IN \to [0,\infty)$ satisfying $\iidrate(n) \to 0$ as
  $n\to\infty$ such that
 $
    \sup_{x\in\IR} \abs{\pMargin^n(\an_n x + \bn_n) - \EVDist(x)} \leq \iidrate(n).
$
\item\label{itconvrate:2} There exists a rate $\rate \colon \IN \to (0,\infty)$ with $ \rate_n \to \infty$ as $n\to\infty$ and functions  $\pD \colon [0,1] \to [0,1]$ and $s \colon \IN \to [0,\infty)$ with  $s(n) \to 0$ for $n\to\infty$ such that the diagonal power distortion  $\pD^{\rate}_n$ corresponding to $r$ satisfies
$
\sup_{u\in[0,1]} \abs{\pD_n^r(u) - \pD(u)} \leq s(n).
$
  \item\label{itconvrate:3}%
  The limit $\pD$ in~\ref{itconvrate:2} is H\"older continuous with constant
    $K$ and parameter $0 < \kappa \leq 1$.
  \end{enumerate}
  Then
  \begin{align*}
    \sup_{x\in\IR} \abs{\Prob(M_{n} \le \an_{\ceil{\rate_n}} x +\bn_{\ceil{\rate_n}}) - \pD\circ\EVDist(x)} \leq K\left\{\iidrate\left(\ceil{\rate_n}\right) + 3\e^{-1}\I_{\{r_n\notin\IN\}}/\rate_n\right\}^{\kappa} + s(n),
  \end{align*}
  where the right-hand side converges to $0$
  for $n\to\infty$.
\end{theorem}

Assumption~\ref{itconvrate:2} in Theorem~\ref{theorem_convergence_rate} is not
strong. Indeed, if $\pD_n^r$ converges pointwise to $\pD$,
Lemma~\ref{lemma_uniform_convergence} implies that this convergence is uniform on
$[0,1]$, so assumption~\ref{itconvrate:2} merely introduces the notation for
the convergence rate.  To apply Theorem~\ref{theorem_convergence_rate} to any
model where the rate in the associated iid case is known, one therefore only
needs to explicitly identify the rate function $s$ linked to the dependence
structure.
\begin{example}\label{example_moving_max_convergence_rate}
  For the moving maximum process in Example~\ref{example_moving_max}, set
  $\rate_n = n$ to obtain $\pD_n^r(u) = u^{(n+k)/(n(k+1))}$ and
  $\pD(u) = u^{1/(k+1)}$. Because $k \ge 0$, Lemma~\ref{lem:unif:rate:pow:fun} in
  the Supplementary Material \cite{HerrmannHofertNeslehova2024} implies that
  $s(n) =  \sup_{u\in[0,1]} \abs{\pD_n^r(u) - \pD(u)} = \frac{k}{n+k} (1+\frac{k}{n})^{-n/k}$.
  With $k/(n+k) \to 0$ and $(1+k/n)^{-n/k} \to 1/\e$, we indeed see that $s(n) \to 0$ as $n \to \infty$.
  Furthermore, $\pD(u) = u^{1/(k+1)}$ is H\"older continuous with constant $K = 1$ and $\kappa = 1/(k+1)$.

  Using Theorem~\ref{theorem_convergence_rate}, we can now combine the above
  convergence rate of $\pD_n^r$ with uniform convergence rates from the iid
  case. For example, if $\pMargin$ is the standard normal distribution function
  $\Phi$, \cite{Hall1979} provides sequences of constants $(\an_n)$ and
  $(\bn_n)$ such that
  $ \sup_{x\in\IR} \abs{\Phi^n(\an_n x + \bn_n) - \Lambda(x)} \leq 3/\log(n) $,
  where $\Lambda(x) = \EVDist_{0,0,1}(x)$ is the (standard) Gumbel distribution.
  Theorem~\ref{theorem_convergence_rate} applied to a sequence of standard
  normal random variables with dependence specified by the moving maximum
  process thus leads to the convergence rate
  $\sup_{x\in\IR} \abs{\Prob(M_{n} \leq \an_{n} x + \bn_{n}) -
    \Lambda(x)^{1/(k+1)}} \leq (\frac{3}{\log(n)})^{1/(k+1)} +
  \frac{k}{n+k} (1+\frac{k}{n})^{-n/k}$.  From the power $1/(k+1)$ of
  $\iidrate(n)=3/\log(n)$ we see that the upper bound tends to $0$ slower
  than in the iid case, with higher values of $k$ leading to slower convergence
  rates.
\end{example}
Applying Theorem~\ref{theorem_convergence_rate} is particularly easy in case of
sequences with asymptotic power diagonals in the sense of
Definition~\ref{def:3.1}. Indeed, assumption~\ref{itconvrate:3} always holds
given that the limiting distortion $u^\theta$ is H\"older continuous with
$\kappa=\theta$. In the special case of sequences with power diagonals,
$r_n = \eta_n$ and $\pD_n(u) =u$ as discussed in Section~\ref{sec:3}, so that
assumption~\ref{itconvrate:2} is trivially true with $s(n) =0$. This leads to
the following corollary.
\begin{corollary}\label{corollary_PQD_convergence_speed}
  Suppose that $X_1,X_2,\ldots \sim\pF\in\MDA(\EVDist)$ is a sequence with a
  power diagonal and that the continuous margin $\pF$ satisfies
  assumption~\ref{itconvrate:1} of
  Theorem~\ref{theorem_convergence_rate}. Then
  $\sup_{x\in\IR} \abs{\Prob(M_{n} \leq \an_{\ceil{\eta_n}} x
    +\bn_{\ceil{\eta_n}} ) - \EVDist(x)} \leq \iidrate\left(\ceil{\eta_n}\right)
  + 3\e^{-1}\I_{\{\eta_n\notin\IN\}}/\eta_n$.  Additionally, if $\iidrate$ is
  monotonically decreasing and $\Copula_n$ is positively quadrant dependent
  (PQD) for each $n \in \IN$, i.e.\ $\Copula_n(\bm{u}) \ge \Pi_n(\bm{u})$ for
  all $\bm{u}\in[0,1]^d$, then
  $\iidrate\left(\ceil{\eta_n}\right) \geq \iidrate\left(n\right)$.
\end{corollary}

Corollary~\ref{corollary_PQD_convergence_speed} implies that under its
hypothesis and when $\Copula_n$ is PQD for each $n$, the upper bound to the
uniform convergence rate in \eqref{eq:central} is higher compared to the iid
case, but the limiting distribution remains unchanged. The following example
illustrates this point.
\begin{example}\label{ex:4.2}
  As discussed in Section~\ref{sec:3.1}, meta-extreme sequences have power
  diagonals provided that $\STDF_n(1,\dots, 1) \to \infty$ as $n\to \infty$. Because
  any stdf is bounded above by $1$, their copulas are
  also PQD.
  Consider e.g. the logistic sequence with standard
  normal margins, constructed as in Example~\ref{ex:Mai}. Here, for each
  $n \in \IN$, $\Copula_n$ is the Gumbel--Hougaard copula with parameter
  $\theta \ge 1$, where $\theta=1$ corresponds to the iid case, and
  $\eta_n = \rate_n=
  n^{1/\theta}$. %
  Using the same normalizing sequences $(\an_n)$ and $(\bn_n)$ as in
  Example~\ref{example_moving_max_convergence_rate},
  Corollary~\ref{corollary_PQD_convergence_speed} leads to the uniform bound
  $\sup_{x\in\IR} \abs{\Prob(M_{n}\leq  \an_{\ceil{\eta_n}} x + \bn_{\ceil{\eta_n}} ) - \Lambda(x)}
  \leq \frac{3}{\log \left( \ceil{n^{1/\theta}}\right)} + 3\e^{-1}n^{-1/\theta} \approx  \frac{3\theta}{\log \left( n \right)}$.
  Again, a higher degree of dependence, i.e.\ a larger value of $\theta$,
  leads to a higher upper bound of the uniform convergence rate.
\end{example}

  While it may be intuitive that dependence slows down convergence rates, the
  next example shows that this may not always be true. In fact, convergence may
  even be faster compared to the iid case if the maximum does not need to be
  normalized. An extreme case in point is a comonotone sequence, where
  $M_n \sim F$ for each $n$ and the convergence is instantaneous.
\begin{example}\label{example_power_diagonal_convergence_speed}
  Consider the meta-extreme sequence in Example~\ref{example_Cuadras_Auge},
  where $\Copula_n$ is the Cuadras-Aug\'e copula with parameter
  $\theta \in (0,1)$. Because
  $\eta_n = \STDF_n(1,\ldots,1) = (1-(1-\theta)^n)/\theta$ so that
  $\eta_n \to 1/\theta$, this sequence does not have a power diagonal in the
  sense of Definition~\ref{def:3.1}, so that
  Corollary~\ref{corollary_PQD_convergence_speed} does not apply. Nonetheless,
  we can calculate that
  $\sup_{x\in\IR} \abs{\Prob(M_{n}\leq x) - \pMargin^{1/\theta}(x)} = \sup_{x\in\IR}\abs{\pMargin^{\eta_n}(x) - \pMargin^{1/\theta}(x)} = \sup_{u\in[0,1]} \abs{u^{1/\theta-(1-\theta)^n/\theta} - u^{1/\theta}}$,
where equality holds because $\pMargin$ is continuous.
By Lemma~\ref{lem:unif:rate:pow:fun}~\ref{LemmaA3:1} in the Supplementary Material \cite{HerrmannHofertNeslehova2024},
$\beta(n) := \sup_{u\in[0,1]} |u^{1/\theta-(1-\theta)^n/\theta} - u^{1/\theta}| \le (1 - \theta)^n\{1-(1-\theta)^n\}^{(1-\theta)^{-n}-1}$.
For all $n$ large enough we have, by
Lemma~\ref{lem:unif:rate:pow:fun}~\ref{LemmaA3:2}, that
$\beta(n) \leq 3\e^{-1}(1-\theta)^n$, which shows that not only
$\lim_{n\to\infty} \beta(n) = 0$ but also (in contrast to
Example~\ref{example_moving_max_convergence_rate} and Example~\ref{ex:4.2}) that a
higher degree of dependence, i.e.\ a larger value of $\theta$, leads to a
\emph{smaller} upper bound on the uniform convergence rate. When $\theta=1$, the
sequence is comonotone and $\beta(n) =0$ for all $n$.
\end{example}
We conclude this section with a partial converse to
Theorem~\ref{theorem_convergence_rate} which allows us to infer the convergence
rate of the dependence structure when the overall convergence rate is known.

\begin{proposition}%
\label{proposition_convergence_rate_dependence_structure}
Consider $X_1,X_2,\ldots\sim\pMargin$ where the continuous margin $\pF$ satisfies assumption~\ref{itconvrate:1} of Theorem~\ref{theorem_convergence_rate}. Suppose also that there exists  $\rate \colon \IN \to (0,\infty)$ with $\rate_n \to \infty$ for $n\to\infty$ such that the following conditions are fulfilled:
  \begin{enumerate}[label=(\roman*), labelwidth=\widthof{(iii)}]
  \item%
    There exists a distribution function $\pD : [0,1] \to [0,1]$ such that
    $ \sup_{x\in\IR}\abs{\Prob(M_{n} \le
      \an_{\ceil{\rate_n}}x+\bn_{\ceil{\rate_n}}) - \pD\circ\EVDist(x)} \leq
    \gamma(n) $, where $\gamma(n)\to 0$ as $n\to\infty$;
  \item%
    For each $n \in \IN$, the diagonal power distortion $\pD_n^{\rate}$ with
    respect to $\rate$ is H\"older continuous with constants $K_n$ and
    $\kappa_n$, and there exist $K,\kappa >0$ such that
    $\sup_{n\in\IN} K_n \leq K$ and $\inf_{n\in\IN} \kappa_n \geq \kappa$.
  \end{enumerate}
  Then
  $ \sup_{u\in[0,1]}\abs{\pD_n^{\rate}(u)-\pD(u)} \leq K
  \left(\iidrate\left(\ceil{\rate_n}\right) +
    3\e^{-1}\I_{\{r_n\notin\IN\}}/\rate_n\right)^{\kappa} + \gamma(n), $ where
  the right-hand side
  converges to $0$ as $n\to\infty$.
\end{proposition}

Proposition~\ref{proposition_convergence_rate_dependence_structure} can be used
to transfer convergence-rate results for time series to convergence rates
of the underlying dependence structure represented by
$\pD_n^{\rate}$. %
\begin{example}
  For the moving maximum process in Example~\ref{example_moving_max}, the
  diagonal power distortion is given by $\pD_n^{\rate}(u) = u^{(n+k)/(nk + n)}$
  with $\rate_n=n$. Clearly, $\pD_n^{\rate}$ is H\"older continuous with
  $\kappa_n = (n+k)/(nk + n) \le 1$ and $K_n = 1$. Because
  $\kappa_n \to 1/(k+1)$ as $n\to \infty$, assumptions~\ref{itconvrate:1} and
  \ref{itconvrate:2} in
  Proposition~\ref{proposition_convergence_rate_dependence_structure} hold.
\end{example}

  \section{Conclusion}\label{section_conclusion}
  We have investigated weak convergence limits of maxima of dependent sequences
  of identically distributed random variables after suitable affine
  normaliziation accomplished through subsequences of the normalizing sequences
  from the iid case.  We have worked under broad conditions on the diagonal of
  the underlying copula and demonstrated through various examples that these
  conditions are fulfilled by a large number of sequences.  The limiting
  distribution is a composition of the generalized extreme value limit of the
  associated iid sequence and a function, termed distortion, which can have
  various shapes. This characterization could be particularly useful for
  future development of statistical inference methods.

  We close by remarking that there exist dependent sequences whose maxima
  converge after standardization with sequences that cannot be subsequences of
  the normalizing sequences from the iid case. Consider the following example
  from \cite{Berman1962a}. Let $(Z_i)$ be a sequence of iid standard normal
  random variables, and set, for each $i \in \IN$,
  $X_i = \sqrt{\varrho} Z_0 + \sqrt{1-\varrho}Z_{i}$. It is easily verified that
  $X_i \sim \Normal(0,1)$ for each $i \in \IN$, and that
  $\mathrm{cor}(X_i, X_j) = \varrho$ whenever $i \neq j$. As is well-known,
  there exist sequences $(\an_n)$, $\an_n>0$ and $(\bn_n)$, such that
  $\{\max(Z_1,\dots, Z_n) - \bn_n\}/\an_n \indist \EVDist_{0,\mu,\sigma}$; from
  the convergence to types theorem \cite[Proposition~0.2]{Resnick1987} and
  \cite[Example~3.3.29]{EmbrechtsKluppelbergMikosch1997}, one has
  $\an_n = O( (2\log n)^{-1/2})$. %
  However, a direct calculation shows that
  $\{\max(X_1,\dots, X_n) -\dn_n\}/\cn_n \indist \Normal(0,\varrho)$, where for
  each $n \in \IN$, $\dn_n = \sqrt{1-\varrho} \bn_n$ and $\cn_n = 1$. If
  $\tilde\cn_n$ is of the form $\an_{\lceil \rate_n \rceil}$ for some rate
  $\rate$ with $\rate_n \to \infty$ as $n \to \infty$, we cannot possibly have
  that $\tilde\cn_n / \cn_n \to c \in (0,\infty)$. Even though several existing
  results are embedded in our framework, a unified theory thus remains to be
  developed.

\appendix
\section{Auxiliary results and their proofs}
  \begin{lemma}%
  \label{lemma_uniform_convergence}
  Denote by $f_n$ a sequence of non-decreasing functions that converges
  pointwise to a continuous limit $f$ on an interval $[a,b]$.  Then the
  convergence of $f_n$ to $f$ is uniform.
  \end{lemma}
  \begin{proof}
    Given that $f$ is continuous on a compact interval, it is uniformly
    continuous and hence for a given $\varepsilon > 0$ there is a partition
    $a = t_0 < t_1 < \cdots < t_K = b$ such that $x,y \in [t_i,t_{i+1}] \Rightarrow \abs{f(x)-f(y)} < \varepsilon/2$.
    For every $t_i$, $i=0,\ldots,K$, we also have $f_n(t_i) \to f(t_i)$ and
    hence there is a $n_{\varepsilon}$ such that for all
    $n \geq n_{\varepsilon}$ we have $\abs{f_n(t_i)-f(t_i)} < \varepsilon/2$
    simultaneously for all $i=0,\ldots,K$.  This immediately yields that for all
    $n \ge n_{\varepsilon}$, $f_n(t_{i+1}) < f(t_{i+1}) + \varepsilon/2$ and $f(t_{i}) - \varepsilon/2 < f_n(t_{i})$.
    Now fix $\varepsilon > 0$ and an arbitrary $x \in [a,b]$. Then
    $x \in [t_i,t_{i+1}]$ for some $i\in\{0,\ldots,K\}$.  By the monotonicity of
    $f_n$, $f_n(t_i) \leq f_n(x) \leq f_n(t_{i+1})$, and hence with the previous
    two inequalities $f(t_{i}) - \varepsilon/2 < f_n(t_{i}) \leq f_n(x) \leq f_n(t_{i+1}) < f(t_{i+1}) + \varepsilon/2$.
    Considering now $f(x)$ we have by design of the partition that
    $f(x) - \varepsilon/2 < f(t_{i})$ and $f(t_{i+1}) < f(x) + \varepsilon/2$.
    Therefore, $ f(x) - \varepsilon < f_n(x) < f(x) + \varepsilon$, or
    equivalently $\abs{f_n(x)-f(x)} < \varepsilon$.  Since $\varepsilon$ and $x$
    were arbitrary this concludes the proof.
  \end{proof}

  \begin{lemma}%
    \label{Lemma_function_composition}
    Denote by $(f_n)$ and $(g_n)$ two sequences of functions
    $f_n,g_n \colon \IR \to \IR$, and by $(h_n)_{n\geq 1}$ the composition
    $h_n = f_n \circ g_n$, i.e.\ $h_n(x) = f_n(g_n(x))$ for all
    $x \in \mathbb{R}$.  Suppose that $f_n$ converges uniformly to a function
    $f\colon \IR \to \IR$ which is continuous on $\mathcal{A} \subseteq \IR$ and
    that $g_n$ converges to pointwise on $\mathcal{A}$ to an arbitrary function
    $g\colon \IR \to \IR$, then $h_n$ converges pointwise to $h = f \circ g$ on
    $\mathcal{A}$.
  \end{lemma}
  \begin{proof}
    Fix an arbitrary $x \in \mathcal{A}$ and set $y_n = g_n(x)$ and $y = g(x)$
    for further reference.  Then $h_n(x) = f_n(y_n)$ and $h(x) = f(y)$ which
    leads to $\abs{h_n(x) - h(x)} = \abs{f_n(y_n) - f(y) + f(y_n) - f(y_n)} \leq \abs{f_n(y_n) - f(y_n)} + \abs{f(y_n) - f(y)}$.
    Due to pointwise convergence of $g_n$ to $g$ on $\mathcal{A}$, $y_n \to y$
    and hence the second term tends to zero by continuity of $f$.  For the first
    term we have $\abs{f_n(y_n) - f(y_n)} \leq \sup_{y\in\IR}\abs{f_n(y)-f(y)}$
    which converges to zero because $f_n \to f$ uniformly by assumption.
  \end{proof}
  \clearpage

  \begin{lemma}%
    \label{lem:unif:rate:pow:fun}%
    \begin{enumerate}[label=(\roman*), labelwidth=\widthof{(iii)}]
    \item\label{LemmaA3:1} For arbitrary fixed $a,b$ such that $0 < a < b$,
      \begin{align}\label{eq_max_of_power_difference}
        \sup_{u\in[0,1]} \abs{u^a-u^{b}} = (1-a/b)(b/a)^{a/(a-b)}.
      \end{align}
    \item\label{LemmaA3:2} For any fixed $a > 0$ a sequence $b_n \to 0$ we have
      for $n$ large enough that $\sup_{u\in[0,1]} \abs{u^a-u^{a+b_n}} \leq 3\e^{-1}\abs{b_n}/a$.
    \item\label{LemmaA3:3} For any fixed $b > 0$ a sequence $a_n \to \infty$ we
      have for $n$ large enough that $\sup_{u\in[0,1]} \abs{u^{a_n}-u^{a_n+b}} \leq 3\e^{-1}b/a_n$.
    \end{enumerate}
  \end{lemma}
  \begin{proof}\mbox{}
    \begin{enumerate}[label=(\roman*), labelwidth=\widthof{(iii)}]
    \item Fix arbitrary $0 < a < b$ and observe that the function $f$ given, for
      all $u\in [0,1]$, by $f(u) = u^a - u^{b}$ is continuous, positive, and
      satisfies $f(0) = f(1) = 0$.  Furthermore, $u^* = (b/a)^{1/(a-b)}$ is
      easily seen to be the maximizer of $f$ given that
      $f^\prime(u) = au^{a-1} -
      bu^{b-1}$. Equation~\eqref{eq_max_of_power_difference} then follows from
      the fact that $f(u^*) = (1-a/b)(b/a)^{a/(a-b)}$.
    \item Fix an arbitrary $a >0$ and let $n\in\IN$ be sufficiently large such that
      we have $b_n \in (-a \log(2), a \log(2)/(1-\log(2)))$. If $b_n =0$, then the claim
      is trivially true. Otherwise, a case distinction needs to be made,
      depending on the sign of $b_n$.

      In the case when $b_n >0$, we make use of the inequality
      $1-(1/x)\log(1+x) \le x/(x+1)$ valid for all $ x > 0$. By setting
      $x=b_n/a$, and noting that $b_n < a \log(2)/(1-\log(2))$ is equivalent to
      $b_n/(b_n+a) < \log(2)$, we obtain
      $1-\frac{a}{b_n} \log(1+ \frac{b_n}{a}) \le \frac{b_n}{b_n + a}
      < \log(2)$.  This allows us to apply
      \cite[Lemma~2.4.1~(ii)]{LeadbetterLindgrenRootzen1983} to derive that for
      some $\theta\in (0,1)$,
      \begin{align*}
        (1+b_n/a)^{-a/b_n} - \e^{-1}
        &= \e^{-a \log(1+b_n/a)/b_n} - \e^{-1}\\
        &= \e^{-1}\left[\{1-a \log(1+b_n/a)/b_n\} + \theta\{1-a\log(1+b_n/a)/b_n\}^2\right]\\
        &\le\e^{-1}\frac{b_n}{a+b_n} \left( 1+\theta \frac{b_n}{a+b_n} \right) \le 2\e^{-1}\frac{b_n}{a+b_n},
      \end{align*}
      where the last inequality stems from $b_n/(b_n + a) \le \log(2) <
      1$. Equation~\eqref{eq_max_of_power_difference} from part~\ref{LemmaA3:1}
      then leads to
      \begin{align*}
        \sup_{u\in[0,1]} \abs{u^a-u^{a+b_n}}&=(1+b_n/a)^{-a/b_n} \frac{b_n}{a+b_n} = \bigl\{\e^{-1} + (1+b_n/a)^{-a/b_n} - \e^{-1}\bigr\} \frac{b_n}{a+b_n}\\
                                            &\leq \left(\e^{-1} + 2\e^{-1}\frac{b_n}{a+b_n}\right) \frac{b_n}{a+b_n} = \e^{-1} \frac{b_n}{a+b_n} \left(1 + \frac{2b_n}{a+b_n}\right).
      \end{align*}
      Using again that $b_n/(b_n + a)< 1$ and further that $\theta \in (0,1)$
      and $b_n/(a+b_n) \le b_n/a$, we finally get
      $\sup_{u\in[0,1]} \abs{u^a-u^{a+b_n}} \le 3 \e^{-1} b_n/a$.

      The case when $b_n < 0$ is similar, but here we make use of the inequality
      $\log(1+x) \le x$ valid for all $x > -1$, which leads to
      $1-(1+1/x)\log(1+x) \le -x$ whenever $x \in (-1,0)$. Setting $x = b_n /a$,
      we then have that $1-(1+\frac{a}{b_n})\log(1 + \frac{b_n}{a}) \le -\frac{b_n}{a} < \log(2)$,
      where the last inequality holds because $n$ is assumed large enough so
      that $b_n > -a \log(2)$. \cite[Lemma~2.4.1~(ii)]{LeadbetterLindgrenRootzen1983} then
      implies that for some $\theta \in (0,1)$,
      \begin{align*}
         &\phantom{{}={}}(1+b_n/a)^{-(1+a/b_n)} - \e^{-1} = \e^{-(1+a/b_n)\log(1+b_n/a)} - \e^{-1} & \quad\\
         &= \e^{-1}\left[1 - (1+a/b_n)\log(1+b_n/a) +
           \theta\bigl\{
           1 - (1+a/b_n)\log(1+b_n/a)
           \bigr\}^2
           \right]\\
         &\le \frac{|b_n|}{a}\e^{-1}\left(1 + \theta \frac{\abs{b_n}}{a}\right) \le 2\e^{-1}\frac{\abs{b_n}}{a},
      \end{align*}
      where the last inequality derives from $\abs{b_n}/a < \log(2)
      <1$. Equation~\eqref{eq_max_of_power_difference} from part~\ref{LemmaA3:1} then
      yields
      \begin{align*}
        \sup_{u\in[0,1]} \abs{u^a-u^{a+b_n}}
        &= -\frac{b_n}{a}(1+b_n/a)^{-(1+a/b_n)}
          = -\frac{b_n}{a}\bigl\{\e^{-1} + (1+b_n/a)^{-(1+a/b_n)} - \e^{-1}\bigr\}\\
        &\leq -\frac{b_n}{a}\left(\e^{-1} +2\e^{-1}\frac{\abs{b_n}}{a}\right)\le 3\e^{-1}\frac{\abs{b_n}}{a},
      \end{align*}
      as claimed, where the last step utilizes that $\abs{b_n}/a <1$ one more
      time.
    \item First note that $(1-\log(2))/\log(2) \approx 0.4427$. Suppose that $n$
      is sufficiently large so that $b(1-\log(2))/\log(2) \leq a_n$. We then
      have $b/(a_n+b) \leq \log(2)$. Calling again on the inequality
      $1 - \log(1+x)/x \leq x/(x+1)$ valid for all $x>0$, we obtain with
      $x = b/a_n$ that $1 - (a_n/b)\log(1+b/a_n) \leq b/(a_n+b) \le \log(2)$.
      Applying \cite[Lemma~2.4.1~(ii)]{LeadbetterLindgrenRootzen1983} then yields that for some $\theta\in(0,1)$,
      \begin{align*}
        (1+b/a_n)^{-a_n/b} - \e^{-1}
        &= \e^{-a_n \log(1+b/a_n)/b} - \e^{-1}\\
        &=  \e^{-1}\left[\{1-a_n \log(1+b/a_n)/b\} + \theta\{1-a_n\log(1+b/a_n)/b\}^2\right]\\
        &\le \e^{-1}\frac{b}{a_n+b} \left( 1+ \frac{\theta b}{a_n+b} \right) \le 2\e^{-1}\frac{b}{a_n+b},
      \end{align*}
      where the last inequality holds since $b/(b+a_n) \leq \log(2) <
      1$. Equation~\eqref{eq_max_of_power_difference} in part~\ref{LemmaA3:1} then
      gives
      \begin{align*}
        \sup_{u\in[0,1]} \abs{u^{a_n}-u^{a_n+b}}&\le \frac{b}{b+a_n}(1+b/a_n)^{-a_n/b} = \frac{b}{b+a_n}\left\{\e^{-1} + (1+b/a_n)^{-a_n/b} - \e^{-1}\right\}\\
                                                &\leq  \frac{b}{b+a_n} \left(\e^{-1} + 2\e^{-1}\frac{b}{a_n+b}\right) \leq 3 \e^{-1} \frac{b}{b+a_n},
      \end{align*}
      where $b/(b+a_n)< 1$ is used in the last step. Since $b/(b+a_n) < b/a_n$,
      the claim follows.
    \end{enumerate}
  \end{proof}

  \begin{corollary}%
    \label{corollary_pointwise_convergence_power_limit}
    Denote by $(y_n)$ a sequence of real numbers such that $y_n \to y > 0$ and
    by $(\pMargin_n)$ a sequence of distribution functions on $\IR$ that
    converges weakly to a distribution function $\pMargin$ on $\IR$.  Then
    $\pMargin_n^{y_n}$ converges weakly to $\pMargin^{y}$.
  \end{corollary}
  \begin{proof}
    Because $\pMargin^{y}$ and $\pMargin$ have the same points of continuity, we
    need to show that $\pMargin_n^{y_n}(x) \to \pMargin^y(x)$ as $n \to \infty$
    at any continuity point $x$ of $\pMargin$. To this end, set $a = y$ and
    $b_n = y_n-y$ and define the functions $g_n$ and $g$ for all $u \in [0,1]$
    by $g_n(u) = u^{y_n}$ and $g(u) = u^y$,
    respectively. Lemma~\ref{lem:unif:rate:pow:fun}~\ref{LemmaA3:2} then
    implies that $\sup_{u\in[0,1]} \abs{g(u)-g_n(u)} = \sup_{u\in[0,1]} \abs{u^{a}-u^{a + b_n}} \leq \frac{3}{y} \e^{-1}\abs{y_n-y} \to 0$,
    i.e.\ $g_n$ converges uniformly to $g$. Given that $g$ is obviously
    continuous, and $\pMargin_n(x) \to \pMargin(x)$ at any continuity point of
    $\pMargin$ we can conclude from Lemma~\ref{Lemma_function_composition} that
    the composition $g_n \circ \pMargin_n$ converges pointwise to
    $g \circ \pMargin$ at any continuity point of $\pMargin$ as claimed.
  \end{proof}

  \begin{corollary}%
    \label{corollary_uniform_rate_upper_Iverson}
    Denote by $(r_n)$ a sequence of positive real numbers such that
    $r_n \to \infty$.  Then for all $n\in\IN$ we have
    \begin{align*}
      \sup_{u\in[0,1]} \abs{u-u^{r_n/\ceil{r_n}}} \leq 3\e^{-1}\frac{1}{\ceil{r_n}},
    \end{align*}
    where the right-hand side tends to $0$ as $n\to \infty$.
  \end{corollary}
  \begin{proof}
    Set $a=1$ and $b_n = r_n / \ceil{r_n} -1 = (r_n - \ceil{r_n})/\ceil{r_n}$.
    Clearly, $b_n \leq 0$ and $b_n \to 0$ as $n\to
    \infty$. Part~\ref{LemmaA3:2} of Lemma~\ref{lem:unif:rate:pow:fun} and its
    proof then yield the inequality holds as soon as $b_n > -\log(2)$, which
    happens whenever $r_n > 1$.  However, for $r_n \leq 1$ we have
    $3\e^{-1}/\ceil{r_n} = 3\e^{-1} \approx 1.103638$ in which case the
    statement continues to be correct since we always have
    $\abs{u-u^{r_n/\ceil{r_n}}} \leq 1$.
  \end{proof}

  \begin{corollary}%
    \label{corollary_uniform_difference_upper_Iverson}
    Denote by $(r_n)$ a sequence of positive real numbers such that
    $r_n \to \infty$ and by $\pMargin$ a distribution function on $\IR$.  We
    then have for all $n\in\IN$ that
    \begin{align*}
      \sup_{x\in\IR} \abs{\pMargin^{\ceil{r_n}}(x) - \pMargin^{r_n}(x)} \leq 3\e^{-1}\frac{1}{r_n}
    \end{align*}
    and the right-hand side tends to $0$ as $n\to \infty$.
  \end{corollary}
  \begin{proof}
    First note that
    $\sup_{x\in\IR} \abs{\pMargin^{\ceil{r_n}}(x) - F^{r_n}(x)}
      \leq \sup_{x\in\IR} \abs{\pMargin^{r_n + 1}(x) - F^{r_n}(x)}
      \leq \sup_{u\in[0,1]}$ $\abs{u^{r_n}-u^{r_n+1}}$. %
    Upon setting $b=1$ and $a_n=r_n$, part~\ref{LemmaA3:3} of
    Lemma~\ref{lem:unif:rate:pow:fun} and its proof imply that the right-hand
    side is bounded above by $3 \e^{-1}/r_n$ as soon as
    $r_n \geq (1-\log(2))/\log(2) \approx 0.442695$.  However, when
    $r_n < (1-\log(2))/\log(2)$ we have $3\e^{-1}/r_n > 1$ and so
    $|\pMargin^{\ceil{r_n}}(x) - \pMargin^{r_n}(x)| \leq 1$ is always true.  As
    such the statement holds for all $n\in\IN$, even though it only becomes
    informative when $3\e^{-1}/r_n < 1$.
  \end{proof}

  \section{Proofs of the main results}

  \begin{proof}[Proof of Theorem~\ref{theorem_continuous_limit}] For
    part~\ref{mainThm:1}, define $f_n$ and $g_n$ for all $x \in \mathbb{R}$
    and $u \in [0,1]$ by
    $f_n(x) = \pMargin^{\ceil{\rate_n}}(\an_{\ceil{\rate_n}} x +
    \bn_{\ceil{\rate_n}})$ and $g_n(u) = u^{\rate_n/\ceil{\rate_n}}$,
    respectively. The fact that $r_n \to \infty$ as $n \to \infty$ implies that
    $f_n \to \EVDist$ pointwise as $n\to\infty$. Now consider
    $h_n = g_n \circ f_n$ so that for all $x \in \mathbb{R}$,
    $h_n(x) = \pMargin^{\rate_n}(\an_{\ceil{\rate_n}} x +
    \bn_{\ceil{\rate_n}})$. Corollary~\ref{corollary_uniform_rate_upper_Iverson}
    together with Lemma~\ref{Lemma_function_composition} implies that
    $h_n \to \EVDist$ pointwise on $\mathbb{R}$ as $n\to\infty$. Furthermore,
    Lemma~\ref{lemma_uniform_convergence} implies that the convergence
    $\pD^{\rate}_n \to \pD$ is uniform. Hence, for all $x \in \mathbb{R}$ by
    Lemma~\ref{Lemma_function_composition},
    $\Prob(M_n \leq \an_{\ceil{\rate_n}} x + \bn_{\ceil{\rate_n}}
    )= %
    \Cdiag_n[\{\pMargin^{\rate_n}(\an_{\ceil{\rate_n}}
    x+\bn_{\ceil{\rate_n}})\}^{1/\rate_n}] = (\pD_n^{\rate} \circ h_n)(x) \to
    \pD\{\EVDist(x)\}$.
    As a uniform limit of functions with these properties, $D$ is continuous,
    non-decreasing, $\pD(0)=0$, and $\pD(1)=1$. Hence $\pD$ is a continuous
    distribution function on $[0,1]$.

    For part~\ref{mainThm:2}, $D_n^r(u) = \delta_n(u^{1/r_n})$ implies
    $\pD_n^{\rate}(u) = \pD(u)$ for $u \in \{0,1\}$ and all $n\in\IN$.  For a
    given continuity point $u\in(0,1)$ of $\pD$, pick $x\in\IR$ such that
    $u = \EVDist(x)$ and set
    $u_n = \pMargin^{\rate_n}(\an_{\ceil{\rate_n}}x + \bn_{\ceil{\rate_n}})$. By
    construction and
    Corollary~\ref{corollary_pointwise_convergence_power_limit},
    $u_n = (\pMargin^{\ceil{\rate_n}}(\an_{\ceil{\rate_n}}x +
    \bn_{\ceil{\rate_n}}))^{\rate_n/\ceil{\rate_n}} \to u$, and, by the triangle
    inequality,
    \begin{align}\label{Theorem2.3-decomposition}
      \abs{\pD_n^{\rate}(u)-\pD(u)} &\leq \abs{\pD_n^{\rate}(u)-\pD_n^{\rate}(u_n)} + \abs{\pD_n^{\rate}(u_n)-\pD(u)}.
    \end{align}
    Now choose an arbitrary $\varepsilon>0$. By assumption, there exists
    $n_1\in\IN$ such that
    $\abs{\pD_n^{\rate}(u_n)-\pD(u)}=
    \abs{\Prob((M_{n}-\bn_{\ceil{\rate_n}})/\an_{\ceil{\rate_n}} \leq x) -
      \pD\circ\EVDist(x)} < \varepsilon/2$ for all $n>n_1$.  To handle the first
    term in \eqref{Theorem2.3-decomposition}, Lipschitz continuity of copulas
    \cite[Theorem~2.10.7]{Nelsen2006} implies that
    $\abs{\pD_n^{\rate}(u)-\pD_n^{\rate}(u_n)}\leq
    n\abs{u^{1/\rate_n}-u_n^{1/\rate_n}}$.  Next, let $f_n(y) = y^{1/\rate_n}$,
    $y \in [0,1]$, and pick $\delta > 0$ small enough so that
    $\delta < \min(u,1-u)$. By assumption and $u_n\to u$, there exists $n_2$ so
    that for all $n > n_2$, $1/r_n < 1$ and at the same time
    $u_n \in (u-\delta, u+\delta)$. The derivative
    $f_n^\prime(y) = (y^{1/\rate_n-1})/\rate_n$ is then continuous and
    decreasing in $y$ on $[u-\delta, u+\delta]$ and attains its maximum at
    $u-\delta$. This means that $f_n$ is Lipschitz on $[u-\delta,u+\delta]$ with
    constant $L_n = f_n^\prime(u-\delta)$ whenever $n > n_2$. This constant can
    be bounded above via $L_n = (u-\delta)^{1/r_n-1}/r_n <
    1/((u-\delta)r_n)$. Because $1/r_n = O(1/n)$, there exists $\kappa > 0$ so
    that for all $n > n_2$, $n/r_n \le \kappa$. Put together, we obtain that
    $n\abs{u^{1/\rate_n}-u_n^{1/\rate_n}} \leq n L_n \abs{u-u_n} <
    \frac{n}{(u-\delta)\rate_n} \abs{u-u_n} \leq \frac{\kappa}{(u-\delta)}
    \abs{u-u_n}$.  Hence, there exists $n_3$ so that for all $n > n_3$,
    $|u-u_n| < \varepsilon(u-\delta)/2\kappa$. Thus
    $|\pD_n^{\rate}(u)-\pD_n^{\rate}(u_n)| < \varepsilon/2$ if
    $n > \max(n_2,n_3)$ and $|\pD_n^{\rate}(u)-\pD(u)| < \varepsilon$ if
    $n > \max(n_1,n_2,n_3)$.
  \end{proof}

  \begin{proof}[Proof of Proposition~\ref{cor:r(n)finite}]
    Part~\ref{cor:case:1} follows immediately via the same techniques as
    Theorem~\ref{theorem_continuous_limit}~\ref{mainThm:1}.  For part
    \ref{cor:case:2}, Lemma~\ref{lem:unif:rate:pow:fun}~\ref{LemmaA3:2}
    yields that
    $\sup_{x\in\IR} \abs{\pMargin^{\varrho}(x)-\pMargin^{\rate_n}(x)} =
    \sup_{u\in[0,1]} \abs{u^{\varrho}-u^{\varrho + (\rate_n-\varrho)}} \leq
    \frac{3}{\varrho \e}\abs{\rate_n-\varrho}$.  The result now follows along
    the same steps as in the proof of
    Theorem~\ref{theorem_continuous_limit}~\ref{mainThm:2} with the
    alternative definitions $u = \pMargin^{\varrho}(x)$ and
    $u_n = \pMargin^{\rate_n}(x)$; these are possible since $\pMargin$ is
    continuous.  The key inequality then becomes
    $n\abs{u^{1/\rate_n}-u_n^{1/\rate_n}} < \frac{n}{\rate_n
      (u-\delta)}\abs{u-u_n} = \frac{n}{\rate_n
      (u-\delta)}\abs{\pMargin^{\varrho}(x) - \pMargin^{\rate_n}(x)} \le
    \frac{n}{\rate_n (u-\delta)} \frac{3}{\varrho \e}\abs{\rate_n-\varrho} \to
    0$ which implies $\pD^{\rate}_n(u) \to \pD(u)$ for $u\in\{0,1\}$ and all
    $u\in(0,1)$ where $\pD$ is continuous.
  \end{proof}

  \begin{proof}[Proof of Corollary~\ref{corollary_FTG}]
    As in Theorem~\ref{theorem_continuous_limit}~\ref{mainThm:1}, let $M_n^*$
    denote the maximum of the first $n$ variables of an iid sequence with
    marginal distribution function $\pMargin$ and $(c_n^*)$, $c_n^* > 0$ and
    $(d_n^*)$ be such that
    $(M_n^* -
    d_n^*)/c_n^*\indist\EVDist_{\EVI,\mu,\sigma}$. From~\eqref{eq_generalized_FTP}
    and the Convergence to Types Theorem, we have that
    $\cn_n/\an_{\ceil{\rate_n}} \to c$ and
    $(\dn_n-\bn_{\ceil{\rate_n}})/\an_{\ceil{\rate_n}} \to d$ as $n \to \infty$
    for some $c > 0$ and $d \in \IR$, and that
    $\pG(x) = \pD\circ \EVDist_{\EVI,\mu,\sigma}(cx + d)$, $x \in \IR$. It is
    easily seen that
    $\EVDist_{\EVI,\mu,\sigma}(cx + d) =
    \EVDist_{\EVI,\tilde{\mu},\tilde{\sigma}}(x)$ with $\tilde{\mu} = (\mu-d)/c$
    and $\tilde{\sigma} = \sigma/c$.
  \end{proof}

  \begin{proof}[Proof of Theorem~\ref{theorem-FTG2}]
    Because $\pD$ is a strictly increasing, continuous bijection, it has a
    continuous inverse $\pD^{-1}$
    \cite[p.~53]{RoydenFitzpatrick2010}. %
    Hence, we can set $\EVDist = \pD^{-1} \circ \pG$ and note that the
    continuity points of $\EVDist$ and $\pG$ coincide. We now need to prove that
    $\EVDist$ is max-stable.  By assumption, for each continuity point $x$ of
    $\pG$,
    $\pD_n^{\rate} \{\pMargin^{\rate_n}(\cn_n x + \dn_n)\} \to \pG(x)$
    as $n\to \infty$. Again by assumption, $\pD_n^{\rate}$ is strictly
    increasing, so it is invertible and $(\pD_n^{\rate})^{-1} \to D^{-1}$
    pointwise on $[0,1]$; see \cite[Proposition~0.1]{Resnick1987}.  By
    Lemma~\ref{lemma_uniform_convergence} this convergence is uniform and
    hence, by Lemma~\ref{Lemma_function_composition},
    \begin{align}\label{eq:neededforProp3.1}
      \lim_{n\to \infty} \pMargin^{\rate_n}(\cn_n x + \dn_n) = \lim_{n \to \infty}(\pD_n^\rate)^{-1} \bigl[ \pD_n^\rate \{\pMargin^{\rate_n} (\cn_n x + \dn_n)\}\bigr]=\pD^{-1}\{\pG(x)\} = \EVDist(x)
    \end{align}
    for each continuity point $x$ of $\EVDist$.
    Therefore, for any $t > 0$ and any continuity point $x$ of $\EVDist$,
    \begin{align}
      \lim_{n\to\infty} \pMargin^{t\rate_n}(\cn_n x + \dn_n) = \EVDist^t(x)\label{theorem-FTG2:proof:1}
    \end{align}
    as $n \to \infty$. At the same time, Corollary~\ref{corollary_pointwise_convergence_power_limit} yields
    \begin{align}
      \lim_{n \to \infty} \pMargin^{t\rate_n}(\cn_{\ceil{tn}} x + \dn_{\ceil{tn}}) = \lim_{n \to \infty} \bigl\{\pMargin^{\rate_{\ceil{tn}}}(\cn_{\ceil{tn}} x + \dn_{\ceil{tn}})\bigr\}^{t \rate_n/\rate_{\ceil{tn}}} = \{\EVDist(x)\}^{t/\lambda(t)}.\label{theorem-FTG2:proof:2}
    \end{align}
    Applying the convergence to types theorem with $U = \EVDist^{t/\lambda(t)}$
    and $V = \EVDist^{t}$, to~\eqref{theorem-FTG2:proof:1}
    and~\eqref{theorem-FTG2:proof:2}, there exist functions $c(t) > 0$ and
    $d(t) \in \IR$ so that $\EVDist^{\lambda(t)}(x) = \EVDist\{c(t) x + d(t)\}$,
    $x\in \IR$. Because $\lambda$ is a bijection, it then follows that for any
    $t > 0$ and $x\in \IR$,
    $\EVDist^t(x) = \EVDist\bigl[c\{\lambda^{-1}(t)\} x +
    d\{\lambda^{-1}(t)\}\bigr]$. This means that $\EVDist$ is max-stable, and
    hence, by the classical Fisher--Tippett--Gnedenko theorem, $\EVDist\in\GEV$.
  \end{proof}

  \begin{proof}[Proof of Corollary~\ref{corollary_copula_diagonal_extremal_index}]
    To prove~\ref{item:diag:1}, set $x = \EVDist^{-1}(u)$ and
    $v_n = F^n(\cn_n x + \dn_n)$ so that $v_n \to \EVDist(x) = u$ as
    $n\to \infty$. For this $x$ and rate $r(n) = n$,
    $\Prob(M_n \le \cn_n x + \dn_n) = \pD_n^{\rate}(v_n)$. As in the proof of
    Theorem~\ref{theorem_continuous_limit}~\ref{mainThm:2}, we can argue that
    $\abs{\pD_n^{\rate}(v_n) - \pD_n^{\rate}(u)} \to 0$ as $n\to \infty$, so
    that $\Prob(M_n \le \cn_n x + \dn_n) \to \gamma$. From
    \cite[Theorem~2.2]{Leadbetter1983} we deduce that there exists
    $\theta \in [0,1]$ so that
    $\Prob(M_n \le \cn_n z + \dn_n) \to \EVDist^{\theta}(z)$ for any
    $z \in \IR$. Because $\gamma \in (0,1)$, the degenerate case $\theta =0$ is
    excluded. The claim now follows from
    Theorem~\ref{theorem_continuous_limit}~\ref{mainThm:2} and
    Definition~\ref{def:3.1}.

    Part~\ref{item:diag:2} is a direct consequence of
    \cite[Theorem~3.5.2]{LeadbetterLindgrenRootzen1983} and
    Theorem~\ref{theorem_continuous_limit}~\ref{mainThm:2}.
  \end{proof}

  \begin{proof}[Proof of Proposition~\ref{prop:Dun-violated}]
    Under the assumptions of Proposition~\ref{prop:Dun-violated}, the
    conditions of Theorem~\ref{theorem-FTG2} are satisfied and hence
    $\pG = \EVDist^\theta$ for some $\EVDist\in\GEV$. This implies that
    $\pG \in \GEV$. To show that the distributional mixing condition is
    violated, pick an arbitrary $x \in \IR$ with $\pG(x) \in (0,1)$, which also
    means that $\EVDist(x) \in (0,1)$. Because $\lambda$ is not of the form
    $\lambda(t) = \alpha t$, $t \in (0,1)$, $\alpha > 0$,
    (\cite[Proposition~1]{aczel1987}) implies that there exist
    $t_1, t_2 \in (0,1)$ such that $t_1 + t_2 \in (0,1)$ while
    $\lambda(t_1 + t_2) \neq \lambda(t_1) + \lambda(t_2)$.
    For large enough $n$ so that $\lceil n t_1 \rceil + \lceil nt_2 \rceil < n$,
    set $A = \{ 1,\dots, \lceil n t_1 \rceil\}$ and
    $B = \{n- \lceil n t_2 \rceil + 1, \dots, n\}$. The index gap between these
    two sets is $n - \lceil n t_2 \rceil - \lceil nt_1 \rceil= O(n)$, and is
    also eventually larger than $s_n$ for any sequence $(s_n) = o(n)$.  Next,
    recall from the proof of Theorem~\ref{theorem-FTG2} that
    \eqref{eq:neededforProp3.1} holds, which,
    together with Lemma~\ref{lemma_uniform_convergence} and
    Lemma~\ref{Lemma_function_composition}, implies that
    $\Prob(\max_{i \in A} X_i \leq u_n) = \delta_{\lceil n t_1\rceil}\{ F(\cn_n
    x + \dn_n)\} = \delta_{\lceil n t_1\rceil}([\{ F^{r_n}(\cn_n x +
    \dn_n)\}^{r_{\lceil nt_1 \rceil}/r_n}]^{1/r_{\lceil nt_1 \rceil}})$ tends to
    $H^{\theta \lambda(t_1)}(x)$ as $n \to \infty$. Similarly,
    $\Prob(\max_{i \in B} X_i \leq u_n) \to H^{\theta \lambda(t_2)}(x)$ as
    $n \to \infty$.  Because the sequence is exchangeable, and
    $\ceil{n(t_1 + t_2)} \leq \ceil{nt_1} + \ceil{nt_2} \leq \ceil{n(t_1 + t_2)}
    + 1$ the same argument also implies that
    \begin{align*}
      \lim_{n\to \infty}  \Prob(\max_{i \in A \cup B} X_i \leq u_n) &\le \lim_{n\to \infty}  \Prob(\max_{i \in \{1,\dots, \lceil n(t_1 + t_2)\rceil\}} X_i \leq u_n) = H^{\theta \lambda(t_1 +t_2)},\quad\text{and that}\\
      \lim_{n\to \infty}  \Prob(\max_{i \in A \cup B} X_i \leq u_n) &\ge \lim_{n\to \infty}  \Prob(\max_{i \in \{1,\dots, \lceil n(t_1 + t_2)\rceil+1\}} X_i \leq u_n)\\
                                                                    &\ge \lim_{n\to \infty} \{  \Prob(\max_{i \in \{1,\dots, \lceil n(t_1 + t_2)\rceil\}} X_i \leq u_n) - \Prob(X_{\lceil n(t_1 + t_2)\rceil+1} > u_n)\}.
    \end{align*}
    For the last inequality, note that %
    \begin{align*}
    \Prob(\max_{i \in \{1,\dots, \lceil n(t_1 + t_2)\rceil+1\}} X_i \leq u_n)=\Prob(\max_{i \in \{1,\dots, \lceil n(t_1 + t_2)\rceil\}} X_i \leq u_n, X_{\lceil n(t_1 + t_2)\rceil+1}\le u_n)
    \end{align*}
    Subtracting
    $\Prob(X_{\lceil n(t_1 + t_2)\rceil+1} > u_n) - \Prob(\max_{i \in \{1,\dots, \lceil n(t_1 + t_2)\rceil\}} X_i \leq u_n,
    X_{\lceil n(t_1 + t_2)\rceil+1}> u_n)\ge 0$ from the right hand side leads to the result.
    Now $\Prob(X_{\lceil n(t_1 + t_2)\rceil+1} > u_n) = 1- F(u_n)$ and
    $\lim_{n\to \infty} F(u_n) = \{F^{r_n}(u_n)\}^{1/r_n}=1$ by
    \eqref{eq:neededforProp3.1} and the fact that $r_n \to \infty$ as
    $n \to \infty$. Because
    $\lim_{n\to \infty} \Prob(\max_{i \in \{1,\dots, \lceil n(t_1 +
      t_2)\rceil\}} X_i \leq u_n) = H^{\theta \lambda(t_1 +t_2)}$, this implies
    that
    $\lim_{n\to \infty} \Prob(\max_{i \in A \cup B} X_i \leq u_n) \ge H^{\theta
      \lambda(t_1 +t_2)}$. Put together,
    $\lim_{n\to\infty}| \Prob(\max_{i \in A\cup B} X_i \leq u_n) - \Prob(\max_{i
      \in A} X_i \leq u_n)\Prob(\max_{i \in B} X_i \leq u_n)| = |H(x)^{\theta
      \lambda(t_1 + t_2)} - H(x)^{\theta \{\lambda(t_1) + \lambda(t_2)\}}|$,
    which is non-zero since $H(x) \in (0,1)$ and $\theta \neq 0$ as well as
    $\lambda(t_1 + t_2) \neq \lambda(t_1) + \lambda(t_2)$.
  \end{proof}

  \begin{proof}[Proof of Theorem~\ref{theorem_Archimedean_distortion_limit}]
    The convergence $\pD_n^{\rate}(u) \to \pD(u)$ is immediate for $u=0$ and
    $u=1$, so consider an arbitrary $u\in(0,1)$.  By continuity of $\gen$, we
    only need to establish the convergence of
    $\STDFone_n \invgen(u^{1/\rate_n})$. Setting
    $x_n = 1/\{\rate_n(1-u^{1/\rate_n})\}$, the latter expression can be
    rewritten as $\invgen\{1-1/(x_n \rate_n)\}/\invgen(1-1/\rate_n)$. To compute
    the limit of $x_n$ for $n\to\infty$, note that $r_n \to \infty$ because
    $\STDFone_n \to \infty$ as $n\to\infty$ by assumption, and
    $1 - \gen(1/t) \to 0$ as $t \to \infty$. Because $t(1-u^{1/t}) \to - \log u$
    as $t \to \infty$, we obtain that $x_n \to 1/(-\log u)$ as $n \to
    \infty$. Utilizing \cite[Proposition~1\,(d)]{LarssonNeslehova2011} with
    $\rho \in (0,1]$, we have that $1 - \gen(1/\cdot) \in \RV_{-\rho}$ if and
    only if $\invgen(1-1/\cdot) \in \RV_{-1/\rho}$. The uniform convergence
    theorem for regularly varying functions
    \cite[Theorem~1.5.2]{BinghamGoldieTeugels1987} implies that
    $\lim_{n\to \infty} \frac{\invgen\{1-1/(x_n \rate_n)\}}{\invgen(1-1/\rate_n)}
    = \{(-\log u)^{-1}\}^{-1/\rho} = (-\log u)^{1/\rho}$. Because $\gen$
    is continuous,
    $\gen\{\STDFone_n \invgen(u^{1/\rate_n})\} \to \gen\{(-\log u)^{1/\rho}\}$
    as $n \to \infty$. The remaining part of
    Theorem~\ref{theorem_Archimedean_distortion_limit} then follows from
    Theorem~\ref{theorem_continuous_limit}~\ref{mainThm:1}.
  \end{proof}

  \begin{proof}[Proof of Corollary~\ref{corollary_density_Archimedean_distortion}]
    The fact that $\pD$ is a distribution function with density $\dD$ and
    quantile function $\qD$ follows immediately from properties of $\gen$ and the
    logarithm.
    With the change of variables $(-\log u)^{1/\rho} = y$, one obtains the limit $\lim_{u \to 0} \dD(u)$ as stated. Finally,
    \begin{align*}
\lim_{u \to 1} \dD(u) &= \lim_{y \to 0} -\gen'(y)y^{1-\rho}\e^{(y^{\rho})}/\rho =
          \begin{cases}
            -\gen^\prime(0),&\text{if}\ \rho = 1,\\
            \lim_{y \to 0} -\gen'(y)y^{1-\rho}/\rho,&\text{if}\ \rho \in (0,1),
          \end{cases}
    \end{align*}
    as claimed.
  \end{proof}

  \begin{proof}[Proof of Proposition~\ref{prop_unique_archimax_limit}]
    Given that $\rate_n = 1/(1-\gen(1/\STDFone_n))$, we have that
    $\rate_{\ceil{nt}}/\rate_n = \frac{1/\{1-\gen(1/\STDFone_{\ceil{n
          t}})\}}{1/\{1-\gen(1/\STDFone_n)\}} =
    \frac{\smash{1-\gen(1/\STDFone_n)}}{1-\gen(1/\STDFone_{\ceil{n t}})} =
    \frac{\smash{1-\gen\left(1/\STDFone_n\right)}}{1-\gen[1/\{\STDFone_n(\STDFone_{\ceil{n
          t}}/\STDFone_n)\}]}$.  Setting
    $t_n = \STDFone_{\ceil{nt}}/\STDFone_n$, we have, by assumption, that
    $t_n \to \kappa(t)$. With $1 - \gen(1/\cdot) \in \RV_{-\rho}$,
    $\rho \in (0,1]$, and $\STDFone_n \to \infty$ we obtain, by regular
    variation combined with the uniform convergence theorem for regularly
    varying functions \cite[Theorem~1.5.2]{BinghamGoldieTeugels1987}, that
    $\lim_{n\to\infty}\rate_{\ceil{nt}}/\rate_n = \lambda(t)$, where
    $\lambda(t)=(\kappa(t))^{\rho}$. Clearly,
    $\lambda \colon (0,\infty) \to (0,\infty)$ is a bijection, and
    assumption~\ref{thm:A:dist:limit:2} of
    Theorem~\ref{theorem_Archimedean_distortion_limit} implies that
    $\rate_n \to \infty$. Because $\gen$ is strict, the copula diagonals
    $\Cdiag_n(u) = \gen\{\STDFone_n\invgen(u)\}$ as well as the limiting
    distortion $\pD(u) = \gen\{(-\log u)^{1/\rho}\}$ are continuous and strictly
    increasing.
    Now apply Theorem~\ref{theorem-FTG2}.
  \end{proof}

  \begin{proof}[Proof of Lemma~\ref{Lemma_Archimedean_GEV_Limit}]
    From \cite[Theorem~3.3]{HerrmannHofertNeslehova2023b}, a distribution
    function $\pD$ on $[0,1]$ satisfies $\pG = \pD \circ \EVDist \in \GEV$ for
    all $\EVDist\in\GEV$ if and only if $\pD$ is of the form
    $\e^{-\lambda (-\log u)^\gamma}$ for some $\lambda,\gamma > 0$. Equating
    this expression to $\gen(-\log u)$ for $u \in (0,1)$ and setting
    $t = -\log(u)$, we obtain that
    $\gen(t)=\e^{-\lambda t^\gamma}= \e^{-(ct)^{1/\theta}}$ with
    $\theta = 1/\gamma$ and $c= \lambda^{\theta}$. And the function
    $\e^{-(ct)^{1/\theta}}$ is regularly varying at $0$ with index
    $1/\theta$. However, from \cite[Remark~2]{LarssonNeslehova2011}, $\gen$ is
    convex only if $\theta \ge 1$.
  \end{proof}

  \begin{proof}[Proof of Theorem~\ref{theorem_convergence_rate}]
    Under the assumptions of Theorem~\ref{theorem_convergence_rate} the
    conditions of Theorem~\ref{theorem_continuous_limit}~\ref{mainThm:1} hold
    and hence $(M_n -\bn_{\ceil{\rate_n}})/\an_{\ceil{\rate_n}}$ converges
    weakly to $D \circ H$. By the triangle inequality,
    \begin{multline*}
      \abs{\Prob(M_{n} \leq \an_{\ceil{\rate_n}} x + \bn_{\ceil{\rate_n}}) - \pD(\EVDist(x))} \leq \sup_{x\in\IR} \abs{\pD\{\pMargin^{\rate_n}(\an_{\ceil{\rate_n}} x + \bn_{\ceil{\rate_n}})\} - \pD\{\EVDist(x)\}}\\
      +  \sup_{x\in\IR}\abs{\pD^\rate_n\{\pMargin^{\rate_n}(\an_{\ceil{\rate_n}} x + \bn_{\ceil{\rate_n}})\} - \pD\{\pMargin^{\rate_n}(\an_{\ceil{\rate_n}} x + \bn_{\ceil{\rate_n}})\} }.
    \end{multline*}
    By assumption~\ref{itconvrate:2}, the second term on the right-hand side
    is bounded above by $s(n)$. Using H\"older continuity of $\pD$ in
    assumption~\ref{itconvrate:3}, the first term is bounded above by
    $K\sup_{x\in\IR}\abs{\pMargin^{\rate_n}(\an_{\ceil{\rate_n}} x +
      \bn_{\ceil{\rate_n}}) - \EVDist(x)}^{\kappa}$.  If $\rate$ is
    integer-valued, $\ceil{\rate_n} = \rate_n$ for all $n \in \IN$, and
    assumption~\ref{itconvrate:1} gives the upper bound
    $K \{\beta^*(\rate_n)\}^\kappa$. Otherwise, the supremum in the previous
    display is bounded above by
    $\sup_{x\in\IR}\abs{ \pMargin^{\ceil{\rate_n}}(\an_{\ceil{\rate_n}} x +
      \bn_{\ceil{\rate_n}}) - \EVDist(x) } +
    \sup_{x\in\IR}\abs{\pMargin^{\rate_n}(\an_{\ceil{\rate_n}} x +
      \bn_{\ceil{\rate_n}}) - \pMargin^{\ceil{\rate_n}}(\an_{\ceil{\rate_n}} x +
      \bn_{\ceil{\rate_n}})}$.  The first term is at most
    $\beta^*(\ceil{\rate_n})$ by assumption~\ref{itconvrate:1}, while the
    second is at most $3\e^{-1}/\rate_n$ by
    Corollary~\ref{corollary_uniform_difference_upper_Iverson}.
  \end{proof}

\begin{proof}[Proof of Corollary~\ref{corollary_PQD_convergence_speed}]
  In view of the discussion immediately preceding
  Corollary~\ref{corollary_PQD_convergence_speed}, the case of monotonically
  decreasing $\iidrate$ and a PQD $\Copula_n$ remains to be considered. The
  Fr\'echet--Hoeffding inequality \cite[Theorem~2.10.12]{Nelsen2006} and the
  PQD property imply that $u^n \leq u^{\eta_n} \leq u$ for all $u\in[0,1]$ and
  hence $1 \le \eta_n \le n$.  As $\iidrate$ is monotonically decreasing,
  $\iidrate\left(\ceil{\eta_n}\right) \geq \iidrate\left(n\right)$.
\end{proof}

\begin{proof}[Proof of Proposition~\ref{proposition_convergence_rate_dependence_structure}]
  Calling on assumption~\ref{itconvrate:1} of
  Theorem~\ref{theorem_convergence_rate} and the triangle inequality, we can
  bound
  $\sup_{x\in\IR} \abs{\pMargin^{\rate_n}(\an_{\ceil{\rate_n}} x +
    \bn_{\ceil{\rate_n}}) - \EVDist(x)}$ from above by
  $\sup_{x\in\IR}\abs{\pMargin^{\rate_n}(\an_{\ceil{\rate_n}} x +
    \bn_{\ceil{\rate_n}}) - \pMargin^{\ceil{\rate_n}}(\an_{\ceil{\rate_n}} x +
    \bn_{\ceil{\rate_n}})} + \sup_{x\in\IR}\abs{
    \pMargin^{\ceil{\rate_n}}(\an_{\ceil{\rate_n}} x + \bn_{\ceil{\rate_n}}) -
    \EVDist(x) }$,
  where the second term is at most $\iidrate\left(\ceil{\rate_n}\right)$.  When
  $\ceil{\rate_n} = \rate_n$, the first term disappears; otherwise, it is at
  most $3/(\rate_n\e)$ by
  Corollary~\ref{corollary_uniform_difference_upper_Iverson}.

  For $u=0$ and $u=1$, we immediately have $\pD_n^{\rate}(u) = \pD(u)$ for all
  $n\in\IN$.  For a given $u\in(0,1)$, pick $x\in\IR$ such that $u = \EVDist(x)$
  and set
  $u_n = \pMargin^{\rate_n}(\an_{\ceil{\rate_n}}x + \bn_{\ceil{\rate_n}})$;
  clearly, such an $x$ exists since $\EVDist\in\GEV$ is continuous.  H\"older
  continuity of $\pD_n^{\rate}$ now implies that
  $\sup_{u\in[0,1]}\abs{\pD_n^{\rate}(u)-\pD(u)} \leq
  \sup_{u\in[0,1]}\abs{\pD_n^{\rate}(u)-\pD_n^{\rate}(u_n)} +
  \sup_{u\in[0,1]}\abs{\pD_n^{\rate}(u_n)-\pD(u)} \leq \sup_{u\in[0,1]} K_n
  \abs{u-u_n}^{\kappa_n} + \sup_{x\in\IR}\abs{\Prob(M_{n} \le
    \an_{\ceil{\rate_n}}x+\bn_{\ceil{\rate_n}}) - \pD\circ\EVDist(x)} \leq K_n
  \bigl\{\sup_{x\in\IR}\abs{\pMargin^{\rate_n}(\an_{\ceil{\rate_n}}x$ $+
    \bn_{\ceil{\rate_n}})-\EVDist(x)}\bigr\}^{\kappa_n} + \gamma(n) \leq K
  \bigl\{\iidrate (\ceil{\rate_n}) +
  3\e^{-1}\I_{\{r_n\notin\IN\}}/\rate_n\bigr\}^{\kappa} + \gamma(n)$.  The
  assumptions on $\rate_n$, $\iidrate$ and $\gamma$ imply that the right-hand
  side converges to $0$ as $n\to \infty$.
\end{proof}

\section{Connections to and discussion of the results of \cite{Ballerini1994b}}
In contrast to \cite{Wuthrich2004} and the results in
Section~\ref{section_Archimedean_Archimax}, the results in \cite{Ballerini1994b}
are not derived based on a suitable regular variation condition of the
Archimedean generator (or its inverse).  Instead, comparable results are derived
based on the asymptotic polynomial growth condition
  \begin{align}
    \lim_{t\to\infty} t^{\rho}\{1-\gen(1/t)\} = c\label{eq:poly:growth}
  \end{align}
  for some $\rho\in(0,\infty)$ and $c\in(0,\infty)$; see
  \cite[Theorem~2]{Ballerini1994b}.  However this condition limits the scope of
  the derived results.  While this is already noted in
  \cite[p.~386]{Ballerini1994b}, the precise extent of this limitation has not
  been addressed in the literature.  In the following discussion, we show that
  the regular variation condition $1-\gen(1/\cdot) \in \RV_{-\rho}$ used in
  Section~\ref{section_Archimedean_Archimax} is strictly more general than the
  asymptotic polynomial growth condition \eqref{eq:poly:growth} of \cite{Ballerini1994b}.
  \begin{proposition}\label{prop_asymp_polynomial_growth}
    Let $f$ be a positive measurable function, defined on some neighbourhood $[S,\infty)$ of infinity, and $\rho \in \IR$ and
    $c > 0$ be two constants.  Then $f\in\RV_{-\rho}$ with associated slowly
    varying function $L$ such that
    $\lim_{x\to\infty} L(x) = c$ if and only if
    $\lim_{x\to\infty} x^{\rho} f(x) = c$.
  \end{proposition}
  \begin{proof}
    For necessity, $f\in\RV_{-\rho}$ can be represented as
    $f(x) = x^{-\rho}L(x)$ for some slowly varying function $L$; see
    \cite[Theorem~1.4.1]{BinghamGoldieTeugels1987}.  By assumption, we
    have that
    $\lim_{x\to\infty} x^{\rho} f(x) = \lim_{x\to\infty} x^{\rho} x^{-\rho}L(x)
    = \lim_{x\to\infty}L(x) = c$. For sufficiency, if
    $\lim_{x\to\infty} x^{\rho}f(x) = c$, $c > 0$, we have by assumption that for any $\lambda>0$
    that, $\lim_{x\to\infty} \frac{f(\lambda x)}{f(x)} = \lim_{x\to\infty} \lambda^{-\rho}\frac{(x\lambda)^{\rho} f(\lambda x)}{x^{\rho}f(x)} = \lambda^{-\rho}$,
    that is $f \in \RV_{-\rho}$.  This implies that $f$ is necessarily of the
    form $f(x) = x^{-\rho}L(x)$, where $L$ is slowly varying; see
    \cite[Theorem~1.4.1]{BinghamGoldieTeugels1987}. Again by assumption, we then have that
    $c = \lim_{x\to\infty} x^{\rho}f(x) = \lim_{x\to\infty} L(x)$.
  \end{proof}
  By Proposition~\ref{prop_asymp_polynomial_growth}, and as already noted in Example~\ref{ex:Ballerini}, the class $\RV_{-\rho}$ is
  strictly larger than the set of all functions with
  $\lim_{x\to\infty} x^{\rho} f(x) = c > 0$, since slowly varying functions with
  finite non-zero limits are only a subset of all slowly varying functions; a
  simple example of a slowly varying function $L$ such that
  $\lim_{x\to\infty} L(x)$ does not exist is $L(x) = \log x$; see
  \cite[p.~16]{BinghamGoldieTeugels1987} for more examples.

  In the context of completely monotone Archimedean generators, the following
  lemma of \cite[Lemma~1]{Ballerini1994b} can be used to to derive a useful
  corollary for the constructing of an example where
  $1-\gen(1/\cdot) \in \RV_{-\rho}$ but where the asymptotic polynomial growth
  condition~\eqref{eq:poly:growth} is violated.
  \begin{lemma}[\cite{Ballerini1994b}, Lemma~1]\label{lemma_Ballerini}
    Denote by $\gen \colon (0,\infty)\to\IR$ a real valued function.  Suppose
    $\gen(t) \geq 0$ on $(0,\infty)$ and $\gen'(t) = -\kappa(t)\gen(t)$ for
    $\kappa\colon(0,\infty)\to\IR$ being completely monotone on $(0,\infty)$. Then
    $\gen$ is completely monotone on $(0,\infty)$.
  \end{lemma}

  Lemma~\ref{lemma_Ballerini} implies the following method for constructing
  completely monotone Archimedean generators $\gen$ with given regular variation
  properties of $1-\gen(1/\cdot)$.
  \begin{corollary}\label{corollary_comp_mon_generator}
    Let $f \colon [0,\infty)\to\IR$ be a differentiable function.
    \begin{enumerate}[label=(\roman*), labelwidth=\widthof{(iii)}]
    \item Then it holds that $f(0)=0$, $\lim_{t\to\infty}f(t)=\infty$ and that the derivative
      $f'\colon (0,\infty)\to\IR$ is completely monotone if and only if
      $\gen(t) = \exp\{-f(t)\}$ is a completely monotone generator of an
      Archimedean copula such that $\gen^{\alpha}$ is completely monotone for
      all $\alpha\in(0,\infty)$.
    \item Consider a completely monotone Archimedean generator of the form
      $\gen(t) = \exp\{-f(t)\}$. Then the mean of the associated frailty
      distribution is given by $\lim_{t\to 0}f'(t)$. If, additionally,
      $f'(1/\cdot)\in\RV_{\beta}$, $\beta\in\IR$, then
      $1-\gen(1/\cdot)\in\RV_{-(1-\beta)}$.
    \end{enumerate}
    \vspace{-\baselineskip}%
  \end{corollary}
  \begin{proof}\mbox{}
    \begin{enumerate}[label=(\roman*), labelwidth=\widthof{(iii)}]
    \item Consider sufficiency. If $\gen(t) = \exp\{-f(t)\}$, we have
      $f(t) = -\log \gen(t)$, which directly yields $f(0)=0$ and
      $\lim_{t\to\infty}f(t)=\infty$. For the derivative we have
      $f'(t) = -\gen'(t)/\gen(t)=\{-\log\gen(t)\}'$. By
      \cite[Proposition~2.1.5\,(5)]{Hofert2010}, $f'$ is completely monotone
      if and only if $\gen^{\alpha}$ is completely monotone
      for all $\alpha\in(0,\infty)$, which holds by assumption.

      Now consider necessity. By construction, we have
      $\gen(0) = 1$ and $\lim_{t\to\infty}\gen(t) = 0$.  Furthermore, we have
      $\gen'(t) = -\exp\{-f(t)\}f'(t) = -\gen(t)f'(t)$. Since, by assumption, $f'$ is
      completely monotone, Lemma~\ref{lemma_Ballerini} guarantees that $\gen$ is
      completely monotone.  By \cite{Kimberling1974}, $\gen$ is therefore the
      generator of an Archimedean copula.
    \item The mean of the associated frailty distribution is given by
      $\lim_{t\to 0}-\gen'(t) = \lim_{t\to 0}\gen(t)$
      $\cdot f'(t) = \lim_{t\to 0}f'(t)$ since $\gen(0)=1$. For the regular
      variation property, we first compute
      $\frac{\d}{\d t}\gen\{1/(\lambda t)\} = -\gen'\{1/(\lambda
      t)\}t^{-2}\lambda^{-1}$.  Using L'Hôpital's rule and $\gen(0)=1$, we have
      $\lim_{t\to\infty}\frac{1-\gen\{1/(\lambda t)\}}{1-\gen(1/t)} =
      \frac{1}{\lambda}\lim_{t\to\infty}\frac{\gen'\{1/(\lambda t)\}}{\gen'(1/t)}
      = \frac{1}{\lambda}\lim_{t\to\infty}\frac{\gen\{1/(\lambda t)\}f'\{1/(\lambda
        t)\}}{\gen(1/t)f'(1/t)} =$ $\frac{1}{\lambda}\lim_{t\to\infty}\frac{f'\{1/(\lambda t)\}}{f'(1/t)}$ for $\lambda>0$.  By
      assumption, $f'(1/\cdot)\in\RV_{\beta}$ and so
      $\lim_{t\to\infty}\frac{1-\gen\{1/(\lambda t)\}}{1-\gen(1/t)}$ $=
      \lambda^{-1}\lambda^{\beta} = \lambda^{-1+\beta}$, that is
      $1-\gen(1/\cdot)\in\RV_{-(1-\beta)}$.\qedhere
    \end{enumerate}
  \end{proof}

  We now use Corollary~\ref{corollary_comp_mon_generator} in the following
  example to construct a regularly varying Archimedean generator which is not
  covered by the framework of \cite{Ballerini1994b}.
  We therefore close a gap in the literature since the proposed construction of
  \cite[Example~3]{Ballerini1994b} based on $\gen_{\beta}(t) = \exp\bigl[-\{\log(1+t)\}^{1/\beta}\bigr],\quad \beta \geq 1$,
  actually yields (via L'Hôspital's rule) that
  $\lim_{t\to\infty} t^{1/\beta}\{1-\gen(1/t)\} = 1$ and not that
  $\lim_{t\to\infty} t^{1/\beta}\{1-\gen(1/t)\} = 0$ as claimed.
  \begin{example}\label{example_counter_Ballerini}
    Consider the function $f(x)= (1 + x) \log(1 + 1/x) + \log(x)$,
    $x\in(0,\infty)$, which satisfies $f(x)\geq 0$, $\lim_{x\to 0}f(x) = 0$,
    $\lim_{x\to\infty}f(x)=\infty$ and $f'(x) = \log(1+1/x)$. Also note that
    $f'$ is completely monotone. %
    By Corollary~\ref{corollary_comp_mon_generator}, the function
    $\gen(t) = \exp\{-f(t)\} = 1/\{t(1+1/t)^{1+t}\}$ is therefore a completely
    monotone Archimedean generator. Furthermore, we have
    $f'(1/\cdot) \in \RV_{0}$, that is $f'(1/\cdot)$ is slowly varying; this can
    be seen from L'Hôpital's rule via
    $\lim_{x\to\infty} \frac{f'\{1/(\lambda x)\}}{f'(1/x)} = \lim_{x\to\infty}
    \frac{\log(1+\lambda x)}{\log(1+x)} = \lambda \lim_{x\to\infty}
    \frac{1+x}{1+\lambda x} = 1$.  By
    Corollary~\ref{corollary_comp_mon_generator},
    $1-\gen(1/\cdot)\in\RV_{-1}$. While it is easily seen that
    $-\gen'(0)<\infty$ implies $\rho=1$, it is interesting to note that here we
    have $-\gen'(0) = f'(0) = \infty$ even though $\rho=1$, demonstrating that
    the converse implication is not true in general; see also
    Table~\ref{table_generator_limit} and the example provided by Family~23 of \cite[Table~1]{Charpentier/Segers:2009}.
    More importantly, considering the limit
    condition in \cite{Ballerini1994b}, we have that
    $\lim_{t\to\infty} t^{\beta}\{1-\gen(1/t)\} \in \{0, \infty\}$ for
    $\beta\neq 1$, and, for $\beta=1$ via L'Hôpital's rule, that
    $\lim_{t\to\infty} t\{1-\gen(1/t)\} = \lim_{t\to\infty}
    \frac{1}{t^{-1}}\{1 - \frac{t}{(1+t)^{1+1/t}}\} =
    \lim_{t\to\infty} \frac{-(1 + t)^{-1 - 1/t} \log(1 + t)/t}{-t^{-2}} =
    \lim_{t\to\infty} \frac{t\log(1 + t)}{(1 + t)^{1 + 1/t} } = \infty$.  Hence
    the generator $\gen(t)=1/\{t(1+1/t)^{1+t}\}$ does not satisfy the asymptotic
    polynomial growth condition~\eqref{eq:poly:growth} of the framework of
    \cite{Ballerini1994b}.
  \end{example}

  As a final step we can compare the limiting distributions from
  \cite{Ballerini1994b} with ours from
  Theorem~\ref{theorem_Archimedean_distortion_limit}. In contrast to
  Theorem~\ref{theorem_Archimedean_distortion_limit}, the limiting distributions in
  \cite[Equation~(2.2)]{Ballerini1994b} reference a positive parameter $c>0$
  linked to the asymptotic polynomial growth condition~\eqref{eq:poly:growth}, where $\rho$ is the
  coefficient of regular variation of $1-\gen(1/\cdot)$. The following remark
  shows that in our more general framework, this dependence is in fact an
  artifact that can be interpreted as a rescaling of the copula generator $\gen$
  without an effect on the final dependence structure. We illustrate this point
  with a detailed example in the remaining part of this supplement.
  \begin{remark}\label{remark_invariance_generator_scaling}
    An Archimedean generator $\gen_1$ is not uniquely identified in the sense
    that the generator $\gen_2(t)=\gen_1(ct)$ generates the same copula for all
    $c>0$. Our results in Theorem~\ref{theorem_Archimedean_distortion_limit}
    about the limiting distribution of $M_n$ in the Archimax case should
    therefore not depend on the rescaling factor $c$. However, following the
    same steps as in Theorem~\ref{theorem_Archimedean_distortion_limit}, it
    seems counter-intuitive at first that the limiting distribution function
    $\pD$ depends on $c$. Specifically, for $1-\gen_1(1/\cdot) \in \RV_{-\rho}$,
    consider $\gen_2(t)=\gen_1(\tilde{c}t)$ with $\tilde{c}=c^{-1/\rho}$, where
    the power $-1/\rho$ of $c$ allows us to connect our discussion to the
    results derived in \cite{Ballerini1994b}.  Verifying that then also
    $1-\gen_2(1/\cdot)\in\RV_{-\rho}$, it follows that
    $\lim_{n\to\infty}\pD^{\rate}_n(u) = \lim_{n\to\infty}\gen_2\{\STDFone_n\invgen_2(u^{1/\rate_n})\}
    =\gen_1\{\tilde{c}(-\log u)^{1/\rho}\} = \gen_1\{c^{-1/\rho}(-\log
    u)^{1/\rho}\}$, where the limit on the right-hand side is
    equal to the distortion implicitly derived in the Archimedean case in
    \cite[Equation~(2.2)]{Ballerini1994b}. To reconcile the fact that $\gen_1$
    leads to the limit $\pD_1(u)=\gen_1\{(-\log u)^{1/\rho}\}$, while $\gen_2$
    leads to the limit
    $\pD_2(u)=\gen_1\{c^{-1/\rho}(-\log u)^{1/\rho}\}=\gen_1\{(-\log
    u^{1/c})^{1/\rho}\}$, it is important to remember that in the Archimedean
    case, the stabilizing constants $(\cn_n)$ and $(\dn_n)$ depend on the
    underlying generator. This leads to different pairs of stabilizing sequences
    for $M_n$ for $\gen_1$ and $\gen_2$, and hence implies by the convergence to
    types theorem, see \cite[Proposition~0.2]{Resnick1987}, that the limiting
    distributions are related by an affine transformation.  That this is indeed
    the case can be seen from the fact that the two limits differ by a power
    transformation of the argument, that is $\pD_1(u)=\pD_2(u^c)$. In case of a
    max-stable distribution as argument, this indeed leads to an affine
    shift. As discussed in Remark~\ref{rem_rate_uniqueness}, for a max-stable
    distribution function $\EVDist$ the constants $a>0$ and $b\in\IR$ such that
    $\EVDist^{c}(x) = \EVDist(ax+b)$ can be calculated explicitly, which hence
    precisely determines the impact of using stabilizing sequences connected to
    either $\gen_1$ or $\gen_2$. Via the convergence to types theorem, also
    their relation to each other can be determined.
  \end{remark}

  The following example illustrates this point.
  \begin{example}
    Denote by $\pMargin_{\lambda}(x) = 1-\e^{-\lambda x}$, $x\ge 0$, the
    distribution function of an exponential distribution with parameter
    $\lambda>0$. Then possible iid stabilizing constants are $\an_n = 1/\lambda$
    and $\bn_n = \log(n)/\lambda$, leading to
    $\lim_{n\to\infty} \pMargin_{\lambda}^n\left(\an_n x + \bn_n\right) =
    \e^{-\e^{-x}} = \Lambda(x)$, that is $\pMargin_{\lambda} \in \MDA(\Lambda)$.
    Now consider a sequence $(X_i)$ where the margins are exponentially
    distributed and the copula $\Copula_{\theta}$ of $(X_1,\ldots,X_n)$ is a
    Clayton copula with parameter $\theta > 0$.  In this case the generator
    takes the form $\gen(t) = \left(1+t\right)^{-1/\theta}$, and, for $c>0$, the
    scale-transformed version is given by
    $\gen_c(t) = \left(1+ct\right)^{-1/\theta}$.  In this case we have
    $\rho = 1$ for $\gen$ and $\gen_c$.  At the same time, we have
    $-\gen_c'(0) = -c\gen'(0) = c/\theta < \infty$, and hence, as discussed in
    Example~\ref{Example_Archimax_no_tail_dep}, we can use the iid stabilizing
    constants $(\an_n)$ and $(\bn_n)$ to obtain
    $\lim_{n\to\infty} \Prob\bigl(\frac{M_n - \bn_{n}}{\an_{n}} \leq x\bigr) =
    \gen_c\{-\log \Lambda^{-1/\gen_c'(0)}(x)\} = \gen\{-c \log
    \Lambda^{\theta/c}(x)\} = \gen\{-\log \Lambda^{\theta}(x)\}$, where, as
    expected, the result is invariant with respect to the value of $c>0$ since
    $\gen$ and $\gen_c$ induce the same copula $\Copula_{\theta}$.

    However, when using the stabilization implied by
    Theorem~\ref{theorem_Archimedean_distortion_limit}, the result seemingly
    depends on the value of $c>0$. In this case, we have for
    $\rate_n = 1-\gen_c(1/n)$ the asymptotic expansion
    $\rate_n = \frac{1}{1-\gen_c(1/n)} = \frac{\theta n}{c} + \frac{\theta+1}{2}
    - \frac{c(\theta^2-1)}{12\theta n} + \mathcal{O}\left(n^{-2}\right)$, which,
    for large $n$, suggests the approximation $\rate_n \approx \theta n/c$.
    With $\ceil{\rate_n} = \ceil{n\theta/c} = c_n(n\theta/c)$, where
    $c_n = \ceil{n\theta/c} c / (n\theta) \to 1$, this leads to the (asymptotic)
    stabilizing constants $\an_{\ceil{\rate_n}}= 1/\lambda$ and
    $\bn_{\ceil{\rate_n}} = \log(\ceil{\rate_n})/\lambda = \log(n)/\lambda +
    \log(\theta)/\lambda - \log(c)/\lambda + \log(c_n)/\lambda$.  By
    Theorem~\ref{theorem_Archimedean_distortion_limit}, we then have for
    $x \in \IR$ the weak limit
    $\lim_{n\to\infty} \Prob\bigl(\frac{M_n -
        \bn_{\ceil{\rate_n}}}{\an_{\ceil{\rate_n}}} \leq x\bigr) = \gen_c
    \{-\log \Lambda(x)\} = \gen\{-\log \Lambda^c(x)\}$.  The
    key to understanding the dependence on $c$ is to realize that in this case
    the stabilizing constants $\an_{\ceil{\rate_n}}$ and $\bn_{\ceil{\rate_n}}$
    also depend on $c>0$.  As such, the appearance of $c>0$ in the limit is a
    consequence of the convergence to types theorem (see
    \cite[Proposition~0.2]{Resnick1987}): First noting that
    $\Lambda^c(x) = \Lambda(x - \log c)$, we also have
    $\an_{\ceil{\rate_n}} / \an_{n} = 1$ and
    $(\bn_{\ceil{\rate_n}} - \bn_{n})/ \an_{n} \to \log(\theta) - \log(c)$.  The
    two different limits are now connected via
    $\gen\bigl[-\log \Lambda^{\theta}\{x + \log(\theta) - \log(c)\}\bigr] =
    \gen\bigl[-\log \Lambda\{x - \log(c)\}\bigr] = \gen\{-\log
    \Lambda^{c}(x)\}$, and the appearance of $c>0$ in the second limit can
    indeed be attributed to the utilized stabilizing constants.
  \end{example}

\printbibliography[heading=bibintoc]
\end{document}

%
%
%
%
